\documentclass[12pt]{article}
\usepackage[utf8]{inputenc}
\usepackage[a4paper, margin=2.5cm]{geometry}
\usepackage{amsmath,amsthm,amssymb,enumerate,enumitem,graphicx,makeidx,authblk,mathtools,tikz}
\usepackage{indentfirst}
\usepackage{todonotes}
\usepackage{multicol}
\usepackage{xcolor}
\usepackage[
pdftex, 
colorlinks=true, 
citecolor=blue,
hyperindex, 
plainpages=false, 
pagebackref=true, 
bookmarksopen, 
bookmarksnumbered
]{hyperref}
\let\ds\displaystyle
\usepackage{caption}
\usepackage{subcaption}
\usepackage{placeins}
\usepackage{algorithm}
\usepackage{algpseudocode}
\usetikzlibrary{arrows.meta,decorations.markings}
\usetikzlibrary{shapes.geometric}
\newtheorem{theorem}{Theorem}

\newtheorem{corollary}[theorem]{Corollary}
\newtheorem{proposition}[theorem]{Proposition}
\newtheorem*{claim*}{Claim}

\numberwithin{equation}{section}
\numberwithin{theorem}{section}

\newcommand{\R}{\mathbb{R}}
\newcommand{\C}{\mathbb{C}}
\newcommand{\N}{\mathbb{N}}

\newcommand{\dist}{\mathrm{dist}}

\renewcommand{\Re}{\mathrm{Re }}
\renewcommand{\Im}{\mathrm{Im }}

\makeindex

\title{Connecting orbits for delay differential equations\\ with unimodal feedback}

\author[1]{\textbf{G\'abor Benedek}}
\affil[1]{Bolyai Institute, University of Szeged, Aradi v\'ertan\'uk tere 1, Szeged, H-6720, Hungary}
\author[2]{\textbf{Tibor Krisztin}}
\affil[2]{HUN-REN-SZTE Analysis and Applications Research Group, Bolyai Institute, University of Szeged, Aradi v\'ertan\'uk tere 1, Szeged, H-6720, Hungary}

\date{\today}

\begin{document}
	
	\maketitle	
	
	\begin{abstract}
		This paper considers a class of delay differential equations 
		with unimodal feedback and describes the structure of certain unstable 
		sets of stationary points and periodic orbits. 
		These unstable sets consist of heteroclinic connections 		
		from stationary points and periodic orbits to 
stable stationary points, stable periodic orbits and some more complicated compact invariant sets. 
		
		A prototype example is the Mackey--Glass type equation
		$y'(t)=-ay(t)+b \frac{y^2(t-1)}{1+y^n(t-1)}$
		having three stationary solutions $0$, $\xi_{1,n}$ and $\xi_{2,n}$ 
		with $0<\xi_{1,n}<\xi_{2,n}$, provided $b>a>0$, and $n$ is large. 
The 1-dimensional leading unstable set 
		$W^u(\hat{\xi}_{1,n})$ of the stationary point $\hat{\xi}_{1,n}$ 
	is decomposed into three disjoint orbits, 	
	$W^u(\hat{\xi}_{1,n})=W^{u,-}(\hat{\xi}_{1,n})\cup \{\hat{\xi}_{1,n}\} \cup W^{u,+}(\hat{\xi}_{1,n})$. 
Here $\hat{\xi}_{1,n}$ is a constant function in the phase space with value $\xi_{1,n}$.  
$W^{u,-}(\hat{\xi}_{1,n})$ is a connecting orbit from $\hat{\xi}_{1,n}$ 
to $\hat{0}$. 
There exists a threshold value $b^*=b^*(a)>a$ such that,  
		in case $b\in (a,b^*)$, 
		$W^{u,+}(\hat{\xi}_{1,n})$ 
		connects $\hat{\xi}_{1,n}$ to $\hat{0}$;  
		and in case $b>b^*$, $W^{u,+}(\hat{\xi}_{1,n})$ connects $\hat{\xi}_{1,n}$ to 
		a compact invariant set $\mathcal{A}_n$ not containing 
		$\hat{0}$ and $\hat{\xi}_{1,n}$. 
		Under additional conditions, there is a stable periodic orbit 
		$\mathcal{O}^n$ with $\mathcal{A}_n= \mathcal{O}^n$. 
		Analogous results are obtained for the 2-dimensional leading 
		unstable sets $W^u(\mathcal{Q}^n)$ of periodic orbits 
		$\mathcal{Q}^n$ close to $\hat{\xi}_{1,n}$, 
		establishing connections from $\mathcal{Q}^n$ to $\mathcal{O}^n$. 	
	\end{abstract}
	
	
	\setcounter{footnote}{0}
	
	\section{Introduction}\label{sec1}
	
	In this paper we study delay differential equations of the form
	\begin{equation}\label{eqn:Ef}
		\tag{$E_{f}$}
		y'(t) = -ay(t) + b f\left(y(t-1)\right),
	\end{equation}
	with positive parameters $a,b$ and a smooth function $f:[0,\infty)\to[0,1]$ satisfying $f(0)=0$.
	These equations appear in several applications; see, e.g., \cite{Glass,Kyrychko,Rihan}. 
	
	For monotone nonlinearities there are many technical tools to obtain detailed information on the dynamics, and the behavior 
	is relatively simple; see, e.g., \cite{K1,K5,KWW,MPS}.   
	Nonmonotone nonlinearities, such as unimodal $f$, can generate complicated dynamics. 
	The functions $f(\xi)=\xi^k e^{-\xi}$ and 
	$f(\xi)=\frac{\xi^k}{1+\xi^n}$ with $n>k>0$ are examples of  
	unimodal nonlinearities; they were introduced in 1977 by Lasota \cite{Lasota} and by Mackey and Glass \cite{MC} to model the feedback control of blood cells. 
	The paper \cite{MC} motivated intensive research on the theory of 
	dynamical systems, since numerical results demonstrated that the above nonlinearities can create complex dynamical behavior; see \cite{Walther}. 
	
	There are several theoretical results on the dynamics generated by \eqref{eqn:Ef} showing stability of equilibria,
	existence of oscillations and periodic solutions, local and global Hopf bifurcations,
	and establishing certain connecting orbits; see, e.g., 
	\cite{BBI,Kuang,Smith,Ruan,Wei}. Further results include estimating the size of the global attractor 
	(\cite{FCP,LizRost1,LizRost2,RostW}) and proving bistable behavior (\cite{HYZ,HWW,Lin,LizRH}).
	However, so far the rigorous mathematical results explain only relatively simple dynamical behavior of \eqref{eqn:Ef}. On the other hand, numerical works show high complexity of the dynamics for some parameter values; see 
	\cite{GS,DH,Morozov}.
	
	The papers \cite{BKV,BKSZ} introduced a new approach to the study of \eqref{eqn:Ef} with nonlinearities such as $f(\xi)=\frac{\xi^k}{1+\xi^n}$. 
	The idea of \cite{BKV,BKSZ} is as follows.
	Letting $n\to\infty$,  
	$f(\xi)=\frac{\xi^k}{1+\xi^n}$ converges pointwise to the discontinuous $g:[0,\infty)\to[0,1]$ defined by $g(\xi)=\xi^k$ for $0\le\xi<1$ and $g(\xi)=0$ for $\xi>1$. 
	The limiting equation 
	\begin{equation}\label{eqn:Eg}
		\tag{$E_{g}$}
		x'(t) = -cx(t) + d g\left(x(t-1)\right),
	\end{equation}
	with $d>c>0$, is a particular 
	nonlinear equation with a discontinuous unimodal feedback function. 
	Its dynamics is still nontrivial. However, the fact that $g(\xi)=0$ for $\xi>1$ 
	allows one to prove several interesting dynamical properties rigorously. 
	In \cite{BKV,BKSZ} the existence of stable periodic orbits was established for \eqref{eqn:Eg} and, under suitable closeness assumptions, for \eqref{eqn:Ef} as well.   
	In particular, for \eqref{eqn:Ef} with the Mackey--Glass‐type nonlinearity $f(\xi)=\frac{\xi^k}{1+\xi^n}$, the required closeness of $f$ to $g$ is satisfied for all sufficiently large $n$. 
	An interesting feature observed in \cite{BKV,BKSZ} is that the projections 
	$\mathbb{R}\ni t\mapsto \left(p(t),p(t-1)\right)\in\mathbb{R}^2$ of the obtained periodic solutions
	can produce complicated-looking figures, similar to those long believed to be signatures of chaos in numerical studies, although
	the corresponding periodic orbits are stable.
	
	This paper continues the investigation of \eqref{eqn:Ef} under the assumption that \eqref{eqn:Ef} is close to a discontinuous equation of the form 
	\eqref{eqn:Eg}. Our motivation comes from works studying the delay differential equation 
	\begin{equation}\label{eqn:Morozov}
		z'(t)=- \alpha z(t)  +  \beta e^{-\delta\tau}\frac{z^2(t-\tau)}{1 + \gamma z^n(t-\tau)}
	\end{equation}
	with positive parameters $\alpha,\beta,\gamma,\delta,\tau$ and $n$.  
	This equation models a stage-structured population incorporating both a strong Allee effect and a maturation delay. 
	It is well known from numerical results that \eqref{eqn:Morozov} can exhibit rich dynamics: multiperiodic orbits, long-term transients, and complicated attractors. The authors of \cite{Morozov} concluded that small variations in the maturation delay or initial population size can qualitatively alter the eventual fate of the population.
	By the transformation $y(t) = \gamma^{1/n}z(\tau t)$ and 
	$a=\alpha\tau$, $b=\beta\gamma^{-1/n}e^{-\delta\tau}\tau$, 
	equation \eqref{eqn:Morozov} becomes our prototype equation 
	\begin{equation}\label{eqn:proto}
		y'(t) = -ay(t) + b \frac{y^2(t-1)}{1 + y^n(t-1)},
	\end{equation}
	with only three parameters $a,b,n$. 
	We prove rigorously that, for certain parameter values,  
	\eqref{eqn:proto} exhibits most of the dynamical features observed numerically. 
	
	To formulate the closeness of an equation of the form 
	\eqref{eqn:Ef} to one of the form \eqref{eqn:Eg}, 
	it is convenient to consider not a single equation \eqref{eqn:Ef} but rather a sequence of equations
	\begin{equation}\label{eqn:Efn}
		\tag{$E_{f_n}$}
		y'(t) = -a_n y(t) + b_n f_n\left(y(t-1)\right),
	\end{equation}
	where, for each $n\in\mathbb{N}$, $a_n>0$, $b_n>0$, and 
	$f_n: [0,\infty)\to [0,1)$ is a smooth function satisfying $f_n(0) = 0$. 
	Under additional hypotheses, the required closeness will be satisfied for all sufficiently large $n$. 
	
	For a subset $J\subset\mathbb{R}$ and a function $u:J\to\mathbb{R}$ set $\|u\|_J=\sup_{\xi\in J}|u(\xi)|.$ 
	For a continuous function $u:I\to\mathbb{R}$ defined on 
	an interval $I\subset\mathbb{R}$, and for $t\in I$ with $t-1\in I$, the 
	segment $u_t\in C([-1,0],\mathbb{R})$ is given by $u_t(s)=u(t+s)$, $s\in[-1,0]$.
	For $\xi\in [0,\infty)$, set $\hat{\xi}\in C([-1,0],\mathbb{R})$ by $\hat{\xi}(s)=\xi$, $s\in[-1,0]$. 
	We write $C=C([-1,0],\mathbb{R})$ and use the partial orderings $\phi\le \psi$ and $\phi\ll \psi$ on $C$ defined by $\phi(s)\le\psi(s)$ for all $s\in[-1,0]$ and 
	$\phi(s)<\psi(s)$ for all $s\in[-1,0]$, respectively.

Introduce the following assumption on $g$ and the parameters $c,d$ in  equation \eqref{eqn:Eg}:
\begin{enumerate}
	\item[$(C_g)$] The function $g:[0,\infty)\to[0,1]$ and the parameters $d>c>0$ satisfy:
	\begin{itemize}
		\item[(a)] The restriction $g|_{[0,1]}$ is $C^1$-smooth, $g(\xi) = 0$ for all $\xi > 1$, 
		$g(0) = g'(0)=0$, $g(1) = 1$, and $g'(\xi)> \frac{g(\xi)}{\xi}$ holds for all $\xi \in(0, 1]$,  where  $g'(1)$ denotes the left hand derivative of $g$ at 1.
		\item[(b)] Equation \eqref{eqn:Eg} admits a periodic solution $p: \mathbb{R} \to (0, \infty),$ with minimal period $\omega_p$, 
		such that 
		 $p(0) = 1$, and $p(t) > 1$ for all $t \in [-1, 0)$, and 
		 $(p(t), p(t-1)) \neq \left(1, g^{-1}(\tfrac{c}{d})\right)$ for all $t \in (0,\omega_p]$.	
		\item[(c)] 	There exists a $j \in \mathbb{N}$ such that, for the unique solution $\Theta_j$ of the equation $\Theta_j  =  - c \tan \Theta_j$ in the interval 
	$\Bigl(2j\pi - \tfrac{\pi}{2},  2j\pi\Bigr)$, the equality
	$c=d  	 g'(\xi_1)  \cos \Theta_j $
	is valid.	
	\end{itemize}
\end{enumerate}

Note that $(C_g)(a)$ implies that  $\frac{g(\xi)}{\xi}$ is strictly increasing in 
$(0,1]$, and $\frac{g(\xi)}{\xi}\to g'(0)=0$ as $\xi\to 0+$. 
Then, for any $d>c>0$, the function $[0,1]\ni\xi\mapsto -c\xi+dg(\xi)\in\R$ 
has exactly two zeros, $0$ and $\xi_1\in (0,1)$. 
$g^{-1}$ in condition $(C_g)(b)$ denotes the inverse of  $g|_{[0,1]}$.  
 In the papers \cite{BKV,BKSZ} assumption 
$(C_g)(b)$ was one of the key properties of the limiting equation \eqref{eqn:Eg} 
 to show the existence of stable periodic orbits for nearby equations.   
Under condition $(C_g)(c)$ a local Hopf bifurcation theorem can be applied.   


The following assumption on the sequence of equations \eqref{eqn:Efn} 
will guarantee that  \eqref{eqn:Efn} is close to \eqref{eqn:Eg} for all 
sufficiently large $n$. 
\begin{enumerate}
\item[$(C_{f_n})$] The sequences $(a_n)_{n=1}^\infty$ and $(b_n)_{n=1}^\infty$ of positive reals, and the sequence of functions $(f_n)_{n=1}^\infty$ satisfy:
\begin{itemize}
	\item[(a)] $a_n  \to  c$ and  $b_n  \to  d$  as $n  \to \infty$.
	\item[(b)]
Each $f_n: [0,\infty)\to [0,1)$ is $C^1$-smooth with $f_n(0) = 0$ and $f_n'(\xi)>0$ for $\xi\in (0,1/2]$. 
For any $\kappa \in (0,1)$,
	$$
	\bigl\| f_n - g\bigr\|_{[0, 1-\kappa]  \cup [1+\kappa, \infty)} 
	 + 
	\bigl\| f_n' - g'\bigr\|_{[0, 1-\kappa]  \cup [1+\kappa, \infty)}  
	 \to  0
	\quad \text{as } n\to\infty,
	$$
	and for any $\kappa \in (0,1)$ and $m\in\N$,
	$$
	\bigl\| f_n'\bigr\|_{[1+\kappa,\infty)}
	 \cdot
	\Bigl( \bigl\|f_n'\bigr\|_{[0,\infty)}
	\Bigr)^m
	 \to  0
	\quad \text{as } n\to\infty.
	$$
	\item[(c)] The restriction $f_n|_{(0,\infty)}$ is $C^2$-smooth. 
\end{itemize}
\end{enumerate}

Under conditions $(C_g)(a)$ and $(C_{f_n})(a)(b)$ 
solutions of equations \eqref{eqn:Eg} and \eqref{eqn:Efn} are 
defined as solutions of corresponding integral equations, 
see Section \ref{sec2}. 
As we are interested in nonnegative solutions, the natural phase space 
to study equations  \eqref{eqn:Efn} and \eqref{eqn:Eg} is 
$$
C^+=\{\psi\in C([-1,0],\mathbb{R}): \psi(s)\ge 0,\ -1\le s\le 0\}.
$$ 
For  $\psi\in C^+$, there exist unique solutions 
$x^\psi:[-1,\infty)\to \R$ and $y^\psi:[-1,\infty)\to \R$ 
of equations \eqref{eqn:Eg} and \eqref{eqn:Efn} with 
$x_0^\psi=\psi$ and $y_0^\psi=\psi$, respectively.
The solutions of equations \eqref{eqn:Eg} and \eqref{eqn:Efn} define 
the  semiflows 
$$
\Gamma:[0,\infty)\times C^+\ni (t,\phi)\mapsto x_t^\phi\in C^+ \text{ and }
\Phi^n:[0,\infty)\times C^+ \ni (t,\psi)\mapsto y_t^\psi\in C^+.
$$
We remark that $\Phi^n$ is continuous, however $\Gamma$ is not continuous. 

For $\kappa_2>\kappa_1\ge 0$ and $L>0 $, define $C_{\kappa_1,\kappa_2}=
\{\psi\in C^+: \kappa_1\le \psi(s)\le\kappa_2,\ -1\le s\le 0\}$
 and 
$C_{\kappa_1,\kappa_2}^L=
\{\psi\in C_{\kappa_1,\kappa_2}: \psi\in C_{\kappa_1,\kappa_2},\ 
|\psi(s_1)-\psi(s_2)|\le L|s_1-s_2|,\ -1\le s\le 0\}$. 
In Section \ref{sec2} we show that $C_{0,d/c}$ and $C_{0,2d/c}^{2d}$ 
are positively invariant under $\Gamma$ and $\Phi^n$, respectively. 

The zeros of $[0,\infty)\ni\xi\mapsto -c\xi+bg(\xi)\in\R$ and 
$[0,\infty)\ni\xi\mapsto -a_n\xi+b_nf_n(\xi)\in\R$
determine the stationary points of the solution semiflows
$\Gamma$ and $\Phi^n$, respectively. 
Hypotheses $(C_g)(a)$ implies that ${\hat 0}$ and $\hat\xi_1$
are stationary points of $\Gamma$. 
Under conditions $(C_g)(a)$ and $(C_{f_n})(a)(b)$, for any  $\kappa\in (\xi_1,1)$, and for all sufficiently large $n$,  the only zeros of the maps $[0,\kappa]\ni\xi\mapsto -a_n\xi +b_nf_n(\xi)\in \R$ are  $0$ and $\xi_{1,n}$, moreover,  
$ \xi_{1,n}\to \xi_1$ as $n\to\infty$, see Proposition \ref{prop:zeros-fn}. 
Consequently, for all sufficiently large $n$, $\hat{0}$ and   $\hat\xi_{1,n}$  are the only stationary points of $\Phi^n$ in  
the set $C_{0,\kappa}$. 
Notice that, for large $n$, $\Phi^n$ has at least 
one more stationary point $\hat\xi_{2,n}$ with $\xi_{2,n}>\kappa$, 
and $\xi_{2,n}\to 1$ as $n\to\infty$. 

For a solution $y:\R\to\R$ of equation \eqref{eqn:Efn}  and for subsets $S_-,S_+$ of $C^+$ we say that the orbit $\{y_t:t\in\R\}$ is a connecting orbit from $S_-$ to $S_+$ if $\mathrm{dist}( y_t,S_-)\to 0$ as 
$t\to-\infty$, and $\mathrm{dist}( y_t,S_+)\to 0$ as 
$t\to\infty$. Or, equivalently,  we  say that $\{y_t:t\in\R\}$ connects $S_-$ to $S_+$. 
Connecting orbits for  \eqref{eqn:Eg}) are defined analogously. 

Assuming $(C_g)(a)$ and $(C_{f_n})(a)(b)$, $\hat\xi_1$ is an unstable stationary point of $\Gamma$, and, for large $n$, $\hat\xi_{1,n}$ is also an unstable stationary point of $\Phi^n$. 
In Section \ref{sec4}, by considering the local 1-dimensional leading unstable manifolds 
$W^u_{\text{loc}}(\hat\xi_1)$ and  $W^u_{\text{loc}}(\hat\xi_{1,n})$ of $\hat\xi_1$  and  $\hat\xi_{1,n}$ for large $n$, respectively, 
we find solutions 
$x^-:\R\to\R$, $x^+:\R\to\R$ of equation \eqref{eqn:Eg} and 
solutions 
$y^{-,n}:\R\to\R$, $y^{+,n}:\R\to\R$ of equation \eqref{eqn:Efn} 
such that 
$$
x^\pm(t)\to\xi_1 \text{ and }y^{\pm,n}(t)\to\xi_{1,n} \text{ as }
t\to-\infty, 
$$
and 
 \begin{equation*}
\begin{aligned}
W^u_{\text{loc}}(\hat\xi_1)&=\{\hat\xi_1\}\cup \{x_t^-:t<s\}\cup\{x^+_t:t<s\},\\
 W^u_{\text{loc}}(\hat\xi_{1,n})&=\{\hat\xi_{1,n}\}\cup \{y_t^{-,n}:t<s_n\}
 \cup\{y^{+,n}_t:t<s_n\}
\end{aligned}
\end{equation*}
for some $s,s_n\in\R$. 
The forward extensions of the local unstable manifolds define the 1-dimensional leading unstable sets 
\begin{equation*}
W^u(\hat\xi_1)=
\{\hat\xi_1\}\cup W^{u,-}(\hat\xi_1)\cup W^{u,+}(\hat\xi_1),\quad 
 W^{u}(\hat\xi_{1,n})=\{\hat\xi_{1,n}\}\cup W^{u,-}(\hat\xi_{1,n})\cup 
 W^{u,+}(\hat\xi_{1,n})
 \end{equation*}
 where
\begin{equation}\label{Wu-repr}
\begin{aligned}
 W^{u,-}(\hat\xi_1)&=
\{x_t^-:t\in\R\}, \ W^{u,+}(\hat\xi_1)=\{x^+_t:t\in\R\},\\ 
W^{u,-}(\hat\xi_{1,n})&= \{y_t^{-,n}:t\in\R\},\ W^{u,+}(\hat\xi_{1,n})=
\{y^{+,n}_t:t\in\R\}.
\end{aligned}
\end{equation}
The sets $W^{u,\pm}(\hat\xi_1)$ and $W^{u,\pm}(\hat\xi_{1,n})$ 
are uniqe, and contain single entire orbits of $\Gamma$ and $\Phi^n$, 
respectively. 
The solutions $x^\pm$ and $y^{\pm,n}$ can be made unique 
by fixing their values at $t=0$ in the following way. 
Choose a $\kappa^+\in (0,1-\xi_1)$ and a $\kappa^-\in (-\xi_1,0)$. 
Then the global solutions $x^\pm$ and $y^{\pm,n}$ in \eqref{Wu-repr} are  
unique with the properties $x^\pm(0)=\kappa^\pm+\xi_1$ and 
$y^{\pm,n}(0)=\kappa^\pm +\xi_{1,n}$, see Section 4.

For $d>c>0$, there exist unique $t_1,t_2$ with $t_2>t_1+1>0$ such that 
$x^+(t_1)=x^+(t_2)=1$, $x^+$ strictly increases on $(-\infty,t_1+1]$, 
$x^+(t)>1$ for $t\in (t_1,t_2)$, and $x^+(t_2+s)=e^{-cs}$ for  $s\in[0,1]$.



Our first result considers equation \eqref{eqn:Eg}, and  
describes the structure of the closure $\overline{W^u(\hat\xi_{1})}$ of  $W^u(\hat\xi_{1})$.

\begin{theorem}\label{thm:Eg}
Suppose that condition $(C_g)(a)$ is satisfied. 
\begin{itemize}
\item[(A)] 
If $d>c$ then  $W^{u,-}(\hat\xi_1)$ connects $\hat\xi_1$ to $\hat 0$.
\item[(B)] 
There exists a unique $d^*=d^*(c)>c$ with the following properties.
\begin{itemize}
\item[(i)] 
If $d\in (c,d^*)$ then $W^{u,+}(\hat\xi_1)$ connects $\hat\xi_1$ to $\hat 0$.
\item[(ii)] 
If $d>d^*$, then there exist  $m_0,m_1$ with $0<m_0<1<m_1\le \frac{d}{c}$, 
and a $\sigma>0$  
such that $W^{u,+}(\hat\xi_1)\subset C_{m_0,m_1}^{2d} $, 
and, for all $T\ge t_2$, the intervals $[T,T+\sigma]$ contain at least one 
$t$ with $x^+(t)=1$.   
The $\omega$-limit set $\mathcal{A}=\omega_\Gamma(x^{+}_{t_2+1})$ is a nonempty 
compact subset of $C_{m_0,m_1}^{2d} $ with $x^+_t\to \mathcal{A}$ 
as $t\to\infty$, and 
$\hat{0}\notin \mathcal{A}$, $\hat{\xi}_1\notin\mathcal{A}$.
\item[(iii)] 
If $d>d^*$ and $(C_g)(a)(b)$ are satisfied, then $x^+_t\in \mathcal{O}=
\{p_t:t\in [0,\omega_p]\}$ for all $t\ge t_2+1$, and $\mathcal{A}=\mathcal{O}$.  
 \end{itemize} 
\end{itemize}
\end{theorem}

Consequently, if $d\in (c,d^*)$, then 
\begin{equation}\label{struct-gW0}
\overline{W^u(\hat\xi_{1})}
=\{\hat{0}\}\cup \{x_t^{-}:t\in\R\}\cup\{\hat\xi_{1}\}\cup \{x^{+}_t:t\in\R\},  
\end{equation}
and  if  $d>d^*$, then
\begin{equation}\label{struct-gW1}
\overline{W^u(\hat\xi_{1})}
=\{\hat{0}\}\cup \{x_t^{-}:t\in\R\}\cup\{\hat\xi_{1}\}\cup \{x^{+}_t:t\in\R\} 
\cup \mathcal{A}.
\end{equation}

The second main result of the paper shows that 
$d^*(c)$ of Theorem \ref{thm:Eg} is a threshold value 
for equation \eqref{eqn:Efn} as well. 
In the sequel, on stability of a periodic orbit we mean that it is hyperbolic, orbitally stable, and exponentially attractive with asymptotic phase. 

\begin{theorem}\label{thm:Efn}
Suppose that conditions $(C_g)(a)$ and $(C_{f_n})(a)(b)$ are satisfied. 
\begin{itemize}
\item[(A)] 
If $d>c$ and $n$ is sufficiently large, then  $W^{u,-}(\hat\xi_{1,n})$ connects $\hat\xi_{1,n}$ to $\hat 0$.
\item[(B)] For $d^*=d^*(c)>c$ given in Theorem \ref{thm:Eg} the following properties hold.
\begin{itemize}
\item[(i)] 
If $d\in (c,d^*)$ and $n$ is sufficiently large, then $W^{u,+}(\hat\xi_{1,n})$ connects $\hat\xi_{1,n}$ to $\hat 0$.
\item[(ii)] 
If $d>d^*$, $n$ is sufficiently large and $\delta>0$ is sufficiently small,  then there exist  $\tilde{m}_0,\tilde{m}_1$ with $0<\tilde{m}_0<1<\tilde{m}_1\le 2d/c$, and a $\tilde{\sigma}>0$  
such that 
$W^{u,+}(\hat\xi_{1,n})\subset C_{\tilde{m}_0,\tilde{m}_1}^{8d}$, and,   
for all $T\ge t_2$, the intervals $[T,T+\tilde{\sigma}]$ contain at least one $t$ with $y^{+,n}(t)\in[1-\delta,1+\delta]$.   
The $\omega$-limit set $\mathcal{A}_n=\omega_{\Phi^n}(y^{+,n}_{t_2+1})$ 
is a nonempty, compact invariant subset of $C_{\tilde{m}_0,\tilde{m}_1}^{8d}$ with $y^{+,n}_{t}\to \mathcal{A}_n$ as $t\to\infty$, 
$\hat{0}\notin \mathcal{A}_n$, $\hat{\xi}_{1,n}\notin\mathcal{A}_n$. 
\item[(iii)] 
Suppose $d>d^*$ and conditions $(C_g)(a)(b)$, $(C_{f_n})(a)(b)$. 
Then, for all sufficiently large $n$,   
equation \eqref{eqn:Efn}  has a stable  periodic orbit 
$\mathcal{O}^n$ such that $\mathcal{A}_n=\omega(y^{+,n}_{t_2+1})=\mathcal{O}^n$.  
 \end{itemize} 
\end{itemize}
\end{theorem}

Theorem \ref{thm:Efn} gives a clear picture of the 
 closure of the unstable set $W^u(\hat{\xi}_{1,n})$ for all large $n$, see 
 Figures \ref{fig:1a}, \ref{fig:1b}, \ref{fig:1c}. 
In case $c<d<d^*$, $\overline{W^u(\hat\xi_{1,n})}$ is the union of two  
stationary points, $\hat{0}$ and  $\hat\xi_{1,n}$,  and two connecting orbits:  
\begin{equation}\label{struct-fnW0}
\overline{W^u(\hat\xi_{1,n})}
=\{\hat{0}\}\cup \{y_t^{-,n}:t\in\R\}\cup\{\hat\xi_{1,n}\}\cup \{y^{+,n}_t:t\in\R\},  
\end{equation}
and, in case $d>d^*$, $\overline{W^u(\hat\xi_{1,n})}$ is the union of three  disjoint compact invariant sets, $\hat{0}$, $\hat\xi_{1,n}$ and $\mathcal{A}_n$, and two connecting orbits: 
\begin{equation}\label{struct-fnW1}
\overline{W^u(\hat\xi_{1,n})}
=\{\hat{0}\}\cup \{y_t^{-,n}:t\in\R\}\cup\{\hat\xi_{1,n}\}\cup \{y^{+,n}_t:t\in\R\} 
\cup \mathcal{A}_n.
\end{equation}
 
By Theorems \ref{thm:Eg} and \ref{thm:Efn}, the case $d>d^*$ is analogous to a bistable situation: 
the 1-dimensional unstable sets   
 $W^u(\hat\xi_1)$ and $W^{u}(\hat\xi_{1,n})$ consist of connections from  
the stationary points $\hat\xi_1$ and  $\hat\xi_{1,n}$ to the stable 
stationary point $\hat{0}$, and to some compact subsets in $C^+$, respectively. 
 The stationary points $ \hat{\xi}_1$ and $\hat\xi_{1,n}$ 
belong to the intersection of the boundaries of the regions of attractions of $\hat{0}$, $\mathcal{O}$ and $\hat{0}$, $\mathcal{O}^n$, respectively. 
This result explains the long-term transient behavior observed numerically 
for several equations, in particular, for our prototype equation 
\eqref{eqn:proto}. 
It is expected that not only the stationary point $\hat\xi_{1,n}$ 
is on the intersection of the boundaries of the regions of attractions of 
$\hat{0}$, $\mathcal{O}^n$.
We prove that there are periodic orbits $\mathcal{Q}^n$ close to $\hat\xi_{1,n}$ and  
the leading 2-dimensional unstable set of the periodic orbit $\mathcal{Q}^n$ consists of 
connecting orbits from $\mathcal{Q}^n$ to $\hat{0}$ and to  $\mathcal{O}^n$. 

Assuming $(C_g)(a)(c)$ and $(C_{f_n})(b)(c)$, we find sequences 
$(a_n)_{n=1}^\infty$, $(b_n)_{n=1}^\infty$ satisfying condition 
$(C_{f_n})(a)$ so that  
a local Hopf bifurcation theorem at the stationary point 
$\hat\xi_{1,n}$ 
 guarantees that a non-trivial periodic solution $q^n:\R\to\R $ 
exists with minimal period $\omega^n>0$. Let $\mathcal{Q}^n=\{q^n_t:t\in[0,\omega^n]\}$.
Analogously to the construction of the leading 1-dimensional local unstable manifold at the stationary point $\hat\xi_{1,n}$ of $\Phi^n$, 
a 1-dimensional leading local unstable manifold $W^u_{\text{loc}}(q_0^n,\Phi^n(\omega^n,\cdot))$
can be constructed for the period map $\Phi^n(\omega^n,\cdot)$ at its 
fixed point $q_0^n\in \mathcal{Q}^n$. 
Due to the strict monotonicity of $f_n$ near $\xi_{1,n}$, 
 $W^u_{\text{loc}}(q_0^n,\Phi^n(\omega^n,\cdot))$ can be decomposed 
 into three 
disjoint subsets
$$
W^u_{\text{loc}}(q_0^n,\Phi^n(\omega^n,\cdot)) = 
W^{u,-}_{\text{loc}}(q_0^n,\Phi^n(\omega^n,\cdot)) \cup 
\{q_0^n\}
\cup 
W^{u,+}_{\text{loc}}(q_0^n,\Phi^n(\omega^n,\cdot))
$$
so that $\psi\ll q_0^n$ for $\psi\in W^{u,-}_{\text{loc}}(q_0^n,\Phi^n(\omega^n,\cdot))$, and $q_0^n \ll \psi$ for $\psi\in W^{u,+}_{\text{loc}}(q_0^n,\Phi^n(\omega^n,\cdot))$, See Section \ref{sec3}. 
The 2-dimensional leading unstable set $ W^u\left(\mathcal{Q}^n\right) $
of the periodic orbit $\mathcal{Q}^n$ is defined as the forward extension of 
the leading 1-dimensional local unstable manifold 
$W^u_{\text{loc}}(q_0^n,\Phi^n(\omega^n,\cdot))$ under the semiflow $\Phi^n$, 
that is, 
$$
W^u\left(\mathcal{Q}^n\right)=\Phi^n\left(
[0,\infty)\times  W^u_{\text{loc}}(q_0^n,\Phi^n(\omega^n,\cdot)) \right).  
$$ 
In addition, the decomposition of $W^u_{\text{loc}}(q_0^n,\Phi^n(\omega^n,\cdot))$
induces 
$$
W^u\left(\mathcal{Q}^n\right)=
W^{u,-}\left(\mathcal{Q}^n\right)\cup \mathcal{Q}^n \cup
W^{u,+}\left(\mathcal{Q}^n\right)
$$ 
where 
$$
W^{u,\pm}\left(\mathcal{Q}^n\right)=\Phi^n\left(
[0,\infty)\times  W^{u,\pm}_{\text{loc}}(q_0^n,\Phi^n(\omega^n,\cdot))\right).
$$
See Section \ref{sec3} for the details.

The existence of $\mathcal{Q}^n$ cannot be guaranteed for all 
sequences $(a_n)_{n=1}^\infty$, $(b_n)_{n=1}^\infty$ satisfying condition 
$(C_{f_n})(a)$. Instead, we construct suitable sequences  satisfying condition 
$(C_{f_n})(a)$. 
 
The third main result of this paper is as follows. See also Figure \ref{fig:1d}.

\begin{theorem}\label{thm:Hopf}
	Suppose conditions $(C_g)(a)(c)$ and $(C_{f_n})(b)(c)$.  
	Then there exist sequences 
	$(a_n)_{n=1}^\infty$ and  $(b_n)_{n=1}^\infty$
	of positive reals  satisfying $(C_{f_n})(a)$,
	 such that, for all sufficiently large $n$, 
	 equation \eqref{eqn:Efn} possesses a nontrivial periodic solution 
		$q^n: \mathbb{R} \to \mathbb{R}$ with minimal period $\omega^n>0$.  
The 2-dimensional leading unstable set $W^u\left(\mathcal{Q}^n\right)$	of 
the periodic orbit $\mathcal{Q}^n=\{q_t^n:t\in [0,\omega^n]\}$ 
can be decomposed as 
$$
W^u\left(\mathcal{Q}^n\right)=
W^{u,-}\left(\mathcal{Q}^n\right)\cup \mathcal{Q}^n \cup
W^{u,+}\left(\mathcal{Q}^n\right)
$$ 
so that the following statements hold. 
\begin{itemize}
\item[(i)]
$W^{u,-}\left(\mathcal{Q}^n\right)$ connects $\mathcal{Q}^n$ to $\hat{\xi}_{1,n}$.
\item[(ii)] 
If $d\in(c,d^*)$, $n$ is sufficiently large, then 
$W^{u,+}\left(\mathcal{Q}^n\right)$ consists of 
connecting orbits from $\mathcal{Q}^n$ to $\hat{0}$. 
\item[(iii)]
If $d>d^*$, $n$ is sufficiently large and $\delta>0$ is sufficiently small,  then with the constants  $\tilde{m}_0,\tilde{m}_1$ and  $\tilde{\sigma}>0$  
given in Theorem \ref{thm:Efn}(B)(ii), we have 
$W^{u,+}(\mathcal{Q}^n)\subset C_{\tilde{m}_0,\tilde{m}_1}^{8d}$, and,   
for all $\phi\in W^{u,+}(\mathcal{Q}^n)$ there exists a $t_\phi>0$ so that for all $T\ge t_\phi$ there is a $t\in [T,T+\tilde{\sigma}]$ with 
$y^\phi(t)\in [1-\delta,1+\delta]$ for the solution 
$y^\phi:[-1,\infty)\to \mathbb{R}$ of \eqref{eqn:Efn}.  
Moreover, for any $\phi\in W^{u,+}(\mathcal{Q}^n)$
the $\omega$-limit set 
$\omega_{\Phi^n}(\phi) $ is
a nonempty, compact invariant subset of $C_{\tilde{m}_0,\tilde{m}_1}^{8d}$ with $\Phi^n(t,\phi)\to \omega_{\Phi^n}(\phi)$ as $t\to\infty$, 
$\hat{0}\notin \omega_{\Phi^n}(\phi)$, $\hat{\xi}_{1,n}\notin \omega_{\Phi^n}(\phi)$, 
and $\mathcal{Q}^n\cap \omega_{\Phi^n}(\phi)=\emptyset$.
\item[(iv)]
If $d>d^*$ and   condition $(C_g)(b)$  holds as well, 
and $n$ is sufficiently large, then 
equation \eqref{eqn:Efn}  has a stable  periodic orbit 
$\mathcal{O}^n$ such that $\omega_{\Phi^n}(\phi)=\mathcal{O}^n$ for all 
$\phi\in W^{u,+}(\mathcal{Q}^n)$. 
Therefore,  
$$
\overline{W^{u}\left(\mathcal{Q}^n\right)}
=
\{\hat{0}\}\cup W^{u,-}\left(\mathcal{Q}^n\right) 
\cup \mathcal{Q}^n \cup W^{u,+}\left(\mathcal{Q}^n\right) 
\cup \mathcal{O}^n,
$$
where
$W^{u,-}\left(\mathcal{Q}^n\right)$ consists of 
connecting orbits from $\mathcal{Q}^n$ to $\hat{0}$, and 
$W^{u,+}\left(\mathcal{Q}^n\right)$ consists of 
connecting orbits from $\mathcal{Q}^n$ to $\mathcal{O}^n$. 
\end{itemize}
\end{theorem}

Theorems \ref{thm:Eg} and \ref{thm:Hopf} describe only a part of the global attractors of the semiflows $\Phi^n$ showing different types of connecting 
orbits and periodic orbits. 
Complicated dynamics is not guaranteed in this paper. 
However, the results give information on parameter values for which chaotic dynamics might appear. 
One possibility is the critical case $d=d^*(c)$. 
We suspect that, for some paramaters,  complex dynamics is possible within the sets $\mathcal{A}$ and $\mathcal{A}_n$, given  
in Theorems \ref{thm:Eg} and \ref{thm:Efn}.  

 \begin{figure}[!htbp]
	\centering
	\begin{subfigure}[b]{0.44\textwidth}
		\centering
		\includegraphics[width=0.95\textwidth]{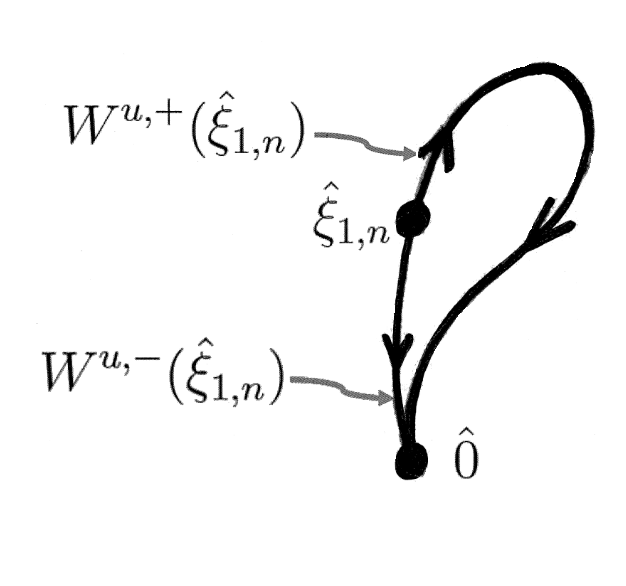}
		\caption{$d \in (c,d^*)$, $n$ large, $(C_g)(a)$ and $(C_{f_n})(a)(b)$ hold.}
		\label{fig:1a}
	\end{subfigure}
	\hfill
	\begin{subfigure}[b]{0.44\textwidth}
		\centering
		\includegraphics[width=0.95\textwidth]{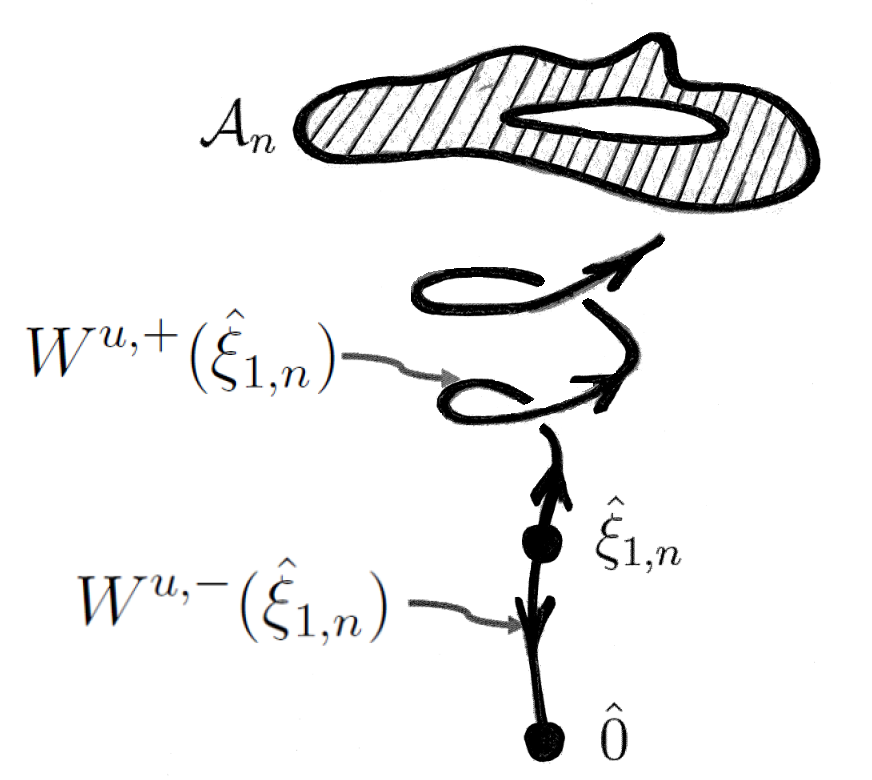}
		\caption{ $d > d^*$, $n$ large, $(C_g)(a)$ and $(C_{f_n})(a)(b)$ hold.
		}
		\label{fig:1b}
	\end{subfigure}
	\vspace{0.5cm}
	\begin{subfigure}[b]{0.44\textwidth}
		\centering
		\includegraphics[width=0.95\textwidth]{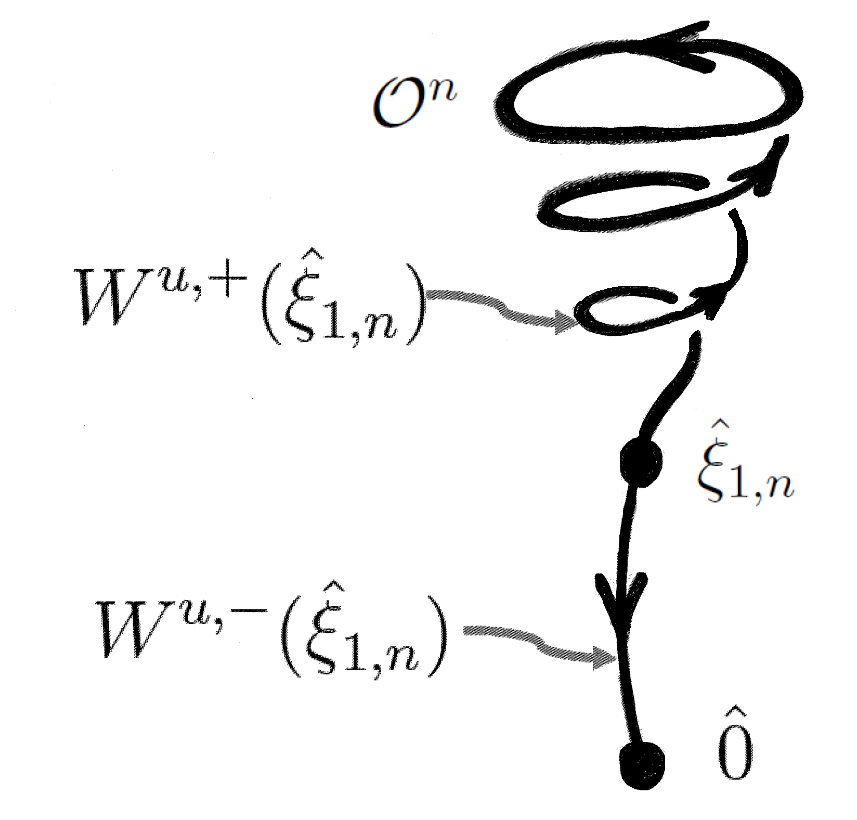}
		\caption{$d > d^*$, $n$ large, $(C_g)(a)(b)$ and $(C_{f_n})(a)(b)$ hold.
		}
		\label{fig:1c}
	\end{subfigure}
	\hfill
	\begin{subfigure}[b]{0.44\textwidth}
		\centering
		\includegraphics[width=0.95\textwidth]{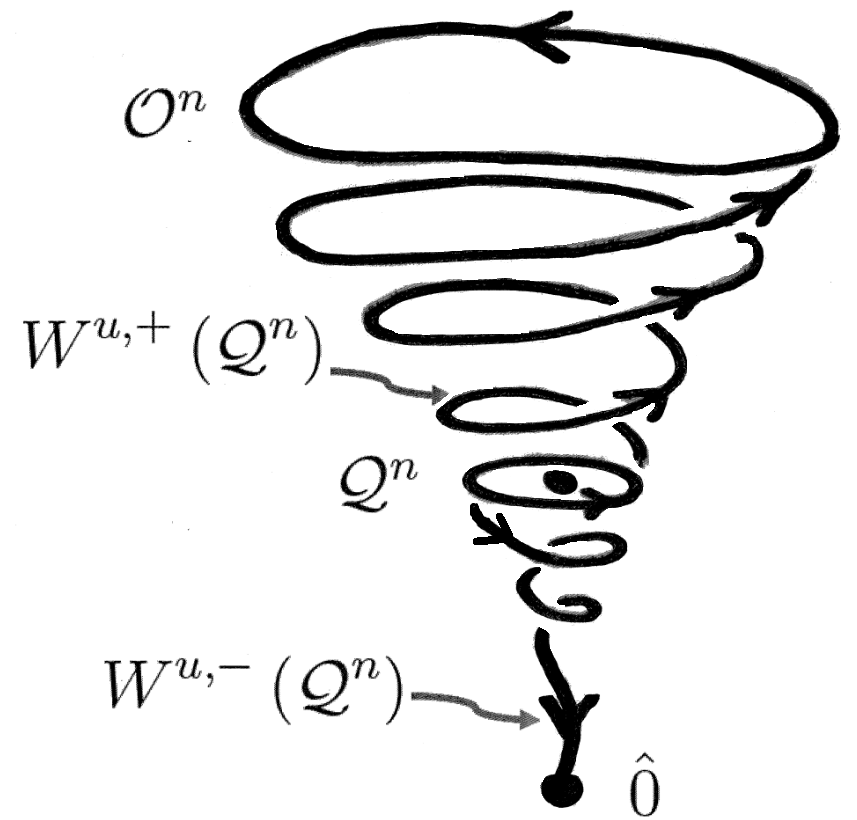}
		\caption{$d > d^*$, $n$ large, $(C_g)(a)(b)(c)$ and $(C_{f_n})(b)(c)$ hold.}
		\label{fig:1d}
	\end{subfigure}
	\caption{Visualization of the results of Theorems \ref{thm:Efn} and \ref{thm:Hopf}}
	\label{fig:1}
\end{figure}

In the proofs of the above theorems the starting point is the construction 
of leading local unstable manifolds for semiflows at stationary points and 
for period maps at periodic orbits. 
All of these local objects are in regions of the phase space where the
semiflows $\Gamma$ and $\Phi^n$  are monotone. 
Then the theory of monotone semiflows \cite{Hirsch}, \cite{Smith}, and in particular the 
technical tools developed in  \cite{KWW} can be applied. 
The difficulties arise when the local unstable manifolds are extended 
forward by the semiflows, and the orbits enter into regions where 
the feedback functions are not monotone. 
This  appears for the 'upper' parts of the unstable sets, denoted 
above by a plus sign in the upper index. 

In order to overcome the difficulties, a key result is a strong stability property of the set $W^{u,+}(\hat\xi_1)$. 
The solutions $x^+$ and $y^{+,n}$ appearing in  \eqref{Wu-repr} 
are unique as they defined before Theorem \ref{thm:Eg}. 
Then  $y^{+,n}(t)\to x^+(t)$ 
as $n\to\infty$ uniformly on compact subsets of $(-\infty,0]$, provided 
condition $(C_{f_n})(a)(b)$ holds, see Proposition  \ref{un-to-u+} and 
Section \ref{sec4}. 
This fact allows to apply the technique developed in \cite{BKSZ} to estimate the distance of 
solutions of \eqref{eqn:Eg} and \eqref{eqn:Efn} on fixed compact intervals, for large $n$.  

The location of the $\omega$-limit sets  $\mathcal{A}$ and $\mathcal{A}_n$ 
in Theorems \ref{thm:Eg}, \ref{thm:Efn} and \ref{thm:Hopf} uses the above  properties  
of $x^+$, and the following observation. For the above defined $x^+$,  $x^+(t_2)=1$ and $x^+(t)=e^{-c(t-t_2)}$ for $t\in[t_2,t_2+1]$. 
Therefore, we can consider solutions of equation \eqref{eqn:Eg} on $[0,\infty)$ with initial function $x(t)=e^{-ct}$, $t\in[0,1]$. 
Depending on the parameter $d$,  different asymptotic behaviors are obtained 
as given in  Theorems \ref{thm:Eg}, \ref{thm:Efn} and \ref{thm:Hopf}. Again, since $g$ is monotone in $[0,1]$, 
results for monotone semiflows can be applied as long as $x^+(t)$ remains 
in $[0,1]$. By condition $(C_{f_n})(a)(b)$ and the stability property of 
$x^+$, the asymptotic properties of $x^+$ 
are preserved for large $n$ for some solutions of \eqref{eqn:Efn}.

As  $t_2>t_1+1$ with $x^+(t)>1$ for all $t\in(t_1,t_2)$,  
 there are $t_3>t_1$ and $\varepsilon>0$ with 
$t_3+1<t_2$ and  $x^+(t)\ge 1+2\varepsilon$ on the interval $[t_3,t_3+1]$.  
Hence, in particular, 
we conclude $y^{+,n}(t)\ge 1+\varepsilon $ in the interval $[t_3,t_3+1]$ for large $n$. 
Under the conditions $(C_g(b)$ and $(C_{f_n})(a)(b)$,  the main result of \cite{BKSZ} guarantees the  existence of a stable periodic orbit $\mathcal{O}^n$. 
 See also 
Propositions \ref{thm:main0} and \ref{prop:main0-application}. 
It is an important information that the set $\{\psi\in C^+:\psi\ge 1+\varepsilon\}$ is in the region of attraction of $\mathcal{O}^n$, for large $n$. 
Therefore, for large $n$, $y^{+,n}_{t_3+1}$ is in the region of attraction of $\mathcal{O}^n$, and $\mathcal{A}_n =\mathcal{O}^n$ follows.  

A similar result is valid for the 'upper' parts of the unstable sets 
$W^u\left(\mathcal{Q}^n\right)$ of the periodic orbits $\mathcal{Q}^n$ 
since $\mathrm{dist}(\hat\xi_1,\mathcal{Q}^n)\to 0$ as $n\to\infty$. 
 
The paper is organized as follows. In Section~\ref{sec2}, we collect some background material and preliminary propositions on delay differential equations with monotone feedbacks. We also review the basic definitions of the function spaces and the semiflow properties that will be used throughout the paper.

In Section~\ref{sec3}, we establish several results for equations whose nonlinearities are strictly increasing on certain intervals. In particular, we construct the 1-dimensional leading local unstable manifolds at the relevant equilibria and prove key monotonicity and ordering properties of the semiflow.

Section~\ref{sec4} is devoted to showing the closeness of solutions for the ``limiting'' discontinuous equation and a sequence of smooth approximations. We apply these estimates to control the global behavior of the perturbed systems when their parameters approach those of the limiting case which makes possible to prove Theorems~\ref{thm:Efn} and \ref{thm:Hopf}.

Section~\ref{sec5} studies the solutions of equation \eqref{eqn:Eg} on $[0,\infty)$ with initial function $x(t)=e^{-ct}$, $t\in[0,1]$,  $c>0$ is fixed, 
and $d>c$ is considered a parameter.  
The threshold value $d^*(c)$ is constructed, and by applying the results 
of Section~\ref{sec4} under condition $(C_{f_n})(a)(b)$, it is shown that 
some asymptotic properties of $x^+$ 
are preserved for the solutions of \eqref{eqn:Efn}, provided  $n$ is large. 
All of these results combined give  the proofs of Theorems~\ref{thm:Eg} and \ref{thm:Efn}. 

Section~\ref{sec6} proves  Theorem \ref{thm:Hopf} in seven steps.  By condition $(C_{g})(a)(c)$ a local Hopf bifurcation 
 is applied in Step 1. Step 2 constructs the sequences 
	$(a_n)_{n=1}^\infty$ and  $(b_n)_{n=1}^\infty$, and the periodic 
	orbits $\mathcal{Q}^n$ for  
	equation \eqref{eqn:Efn}. Step 3 applies the results of Section~\ref{sec3} to get the 2-dimensional leading unstable set $W^u\left(\mathcal{Q}^n\right)$	of  $\mathcal{Q}^n$. 
After these preparations, Steps 4--7 complete the proof of Theorem \ref{thm:Hopf}. 

Finally, Section~\ref{sec7} contains  examples for our prototype 
equation \eqref{eqn:proto}. The solutions are portrayed in the $(t,y)$ plane  complementing the visualization  on Figure \ref{fig:1}, 
and illustrating that, behind the simple-looking phase portraits of Figure \ref{fig:1} in the phase space, 
complex solution structures may appear.

	\section{Preliminaries}\label{sec2}
	
	Throughout the paper, we let $\mathbb{R}$ denote the set of all real numbers, $\mathbb{C}$ the set of complex numbers, $\mathbb{N}=\{1,2,\dots\}$, and $\mathbb{N}_0=\{0\}\cup \mathbb{N}$. 
	
	Let $C = C([-1,0], \mathbb{R})$ be the Banach space of continuous real-valued functions on $[-1,0]$, equipped with the norm $\|\phi\|  =  \max_{s \in [-1,0]} |\phi(s)|$. 
Recall from Section \ref{sec1} the subset $C^+$ of $C$, and for given  
$\kappa_2 > \kappa_1>0$, $L>0$ the sets
$ C_{\kappa_1,\kappa_2}$ and  
$C_{\kappa_1,\kappa_2}^L$. 
	
	 We consider equations of the form 
\begin{equation}\label{eqn:general}
x'(t) = -\alpha x(t)+ G\left(t, x(t-1)\right)
\end{equation} 
	where $\alpha\ge 0$, and  $G:[t_0,\infty)\times \mathbb{R}\to \mathbb{R}$ with some $t_0\in  \mathbb{R}$.   
	A solution of \eqref{eqn:general} on the interval $[t_1-1,\infty)$ with $t_1\ge t_0$ is a continuous function $ x:[t_1-1,\infty)\to\R $  such that 
	$[t_1-1,\infty)\ni s \mapsto G\left(s, x(s-1)\right)\in\R$ is locally integrable, and 
	\begin{equation}\label{soln-def}
	x(t)=e^{-\alpha(t-\tau)}x(\tau)+\int_\tau^t e^{-\alpha(t-s)} 
	G(s,x(s-1)) \  ds 
	\quad \text{ for all }t_1\le\tau<t<\infty
	\end{equation}
	holds. 
In case $G$ is defined on $\R^2$, the continuous function $x:\R\to\R$ is a solution of \eqref{eqn:general} on $\R$ if it is a solution of \eqref{eqn:general} on $[t_1-1,\infty)$ for any $t_1\in\R$. 
Observe that \eqref{soln-def} is obtained from the variation-of-constants 
formula for ordinary differential equations applied to \eqref{eqn:general}. 
Moreover, if $G$ is continuous then a  solution $ x:[t_1-1,\infty)\to\R $ 
of \eqref{eqn:general} in the above sense is differentiable on $ (t_1,\infty) $, and \eqref{eqn:general} holds for all $t\in(t_1,\infty)$. 

In particular, if $T_1\in\mathbb{R}$ and $T_1<T_2\le\infty$, then the 
continuous function $y:[T_1-1,T_2)\to [0,\infty)$ is a solution of 
\eqref{eqn:Efn} on the interval $[T_1-1,T_2)$ if 
\begin{equation}\label{inteqn:y}
y(t)=e^{-a_n(t-\tau)}y(\tau)+b_n\int_\tau^t e^{-a_n(t-s)} 
	f_n(y(s-1))    ds 
	\quad \text{ holds for all }T_1\le\tau<t<T_2.
\end{equation}
 Clearly, 
\eqref{inteqn:y}  implies that $y$ is 
continuously differentiable on $(T_1,T_2)$, and equation 
\eqref{eqn:Efn} is satisfied on $(T_1,T_2)$
Assuming condition $(C_{f_n})(a)(b)$, by the method of steps 
for \eqref{inteqn:y}, or by  
classical results from \cite{{DVLGW},{Hale}},  it is easy to 
obtain that for any $\psi\in C^+$, there is a unique solution 
 $y^\psi:[-1,\infty)\to\R$ of equation \eqref{eqn:Efn} with 
 $y_0=\psi$, and 
 $$
\Phi^n:[0,\infty)\times C^+ \ni (t,\psi)\mapsto y_t^\psi\in C^+
$$
is a continuous semiflow.  	

If $(C_g)(a)$ holds, $T_1\in\mathbb{R}$, $T_1<T_2\le\infty$, and 
$x:[T_1-1,T_2)\to [0,\infty)$ is a continuous function, then 
$[T_1,T_2)\in s\mapsto g(x(s-1))\in [0,1]$ is measurable. 
Therefore, the continuous $x:[T_1-1,T_2)\to [0,\infty)$ 
is a solution of \eqref{eqn:Eg} on the interval $[T_1-1,T_2)$ if 
\begin{equation}\label{inteqn:x}
x(t)=e^{-c(t-\tau)}x(\tau)+d\int_\tau^t e^{-c(t-s)} 
	g(x(s-1))    ds \quad \text{ holds for all }T_1\le\tau<t<T_2.
\end{equation}
Again,  for given   $\phi\in C^+$, the method of steps allows to define a unique solution $x^\phi:[-1,\infty)\to [0,\infty)$ with  $x_0^\phi=\phi$. 
Moreover,  if $t>0$ and $x^\phi(t-1)\ne 1$, then $x^\phi$ is differentiable at $t$, and equation \eqref{eqn:Eg} holds at $t$. 
It is easy to see that the map 
$$\Gamma:[0,\infty)\times C^+\ni (t,\phi)\mapsto x_t^\phi \in C^+
$$ 
defines a semiflow. However, $\Gamma$ is not continuous. For example, 
$\Gamma(1,\hat{1})=e^{-c(1+s)}(1-d/c)+d/c$,  and 
for $\varepsilon>0$, $\Gamma(1,\widehat{1+\varepsilon})(s)=e^{-c(1+s)}(1+\varepsilon)$, $s\in[-1,0]$. Hence, 
$\Gamma(1,\widehat{1+\varepsilon})\not\to \Gamma(1,\hat{1}$ as $\varepsilon\to 0+$. 
On the other hand, $\Gamma$ restricted to $[0,\infty)\times \{\varphi\in C^+: \varphi(s)<1,\ s\in[-1,0]\}$ is continuous.  
 
For some  $\psi\in C^+$, there are unique solutions 
$x:\R\to\R$ and $y:\R\to\R$ of equations \eqref{eqn:Eg} 
and \eqref{eqn:Efn} with 
$x_0=\phi$ and $y_0=\psi$, respectively. 
These solutions will be  denoted by 
$x^\phi$ and $y^\psi$, as well. 
If we want to emphasize the dependence of the solution $y^\psi$ of \eqref{eqn:Efn} on $n$, then 
$y^{n,\psi}$ is used instead of $y^\psi$. 
 
The following boundedness result is valid.

\begin{proposition}\label{prop:bounds} Assume  $(C_g)(a)$ and 
$(C_{f_n})(a)(b)$. 
\begin{itemize}
\item[(i)] 
$$
\Gamma\left( [0,\infty)\times C_{0,d/c}\right) \subset  C_{0,d/c},\quad 
\Gamma\left([1,\infty)\times C_{0,d/c}\right) \subset  C_{0,d/c}^{2d}.
$$
\item[(ii)] For all sufficiently large $n$, 
$$
\Phi^n\left( [0,\infty)\times C_{0,2d/c}\right) \subset  C_{0,2d/c},\quad 
\Phi^n\left([1,\infty)\times C_{0,2d/c}\right) \subset  C_{0,2d/c}^{8d}.
$$
\end{itemize}
\end{proposition}

\begin{proof}
Let $x=x^\phi:[-1,\infty)\to \mathbb{R} $ and 
$y=y^\psi:[-1,\infty)\to \mathbb{R} $ be solutions of \eqref{eqn:Eg} and 
\eqref{eqn:Efn} with $\phi\in C_{0,d/c}$ and $\psi\in C_{0,2d/c}$, respectively. 
The nonegativity of $x$ and $y$ is clear from equations \eqref{inteqn:x}, \eqref{inteqn:y}. 
By using $0\le g\le 1$, $0\le f_n\le 1$ and \eqref{inteqn:x},\eqref{inteqn:y}, for $t\ge 0$ it follows that
$$
x(t)=e^{-ct}x(0)+d\int_0^t e^{-c(t-s)}g(x(s-1))\, ds
\le e^{-ct}\dfrac{d}{c}+\dfrac{d}{c}\left(1-e^{-ct} \right)=\dfrac{d}{c},
$$ 
and
$$
y(t)=e^{-a_nt}y(0)+b_n\int_0^t e^{-a_n(t-s)}f_n(y(s-1))\, ds
\le e^{-a_n t}\dfrac{2d}{c}+\dfrac{b_n}{a_n}\left(1-e^{-a_nt} \right)\le\dfrac{2d}{c}.
$$ 
provided  $n $ is so large that $\frac{b_n}{a_n}\le \frac{2d}{c}$.

For $0\le t_1 < t_2$, equations \eqref{inteqn:x}, \eqref{inteqn:y} and the 
bounds obtained above 
imply
\begin{align*}
|x(t_2)-x(t_1)|&= \left| e^{-c(t_2-t_1)}x(t_1)-x(t_1)+d\int_{t_1}^{t_2} e^{-c(t_2-s)}g(x(s-1))\, ds \right|\\
& \le \dfrac{2d}{c}\left(1-e^{-c(t_2-t_1)} \right)\le 2d(t_2-t_1),
\end{align*}
and
\begin{align*}
|y(t_2)-y(t_1)|&= \left| e^{-a_n(t_2-t_1)}y(t_1)-y(t_1)+b_n\int_{t_1}^{t_2} e^{-a_n(t_2-s)}f_n(y(s-1))\, ds \right|\\
& \le \left(y(t_1)+\dfrac{b_n}{a_n} \right)\left(1-e^{-a_n(t_2-t_1)} \right)
\le \left(\dfrac{2d}{c}+\dfrac{b_n}{a_n} \right)a_n (t_2-t_1)\le 
8d(t_2-t_1),
\end{align*}
provided $\frac{b_n}{a_n}\le \frac{2d}{c}$ and $a_n\le 2c$.

Statements (i) and (ii) are straightforward consequences of the above estimations. 
\end{proof}	
	
The sets 	$C_{0,d/c}^{2d}$ and $C_{0,2d/c}^{8d}$ are compact subsets 
of $C^+$. Then, by Proposition \ref{prop:bounds}, for each $\phi\in C_{0,d/c}$ and $\psi\in C_{0,2d/c}$, the $\omega$-limit sets
$$
\omega_\Gamma(\phi)=\left\lbrace \chi\in C^+: \Gamma(t_n,\phi) \to \chi 
\text{ for some sequence }(t_k)_1^\infty \text{ with }t_k\to\infty  \right\rbrace 
$$
and 
$$
\omega_{\Phi^n}(\phi)=\left\lbrace \chi\in C^+: \Phi^n(t_k,\phi) \to \chi 
\text{ for some sequence }(t_k)_1^\infty \text{ with }t_k\to\infty  \right\rbrace 
$$
are nonempty, compact subsets in   $C_{0,d/c}^{2d}$ and $C_{0,2d/c}^{8d}$ 
with the properties 
$$
\Gamma(t,\phi) \to  \omega_\Gamma(\phi), \quad 
\Phi^n(t,\phi) \to \omega_{\Phi^n}(\phi) \text{ as }t\to\infty, 
$$
respectively. 
In addition, by the continuity of $\Phi^n$, the limit set $\omega_{\Phi^n}(\phi)$ is invariant as well.

Now we recall the main result, Theorem 3.1 of \cite{BKSZ} as a proposition,  to show that our 
conditions $(C_g)(a)(b)$ and $(C_{f_n})(a)(b)$	guarantee the existence of  stable periodic 
orbits $\mathcal{O}^n$ of equations \eqref{eqn:Efn}. We state 
it in a slightly weaker form since here we require stronger conditions 
on $g$ than those in \cite{BKSZ}.

Condition $(C_g)(b)$ assumes the existence of a periodic solution 
$p:\R\to\R$ of \eqref{eqn:Eg} with minimal period $\omega_p$. Let 
	$
	\mathcal{O} =\{ p_t : t \in [0,\omega_p]\}
	$
be the corresponding periodic orbit. For the given periodic solution $p$, Section 2 of \cite{BKSZ} defines 
$m=[\omega_p]$ as the integer part of $\omega_p$, and chooses the 
constants $\kappa_1,\kappa_2$ such that
$$0<\kappa_1<\min_{t\in\R}p(t),\quad   \max_{t\in\R}p(t)<\kappa_2.
$$ 

For the construction of the two additional constants  
	$\varepsilon_0, \varepsilon_1\in(0,1)$, in which 
only the properties of $p$ are used,  
	we refer to Section 2 of \cite{BKSZ}.  It is important to note 
	that for any $\varepsilon>0$ the constants  
	$\varepsilon_0, \varepsilon_1$ can be chosen with 
	$\varepsilon_0<\varepsilon$.

	\begin{proposition}\label{thm:main0}\emph{(\cite{BKSZ}, Theorem 3.1.)}
		Suppose that condition $(C_g)(a)(b)$ holds, parameters  $a>0$, $b>0$, and 
		a $C^1$-smooth function $f:[0,\infty)\to [0,1)$ 
		 with $f(0)=0$ are given  such that
		\begin{itemize}
			\item[(i)] 
			$\displaystyle |c-a| (1+\varepsilon_0) 
			 <  
			\frac{\varepsilon_1}{2},\quad 
			|d-b|
			 <  
			\varepsilon_1,$
			\item[(ii)] 
			$\displaystyle \|f-g\|_{[\kappa_1,1-\varepsilon_1]\cup[1+\varepsilon_1,\kappa_2]} 
			 <  
			\varepsilon_1,
			\quad  
			 b  \|f\|_{[1+\varepsilon_0,\infty)} 
			 <  
			\frac{\varepsilon_1}{2},$ 
			\item[(iii)]
			$\displaystyle 
			b\Bigl(1+\frac{4b}{a}\Bigr) \|f'\|_{[1+\varepsilon_0,\kappa_2]}
			\Bigl(1+b \|f'\|_{[\kappa_1,\kappa_2]}\Bigr)^m
			 <  
			1.
			$
		\end{itemize}
		Then equation \eqref{eqn:Ef} admits a periodic solution 
		$
		q: \mathbb{R}\to\mathbb{R}
		\quad
		(\text{with minimal period } \omega_q>0),
		$
		such that
		$
		|\omega_q-\omega_p|
		 <  
		\frac{\varepsilon_0}{c},
		\quad
		\mathcal{O}_q 
		 =  
		\{ q_t : t\in[0,\omega_q]\}
		$
		is hyperbolic, orbitally stable, exponentially attractive with asymptotic phase, and 
		$\dist(\mathcal{O}_q, \mathcal{O}_p) < \varepsilon_0.$
		Moreover,  the region of attraction of $\mathcal{O}_q$ includes
		the set $
		\{ 
		\phi \in C^+ 
		:  
		\phi(s)\ge 1+\varepsilon_0 \text{ for all } s\in [-1,0]
		\}$.
	\end{proposition}

The next result shows that conditions $(C_g)(a)(b)$ and $(C_{f_n})(a)(b)$ guarantee the existence of stable periodic orbits for equations   
\eqref{eqn:Efn} with sufficiently large $n$. 

	\begin{proposition}\label{prop:main0-application}
		Suppose conditions $(C_g)(a)(b)$ and $(C_{f_n})(a)(b)$ hold. 
		Let $\varepsilon>0$ be given. Then, for all sufficiently large  
		 $n$, equation \eqref{eqn:Efn} admits a hyperbolic, orbitally stable, exponentially attractive periodic orbit $\mathcal{O}^n$, so that 
		the region of attraction of   $\mathcal{O}^n$ contains 
		the set 
		$\{ \psi \in C^+ :  \psi(s)\ge 1+\varepsilon, \text{ for all }
		s\in [-1,0]	\}$.
	\end{proposition}
	
	\begin{proof} 
	We check conditions (i), (ii), (iii) of Proposition~\ref{thm:main0} in turn, under the assumptions $(C_{f_n})(a)(b)$.  By the remark preceding Proposition  \ref{thm:main0}, the constants $\varepsilon_0, \varepsilon_1$ can be chosen 
	so that 	$\varepsilon_0<\varepsilon$. 
	Then in order to prove  the proposition it is sufficient to apply Proposition  \ref{thm:main0} with $a=a_n$, $b=b_n$, $f=f_n$, $m=[\omega_p]$, and show that 
 	conditions (i), (ii), (iii) of Proposition \ref{thm:main0} hold for all 
 	sufficiently large $n$. 
 	
 \((C_{f_n})(a)\) implies that, for all large $n$, (i) is satisfied  with  $a=a_n$, $b=b_n$, and $b_n<2d$, $a_n>c/2$. 
 
 Choose 	$\kappa = \min\{\varepsilon_0, \varepsilon_1\}$. Then 
 \((C_{f_n})(b)\) guarantees that (ii) holds with $f=f_n$  provided $n$ 
 is large enough. 
 
 It remains to verify (iii) of  Proposition \ref{thm:main0}  
 with $a=a_n$, $b=b_n$, $f=f_n$, for all large $n$.
 By $b_n<2d$, $a_n>c/2$ and the choice of $\kappa$, for large $n$, 
 we have 
 \begin{align*}
 	b_n\Bigl(1+\tfrac{4b_n}{a_n}\Bigr)
 	\|f_n'\|_{[1+\varepsilon_0,\kappa_2]}
 	\Bigl(1+b_n \|f_n'\|_{[\kappa_1,\kappa_2]}\Bigr)^m 
 	\\
 	<2d\Bigl(1+\tfrac{16d}{c}\Bigr)
 	\|f_n'\|_{[1+\kappa,\infty)}
 	\Bigl(1+2d \|f_n'\|_{[0,\infty)}\Bigr)^m.  	
 \end{align*}
 
It suffices to guarantee that the last term is less than 1 
for all large $n$.   
 Observe that \((C_g)(a)\) and \((C_{f_n})(b)\) together imply 
 $$
 \|f_n'\|_{[0,\infty)}\to\infty \text{ as }n\to\infty.
 $$
So,  $ \|f_n'\|_{[0,\infty)}>1$ for all large $n$. 
Then, for all large $n$, 
\begin{align*}
\|f_n'\|_{[1+\kappa,\infty)}
 	\Bigl(1+2d \|f_n'\|_{[0,\infty)}\Bigr)^m
 	& = \|f_n'\|_{[1+\kappa,\infty)}
 	\Bigl(\|f_n'\|_{[0,\infty)}\Bigr)^m
 	\Bigl(\dfrac{1}{\|f_n'\|_{[0,\infty)}}+2d \Bigr)^m
 	\\
 	& \le 
 	\|f_n'\|_{[1+\kappa,\infty)}
 	\Bigl(\|f_n'\|_{[0,\infty)}\Bigr)^m
 	\Bigl(1+2d \Bigr)^m,
\end{align*} 
and by \((C_{f_n})(b)\) we conclude that condition (iii) holds for all large $n$. 
This completes the proof. 
	\end{proof}

The next result shows that, for large $n$, the stationary solutions of equation \eqref{eqn:Efn} are close to thae stationary solutions of 
\eqref{eqn:Eg}.

\begin{proposition}\label{prop:zeros-fn} 
Suppose conditions $(C_g)(a)$ and $(C_{f_n})(a)(b)$ hold. Let $\kappa\in (\xi_1,1)$. 
Then, for all sufficiently large $n$, 
  the only zeros of the maps $[0,\kappa]\ni\xi\mapsto -a_n\xi +b_nf_n(\xi)\in \R$ are  $0$ and $\xi_{1,n}$, moreover,  
$ \xi_{1,n}\to \xi_1$ as $n\to\infty$. 
\end{proposition}	
	 
\begin{proof}

For $\xi\in [0,\kappa]$, define the functions 
$$
\alpha(\xi)=-c\xi+dg(\xi),\quad \alpha_n(\xi)=-a_n\xi+b_nf_n(\xi).
$$

From $(C_g)(a)$ one has $ \alpha'(0)<0 $, $ \alpha'(\xi_1)>0 $, and 
$\alpha(\xi)<0$ for $\xi\in(0,\xi_1)$, $\alpha(\xi)>0$ for $\xi\in(\xi_1,\kappa]$. 
Then, by continuity, there exist positive constants $c_1,c_2,c_3$ so that 
 \begin{equation*}
\begin{aligned}
\alpha'(\xi )<-c_2 & \text{ for } \xi\in[0,c_1], \\
\alpha'(\xi )>c_2 & \text{ for } \xi\in[\xi_1-c_1,\xi_1+c_1], \\
\alpha(\xi )<-c_3 & \text{ for } \xi\in[c_1,\xi_1-c_1], \\
\alpha(\xi )>c_3 & \text{ for } \xi\in[\xi_1+c_1,\kappa]. 
\end{aligned}
\end{equation*}

From $(C_{f_n})(a)(b)$ it follows that there is an $n_2$ such that, for all 
$n>n_1$,
 \begin{equation*}
\begin{aligned}
\alpha_n'(\xi )<-\frac{c_2}{2} & \text{ for } \xi\in[0,c_1], \\
\alpha_n'(\xi )> \frac{c_2}{2} & \text{ for } \xi\in[\xi_1-c_1,\xi_1+c_1], 
\\
\alpha_n(\xi )<-\frac{c_3}{2} & \text{ for } \xi\in[c_1,\xi_1-c_1], \\
\alpha_n(\xi )>\frac{c_3}{2} & \text{ for } \xi\in[\xi_1+c_1,\kappa], 
\end{aligned} 
\end{equation*}
and hence the satement for the unique existence of $0$ and $\xi_{1,n}$ 
follows. Clearly, $ \xi_{1,n}\to \xi_1$ as $n\to\infty$. 
\end{proof}

	\section{Results on monotone feedback}\label{sec3} 
	The feedback functions we study are  monotone on certain intervals.  For solutions whose ranges lie in these monotone regions, we may apply various well-known results from the theory of monotone semiflows, see 
	e.g. \cite{Hirsch}, \cite{Smith}. 
	In particular, several results, that we mention without proofs, are from    \cite{KWW}.  
		 
	Assume	that 
	\begin{enumerate}
		\item[(H)]
		the real constants $a>0, b>0,\xi_-<0,\xi_+>0$ and a bounded, continuously differentiable function $h:\mathbb{R}\to\mathbb{R}$ are given, such that $h(0)=0,  b h'(0)>a$,  $ -a \xi_\pm + b h(\xi_\pm)=0,  \frac{h(\xi)}{\xi} > \frac{a}{b} \text{ for all } \xi \in (\xi_-,0)\cup(0,\xi_+), \text{ and } h'(\xi) > 0 \text{ for all } \xi \in (-\infty,\xi_-)\cup(\xi_-,\xi_+)\cup(\xi_+,\infty).$
	\end{enumerate}

Consider the delay differential equation
	\begin{equation}\label{eqn:Eh}
		\tag{$E_{h}$}
		u'(t)= -au(t) + bh\left(u(t-1)\right).
	\end{equation}	

Observe that (H) requires $h'>0$ everywhere with the possible exception 
at $\xi_-$ and $\xi_+$. On the other hand, \cite{KWW} supposed $h'>0$ 
everywhere. Here we consider only 
solutions with ranges in the interval $(\xi_-,\xi_+)$ where \cite{KWW} 
is applicable.  We recall some results from \cite{KWW}. 

For each initial function $\phi \in C$, there is a unique solution $u^\phi: [-1,\infty)\to\mathbb{R}$ of \eqref{eqn:Eh} with $u^\phi_0=\phi$.  The map 
	$	\Phi: [0,\infty)\times C  \ni 
	(t,\phi) \mapsto  u_t^\phi \in C	$ 
	is a continuous semiflow with stationary points $0\in C$, 
	$\hat{\xi}_- \in C$, $\hat{\xi_+} \in C$ where 
	$\hat{\xi}_\pm(s)= \xi_\pm$ for all $ -1\le s\le 0$.  
Define the subset 
$$
C_{\xi_-,\xi_+}=\{\phi\in C: \xi_-<\phi(s)<\xi_+ \text{ for all }s\in[-1,0]\}
$$
of $C$. 
It is easy to see that $C_{\xi_-,\xi_+}$ is positively invariant under the semiflow, that is, 
$$\Phi([0,\infty)\times C_{\xi_-,\xi_+})\subset 
C_{\xi_-,\xi_+}.  
$$	 

	 	Since $h$ is $C^1$-smooth, each map $\Phi(t,\cdot)$ is continuously differentiable.  For $\phi,\psi\in C$,  $D_2\Phi(t,\phi)\psi= v_t$, where $v$ solves the linear variational equation 
\begin{equation}\label{linvar}
v'(s) = -av(s) + bh'\left(u^\phi(s-1)\right)v(s-1)
\end{equation} 
along $u^\phi$	 with $v_0=\psi$.  
In particular, $D_2\Phi(t,0)$ generates a strongly continuous semigroup of bounded linear operators on $C$, and the variational equation along  the stationary point $0$ reads
	 \begin{equation}\label{linvar-equ}
	 	v'(t) = -av(t) + bh'(0)v(t-1).
	 \end{equation}
	
	 The spectrum of the generator associated with $\left(D_2\Phi(t, 0)\right)_{t\ge 0}$ is given by the zeros of the characteristic equation
	 \begin{equation}\label{gen_chareq}
	 	\C\ni\lambda\mapsto \lambda + a - bh'(0)e^{-\lambda} \in\C.
	 \end{equation}	
	 
	 By condition $bh'(0)>a$, there is a unique positive $\lambda_0>0$, and the remaining zeros occur in complex conjugate pairs $\lambda_j,\overline{\lambda_j}$, $j\in\N$,  with 
	 $(2j-1)\pi<\Im \lambda_j < 2j\pi$, 
	 $
	 \lambda_0 > \Re \lambda_1 > \Re \lambda_2 > \ldots$, 
	 and 
	 $ \Re \lambda_j \to -\infty$  as $j\to\infty$. 
	 All zeros are simple, see e.g. \cite{DVLGW}. 
	 
	 The eigenspace corresponding to $\lambda_0$ is one-dimensional, spanned by $\chi_0\in C$, $\chi_0(s) = e^{\lambda_0 s}$, $s\in[-1,0]$.  Denote this one-dimensional subspace by $P=\mathbb{R}\chi_0$, and let $Q$ be the real generalized eigenspace corresponding to the rest of the spectrum.  Then $C=P\oplus Q$. Let $\mathrm{Pr}_P$ and $\mathrm{Pr}_Q$ denote the projections into $P$ and $Q$ along  $Q$ and $P$, respectively. 
	
	For $\phi,\psi\in C$, introduce $\phi \le \psi$ and  $\phi \ll \psi$
if $	\phi(s)\le\psi(s) $,  $ s\in[-1,0]$, and $	\phi(s)<\psi(s) $,  $ s\in[-1,0]$, respectively. We write  $\phi < \psi $ if $	\phi\le\psi $ and $\phi\ne \psi$. 
The relations $\ge$, $\gg$, $>$ are defined analogously. 

By the positive invariance of $C_{\xi_-,\xi_+}$,  \cite{KWW} applies to show that the semiflow 
$\Phi$ is monotone in the sense that
\begin{equation}\label{monotone1}
\begin{array}{rl}
\phi,\psi\in C_{\xi_-,\xi_+},\ \phi\le \psi,\ t\ge 0\ &\text{ imply }\ \Phi(t,\phi)\le \Phi(t,\psi),\\
\phi,\psi\in C_{\xi_-,\xi_+},\ \phi\ll \psi,\ t\ge 0\ &\text{ imply }\ \Phi(t,\phi)\ll \Phi(t,\psi),\\
\phi,\psi\in C_{\xi_-,\xi_+},\ \phi < \psi,\ t\ge 2\ &\text{ imply }\ \Phi(t,\phi)\ll \Phi(t,\psi). 
\end{array}
\end{equation}
Moreover, the derivatives $ D_2\Phi(t,\phi)$ with $\phi\in C_{\xi_-,\xi_+}$ and $t\ge 0$ are monotone, as well. That is,  
\begin{equation}\label{monotone2}
\begin{array}{rl}
\psi\in C,\ 0\le \psi \ t\ge 0\ &\text{ imply }\ 0\le D_2\Phi(t,\phi)\psi, \\
\psi\in C,\ 0\ll\psi \ t\ge 0\ &\text{ imply }\ 0\ll D_2\Phi(t,\phi)\psi,\\
\psi\in C,\ 0<\psi, \ t\ge2 \ &\text{ imply }\ 0\ll D_2\Phi(t,\phi)\psi.
\end{array}
\end{equation}

	Let
	$$
	S  = 
	\bigl\{
	\phi\in C_{\xi_-,\xi_+} 
	: 
	\bigl(u^\phi)^{-1}(0)
	\ \text{is not bounded from above}
	\}. 
	$$
	By \eqref{monotone1}, for any $\phi\in S$, $u^\phi$ contains at least one zero in all intervals $[t-1,t]$ with $t\ge 0$. 	

 Since $S$ is a subset of $C_{\xi_-,\xi_+}$, it is not identical with  the separatrix defined in  \cite{KWW}. 
However, it has some analogous properties.  

\begin{proposition}\label{prop:conv} Suppose $\phi \in S$ and  $\psi \in C_{\xi_-,\xi_+} $.  
\begin{enumerate}
\item[(i)] If $\psi\gg\phi$ then $\Phi(t,\psi)\to \hat{\xi}_+$ as $t\to\infty$.  
		\item[(ii)] If $\psi\ll\phi$ then $\Phi(t,\psi)\to \hat{\xi}_-$ as $t\to\infty$.
\end{enumerate}
\end{proposition}
	
\begin{proof} Suppose $\psi\gg\phi$. 
In $\overline{C_{\xi_-,\xi_+} }$,  the set of stationary points of $\Phi$ 
is $E=\{\hat\xi_-,0,\hat\xi_+\}$. 
Applying a result from \cite{Thieme}, there exist 
$\phi^*,\psi^*\in C$ so that $\phi\ll\phi^*\ll\psi^*\ll\psi$, and 
$\Phi(t,\phi^*)\to E$,  $\Phi(t,\psi^*)\to E$ as $t\to\infty$. 

By \eqref{monotone1}, in order to prove (i) it suffices to show $\Phi(t,\psi^*)\to \hat\xi_+$ as $t\to\infty$. 
If this is not true, then either  $\Phi(t,\psi^*)\to \hat\xi_-$
or $\Phi(t,\psi^*)\to 0$ as $t\to\infty$.  
$\Phi(t,\psi^*)\to \hat\xi_-$, $t\to\infty$, is impossible by \eqref{monotone1} and  $\psi^*\gg\phi\in S$. 
Therefore, $\Phi(t,,\psi^*)\to 0$ as $t\to\infty$. Similarly, one finds 
  $\Phi(t,\phi^*)\to 0$ as $t\to\infty$. 
  
Define   $y=x^{\psi^*}-x^{\phi^*}$, where $x^{\phi^*}, x^{\psi^*}:[-1,\infty)\to\mathbb{R}$ are the solutions of \eqref{eqn:Eh} with initial functions 
$\phi^*,\psi^*$, respectively. 
Then $y$ satisfies $y_t=\psi^*-\phi^*\gg 0$ for all $t\ge 0$, $\lim_{t\to\infty}y(t)=0$, and 
$$
y'(t)=-ay(t)+\beta(t)y(t-1) \text{ for all }t>0
$$
with $\beta(t)=\ds\int_0^1 bh'((1-s)x^{\phi^*}(t-1)+sx^{\psi^*}(t-1)) ds$, $t\ge 0$. 
Clearly,   $\beta(t)>0$ for all $t\ge 0$, and $\beta(t)\to bh'(0)>a$ as 
$t\to\infty$.   

By choosing $\varepsilon\in (0,[bh'(0)-a]e^{-a})$, there is a positive 
real $\mu>0$ so that 
$$
\mu+(a+\varepsilon e^{a})-bh'(0)e^{-\mu}=0.
$$
That is, $\mathbb{R}\ni t\mapsto  e^{\mu t}\in\mathbb{R}$ 
is a solution of the equation 
$z'(t)=-(a+\varepsilon e^{a})z(t)+bh'(0)z(t-1)$. 
 
From $y_t\gg 0$, $t\ge 0$, it follows that $\ds\frac{d}{dt}(e^{at}y(t))
=e^{at}\beta(t)y(t-1)>0$ for all $t>0$, and 
$e^{a(t-1)}y(t-1)< e^{at}y(t)$ for all $t>1$.  
If $T>0$ is so large that $\beta(t)>bh'(0)-\varepsilon$ for all 
$t\ge T$, then for all $t\ge T+1$
$$
y'(t)\ge -ay(t)+[bh'(0)-\varepsilon]y(t-1)
\ge -(a+\varepsilon e^{a})y(t)+bh'(0)y(t-1).
$$

There is a $\delta>0$ such that $\delta e^{\mu t}<y(t)$ for all 
$t\in [T+1,T+2]$. Then the function $v(t)=y(t)-\delta e^{\mu t}$, $t\in[T+1,\infty)$, statisfies $v_{T+2}\gg 0$ and 
$$
v'(t)\ge  -(a+\varepsilon e^{a})v(t)+bh'(0)v(t-1) \text{ for all }t\ge T+2. 
$$
A standard differential inequality argument shows $v(t)>0$ for all 
$t\ge T+1$. 
Consequently, $y(t)>\delta e^{\mu t}$, $t\ge T+1$, a contradiction  
to $\lim_{t\to\infty} y(t)=0$. 

The proof of (ii) is similar. 
\end{proof}	
	 
The next result is analogous to that of Proposition 5.1 in \cite{KWW}. 
	 
	\begin{proposition}\label{prop:P}
		If $v:\mathbb{R}\to\mathbb{R}$ is a solution of \eqref{linvar-equ}, then $ v_0\in P$ is equivalent to 
$$
v(t)\ge 0 \text{ for all } t\le 0 \ \text{ or } \ 
v(t)\le 0 \text{ for all } t\le 0. 
$$ 
	\end{proposition}
	
	\begin{proof}
	If $v_0\in P$ then $v(t)=ce^{\lambda_0 t}$, $t\in\R$, for some $c\in\R$, 
	and the nonnegativity or nonpositivity of $v$ follows. 
	
	To prove the other implication, without loss of generality we may assume  $v(t)\ge 0$ for all $t\le 0$. 
		From \eqref{linvar-equ} we have $v'(t)\ge - a v(t)$ for all $t\le 0$.  Integrating from $t$ to $0$ gives $v(t)\le e^{- a t} v(0)$ for $t\le 0$.  Hence 
		$$
		\|v_t\| \le  e^a e^{- a t} \|v_0\|
		\quad\text{for } t\le 0.
		$$
		
		Choose $\alpha<-a$ so that $\Re \lambda_{k+1}<\alpha<\Re \lambda_k$ for some $k\in\mathbb{N}$.  Let $\tilde{P}_k$ be the real generalized eigenspace for the eigenvalues $\lambda_0,\lambda_1,\overline{\lambda_1},\dots,\lambda_k,\overline{\lambda_k}$, and let $\tilde{Q}_k$ be that for the rest of the spectrum.  Then $C=\tilde{P}_k\oplus\tilde{Q}_k$, and there is a constant $K>1$ such that
		$$
		\|D_2\Phi(t,0) \chi\| \le  K e^{\alpha t} \|\chi\|
		\quad\text{for all } t\ge 0, \ \chi\in\tilde{Q}_k.
		$$
		For each $t\le 0$ and $\tau\le t$, we have 
		\begin{align*}	
		\|\mathrm{Pr}_{\tilde{Q}_k}v_t\|
		 &= 
		\|D_2\Phi(t-\tau,0) \mathrm{Pr}_{\tilde{Q}_k}v_\tau\|
		 \le 
		K e^{\alpha(t-\tau)}\|\mathrm{Pr}_{\tilde{Q}_k}v_\tau\|
		\\
		&\le 
		K e^{\alpha(t-\tau)}\|\mathrm{Pr}_{\tilde{Q}_k}\| \|v_\tau\|
\le 		
	K e^{\alpha(t-\tau)}\|\mathrm{Pr}_{\tilde{Q}_k}\| e^a e^{- a \tau} \|v_0\|
	\\
	&=
	K e^{-(\alpha+a)\tau}e^{\alpha t -a}
	\|\mathrm{Pr}_{\tilde{Q}_k}\|\|v_0\|	
		\end{align*}
since $\|v_\tau\|\le e^a e^{- a \tau} \|v_0\|$. Letting $\tau\to -\infty$ implies $\mathrm{Pr}_{\tilde{Q}_k}(v_t)=0$.  Hence $v_t\in \tilde{P}_k$ for $t\le 0$.  

It follows that $v$ can be written as
$$
v(t)
=  
a_0 e^{\lambda_0 t}
+ 
\sum_{j=1}^k 
a_j e^{\Re(\lambda_j) t}\cos\bigl(\Im(\lambda_j) (t+b_j)\bigr),
\quad
t\le 0,
$$
for some reals $a_0,a_j,b_j$.  Assume $j\in \{1,\dots,k\}$ is maximal so that $a_j\neq 0$. Then as $t\to -\infty$ the dominant term is $a_j e^{\Re(\lambda_j) t}\cos\bigl(\Im(\lambda_j) (t+b_j)$, which changes sign and violates $v(t)\ge0$.  Hence the only possibility is $v(t)=a_0 e^{\lambda_0 t}$ for $t\le 0$.  Thus $v_0\in P$.
\end{proof}
	
	We recall the definition of the 1-dimensional leading unstable set of the origin following  Chapter 4 of \cite{KWW}.   
	Let $\beta>1$ be such that $e^{\Re(\lambda_1)}<\beta<e^{\lambda_0}$.   By Theorem~I.4 from the Appendix of \cite{KWW}, there exist open convex neighborhoods $N_P\subset P$ and $N_Q\subset Q$ of $0$, and a $C^1$-map $w_{\text{loc}}: N_P \to Q$ with range in $N_Q$, and $w_{\text{loc}}(0)=0$,  $Dw_{\text{loc}}(0)=0$.  The local  $\beta$-unstable manifold of the fixed point $0$ of the time-1 map $\Phi(1,\cdot)$ is defined by 
	\begin{align*}
W_\beta\left(N_P+N_Q\right)=&
\{\phi\in N_P+N_Q: \exists  \text{ a trajectory } (\phi^k)_{k=-\infty}^0
\text{ of }
\Phi(1,\cdot)\text{ with }\phi^0=\phi,\\
& \beta^{-k} \phi^k\in N_P+N_Q \ 
 \text{ for all integers } k\le0,\ \text{ and }\lim_{k\to-\infty}\beta^{-k} \phi^k=0\},
	\end{align*}
	and it coincides with the graph
	$$
	W^u_{\text{loc}} (0)
	 = 
	\{
	 \chi + w_{\text{loc}}(\chi)
	:
	\chi\in N_P
	\}.
	$$

	 
	It is easy to show that each $\phi\in W^u_{\text{loc}} (0)$ uniquely determines a solution $u^\phi:\mathbb{R}\to\mathbb{R}$ of \eqref{eqn:Eh} with $\lim_{t\to -\infty}u^\phi(t)=0$. 
	Furthermore, there exists some $t_*\in\mathbb{R}$ such that $u_s^\phi\in W^u_{\text{loc}} (0)$ for all $s\le t_*$.  Define the forward extension of $W^u_{\text{loc}} (0)$ as 
	$$
	W^u(0)
	 = 
	\Phi\left([0,\infty)\times W^u_{\text{loc}} (0)\right).
	$$	
$W^u(0)$ is called the 1-dimensional leading unstable set of $0$. 

Then, for each $\phi\in W^u(0)$, there is a unique solution  $u^\phi:\mathbb{R}\to\mathbb{R}$ of equation \eqref{eqn:Eh} with $u_0^\phi=\phi$, and for this solution  $u_t^\phi\in W^u(0)$ 
for all $t\in\R$, and  $\lim_{t\to -\infty}u^\phi(t)=0$. 	In addition, 
we have $\hat{\xi_-}\ll \phi\ll \hat{\xi_+}$ for all $\phi\in W^u(0)$.
 
	 The next result shows that the structure of $W^u(0)$ 	is simple.  

	\begin{proposition}\label{prop1}
		There exist $\chi,\eta$ in $W^u(0)$ such that 
		\begin{enumerate}
		\item[(i)] the solution $u^\chi:\R\to\R$ is negative, strictly decreasing with negative derivative and $
		\lim_{t\to-\infty} u^\chi(t)= 0$, $\lim_{t\to\infty} u^\chi(t)
		 = 		\xi_-$,
		\item[(ii)] the solution $u^\eta:\R\to\R$ is positive, strictly increasing with positive derivative and $
		\lim_{t\to-\infty} u^\eta(t)= 0$, $\lim_{t\to\infty} u^\eta(t)
		 = 		\xi_+$,
		 \item[(iii)]
	\begin{equation}\label{Wstructure}
		W^u(0)=\{ u_t^\chi : t\in \mathbb{R}\} \cup \{ 0\} \cup \{ u_t^\eta : t\in \mathbb{R}\}.
		\end{equation}	
		\end{enumerate}
	\end{proposition}

	\begin{proof}
		Choose $\varepsilon_1>0$ so small that  
		$$
		\delta \chi_0 + w_{\text{loc}}(\delta \chi_0)
		 \in W^u_{\text{loc}}(0)
		\quad\text{and}\quad
		\bigl\|D w_{\text{loc}}(\delta \chi_0)\bigr\|
		 < 
		\tfrac12 e^{-\lambda_0}\text{ for }|\delta|<\varepsilon_1.
		$$
Let $\delta_1,\delta_2$ be given with $-\varepsilon_1<\delta_1<\delta_2<\varepsilon_1$, 
and set 
		$$
		\eta_1  = \delta_1 \chi_0  +  w_{\text{loc}}(\delta_1 \chi_0),
		\qquad
		\eta_2  = \delta_2 \chi_0  +  w_{\text{loc}}(\delta_2 \chi_0).
		$$		
Then 
$$
\eta_2-\eta_1
=(\delta_2-\delta_1)\chi_0	+
\int_0^1 Dw_{\text{loc}}\left([s\delta_2+(1-s)\delta_1]\chi_0\right)
(\delta_2-\delta_1)\chi_0  ds, 
$$		
and $\eta_1 \ll \eta_2$ follows since the norm of the integral is less 
than $\frac{1}{2}(\delta_2-\delta_1)e^{-\lambda_0}\|\chi_0\|$, and 
$(\delta_2-\delta_1)\min_{s\in[-1,0]}\chi_0(s)\ge (\delta_2-\delta_1)e^{-\lambda_0}$, $\|\chi_0\|=1$. 
In particular, $-\varepsilon_1<\delta_1<0<\delta_2<\varepsilon_1$	
implies $\eta_1 \ll 0 \ll \eta_2$. 

Fix $\delta\in(0,\varepsilon_1)$, set $\eta=\delta\chi_0+w_{\text{loc}}(\delta\chi_0)$, and consider the solution $u:\R\to\R$ of \eqref{eqn:Eh}
with $u_0=\eta$. 

Let $s,t$ be given with $s\ne t$. Clearly, by the facts 
 $0\ll \eta$ and $u_\tau\to 0$ as $\tau\to-\infty$,  
 $u$ cannot be periodic. Therefore, $u_s\ne u_t$.  
For sufficiently large $T>0$, 
$$
u_{s-T},u_{t-T}\in \{\delta \chi_0 + w_{\text{loc}}(\delta \chi_0):|\delta|<\varepsilon_1\}.
$$
Then there exist $\delta_1,\delta_2$ in $(-\varepsilon_1, \varepsilon_1)$
such that $\delta_1\ne \delta_2$ and 
$$
		u_{s-T}  = \delta_1 \chi_0  +  w_{\text{loc}}(\delta_1 \chi_0),
		\qquad
		u_{t-T} = \delta_2 \chi_0  +  w_{\text{loc}}(\delta_2 \chi_0).
		$$	

Then, by the above reasonings, either 
$u_{s-T}\ll u_{t-T}$ or  $u_{t-T}\ll u_{s-T}$. 
Hence, \eqref{monotone1} implies that the same strict ordering is preserved forward in time, either $u_s\ll u_t$ or $u_t\ll u_s$ 
follows. 
Consequently, $\R\ni t\mapsto u(t)\in\R$ is injective. 
The injectivity, $\lim_{t\to-\infty}u(t)=0$, and $u(0)=\eta(0)>0$ combined 
implies that $u$ is strictly increasing. 
By $0\in S$ and $0\ll\eta$, Proposition \ref{prop:conv} gives $\lim_{t\to\infty}u(t)=\xi_+$. 
Statement (i) is shown analogously. 

It is obvious that
$$
\{ u_t^\chi : t\in \mathbb{R}\} \cup \{ 0\} \cup \{ u_t^\eta : t\in \mathbb{R}\} \subset W^u(0).
$$

By the choice of $\chi,\eta$, there exists $\varepsilon_2\in(0,\varepsilon_1)$ so that 
$$
\{\delta \chi_0 + w_{\text{loc}}(\delta \chi_0):|\delta|<\varepsilon_2\}
\subset
\{ u_t^\chi : t\in \mathbb{R}\} \cup \{ 0\} \cup \{ u_t^\eta : t\in \mathbb{R}\}.
$$

If $\psi\in W^u(0)$ then for the unique solution $u^\psi:\R\to\R$ with 
$u_0^\psi=\psi$ there is a $t<0$ such that 
$$ 
u_t^\psi\in  \{\delta \chi_0 + w_{\text{loc}}(\delta \chi_0):|\delta|<\varepsilon_2\}.
$$
Then $u_t^\psi=\delta \chi_0 + w_{\text{loc}}(\delta \chi_0)$ for some 
$\delta\in (-\varepsilon_2,\varepsilon_2)$, and hence 
$\psi$ is a segment of $u^\eta$, or $u^\chi$, or $0$. This completes the proof of (iii).  

As $w_{\text{loc}}$ is $C^1$-smooth and $Dw_{\text{loc}}(0)=0$, for all sufficiently small $\delta>0$, the function $\delta\chi_0+w_{\text{loc}}(\delta\chi_0)\in C$ is $C^1$-smooth with positive derivative, that is,  
$$
\left[\delta\chi_0+w_{\text{loc}}(\delta\chi_0)\right]'(s) 
=
\left[
\delta \lambda_0 \chi_0 +Dw_{\text{loc}}(\delta\chi_0)\left(\delta \lambda_0\chi_0\right) \right](s)>0
$$
for all $s\in[-1,0]$ since $\chi_0'=\lambda_0\chi_0$. 
Then, for all large negative $t$, we find $0\ll (u_t^\eta)'$, and as 
$(u^\eta)'$ satisfies the linear variational equation \eqref{linvar}  along 
along $u^\eta$, \eqref{monotone2} implies  $0\ll (u_\tau^\eta)'$ for all 
$\tau\ge t$. As $t$ was arbitrarily large negative time, the derivative of 
 $u^\eta$ is positive for all $t\in\R$. 
The proof of $(u^\chi)'<0$ is analogous. 
\end{proof}
	
Clearly, we may assume
$$
W^{u}_{\text{loc}}(0)	= W^{u,-}_{\text{loc}}(0)\cup \{0\}\cup W^{u,+}_{\text{loc}}(0)
$$
where
$$
W^{u,+}_{\text{loc}}(0)	=\{u^\eta_t:t<s\},\quad W^{u,-}_{\text{loc}}(0)	=\{u^\chi_t:t<s\}
$$
for some $s\in\R$, and this defines the decomposition \eqref{Wstructure}
with
\begin{align*}
W^{u,+}(0)=\Phi([0,\infty)\times W^{u,+}_{\text{loc}}(0)) & 
=\{u^\eta_t:t\in\R\}, 
\\ 
W^{u,-}(0)=\Phi([0,\infty)\times W^{u,-}_{\text{loc}}(0)) & 
=\{u^\chi_t:t\in\R\}. 
\end{align*}

Let $\kappa^-$ and $\kappa^+$ be fixed such that $ \xi_- < \kappa^- < 0< \kappa^+ < \xi_+$. 
Then there exist unique times $\tau^-$ and $\tau^+$ such that 
$$
u^\eta(\tau^+)  = \kappa^+
\quad\text{and}\quad
u^\chi(\tau^-) = \kappa^-.
$$

Define 
$$
u^+(t)  =  u^\eta\left(t + \tau^+\right),
\quad
u^-(t)  =  u^\chi\left(t + \tau^-\right) \text{ for }t\in\R .
$$
The functions $u^+$ and $u^-$ 
can be characterized by their nonnegative and nonpositive properties, respectively. 
The proof is motivated by Proposition~5.3 in \cite{KWW}.  

\begin{proposition}\label{prop:u=u+} Let $u:\mathbb{R}\to\mathbb{R}$ be a solution of equation \eqref{eqn:Eh}.
	\begin{itemize}
	\item[(i)]
	If  $u(t)\in [0,\kappa^+]$ for all $t\in (-\infty,0]$ and $u(0)=\kappa^+$, then $u(t)=u^+(t)$ for all $t\in \mathbb{R}$.
	\item[(ii)]
	If $u(t)\in [\kappa^-, 0]$ for all $t\in (-\infty,0]$ and $u(0)=\kappa^-$, then $u(t)=u^-(t)$ for all $t\in \mathbb{R}$.
	\end{itemize}
\end{proposition}

\begin{proof} Let $u:\mathbb{R}\to\mathbb{R}$ be a solution of  \eqref{eqn:Eh} with $u(t)\in [0,\kappa^+]$ for all $t\in (-\infty,0]$ and $u(0)=\kappa^+$.
	
	\textit{Step 1.} 
	We first claim that $u(t)\to 0$ as $t\to -\infty$, i.e.\ $\alpha(u)=\{0\}$. For the $\alpha$-limit set we have
\begin{equation}\label{conv-to-zero}
	\alpha(u) \subset \bigl\{\chi\in C : \chi(s)\in [0,\kappa^+] \text{ for all } s\in[-1,0]\bigr\}.
\end{equation}
Suppose that there is some $\phi \in \alpha(u)$ with $0<\phi$.  Then, by the monotonicity of the semiflow, see \eqref{monotone1},  $0\ll\psi=\Phi(2,\phi)$.  Since $\alpha(u)$ is invariant, the segments of the solution $u^\psi$ lie in $\alpha(u)$.  Then
	it follows that $u^\psi(t)\in[0,\kappa^+]$ for all $t\in\mathbb{R}$.  
	On the other hand, by Proposition \ref{prop:conv}, 
	we have $u^\psi(t)\to\xi_+$ as $t\to\infty$, so eventually $u^\psi(t)>\kappa^+$, a contradiction.  
	Hence $u(t)\to 0$ as $t\to -\infty$.
	
	\textit{Step 2.} 
	It is sufficient  to show that $u_t\in W^u_{\mathrm{\text{loc}}}(0)$
for some large negative $t$.   Recall the definition of $W_{\mathrm{\text{loc}}}^u(0)$, along with the choice of $\beta>1$ satisfying $e^{\mathrm{Re}(\lambda_1)}<\beta< e^{\lambda_0}$.  Fix any $\hat{\beta} \in (\beta, e^{\lambda_0})$. By Theorem~I.1 in the Appendix of \cite{KWW}, there exists a norm $|\cdot|$ on $C$ (equivalent to $\|\cdot\|$) such that
	$$
	\bigl|D_2 \Phi(1,0) \phi\bigr| \ge \hat{\beta} \bigl|\phi\bigr|
	\quad
	\text{for all } \phi \in P.
	$$
	
	We claim that
	\begin{equation}\label{1-beta}
		\gamma 
		 = 
		\limsup_{t\to -\infty}
		 \frac{|u_{t-1}|}{|u_t|}
		 < \frac{1}{\beta}.
	\end{equation}
	
	From the positive-feedback property in (H), for $t\le0$ we have $u'(t)\ge - a u(t)$. Integration from $t_1$ to $t_2$ with $-\infty<t_1<t_2\le 0$ gives
	$$
	u(t_1) \le  e^{- a (t_1-t_2)} u(t_2).
	$$
	In particular, if $t_1\in[t-2,t-1]$ and $t_2\in[t-1,t]$ are chosen so that $u(t_1)=\|u_{t-1}\|$ and $u(t_2)=\|u_t\|$, then
	$$
	\|u_{t-1}\| = u(t_1)
	 \le 
	e^{- a (t_1-t_2)} u(t_2)
	 = 
	e^{- a (t_1-t_2)} \|u_t\|
	 \le 
	e^{2a} \|u_t\|.
	$$
	This shows $\gamma<\infty$.  Next, choose a sequence $t_n\to -\infty$ such that $\tfrac{|u_{t_n-1}|}{|u_{t_n}|}\to\gamma$.  
	
	For each $n$, define $z^n:\mathbb{R}\to\mathbb{R}$ by
	$$
	z^n(t)  =  \frac{u(t_n + t)}{|u_{t_n}|} \quad (t\in\R).
	$$
	We have $\bigl|z_0^n\bigr| = 1$, and $z^n(t)\ge 0$ for $t\le 0$. Each $z^n$ solves
	$$
	(z^n)'(t)  =  - a z^n(t) + b^n(t) z^n(t-1), 
	\quad
	t\le 0,
	$$
	where 
	$$
	b^n(t)
	 = 
	\int_{0}^{1} b h' \left(s u(t+t_n-1)\right) ds.
	$$
	
	Since $u_t\to 0$ as $t\to -\infty$, we get $b^n(t)\to b h'(0)$ uniformly on $(-\infty,0]$ as $n\to\infty$.  Also, from $z^n(t)\ge0$ and $b^n(t)>0$ for $t\le 0$, it follows that $(z^n)'(t)\ge - a z^n(t)$; integration yields $z^n(t)\le C e^{- a t}$ for some constant $C>0$.  
	
Applying the Arzelà–Ascoli theorem (see also Lemma~VI.4 in \cite{KWW}), there is a subsequence $(z^{n_k})$ converging to a continuously differentiable function $z:(-\infty,0]\to\mathbb{R}$ satisfying
	$$
	{z}'(t) = - a z(t)  + b  h'(0) z(t-1),
	\quad 
	z(t)\ge0,\quad
	z_0\neq 0,\quad
	|z_0|=1,\quad
	|z_{-1}| = \gamma.
	$$
	
	By  Proposition~\ref{prop:P}, $z_0\in P\setminus\{0\}.$ Then 
	$$
	1  =  |z_0|=
	\bigl|D_2\Phi(1,0) z_{-1}\bigr|
	 \ge 
	\hat{\beta} \bigl|z_{-1}\bigr|
	 = 
	\hat{\beta} \gamma,
	$$
	so $\gamma \le \tfrac{1}{\hat{\beta}}<\tfrac{1}{\beta}$.  Thus \eqref{1-beta} holds.  
	
	\textit{Step 3.}
	From \eqref{1-beta} we finally deduce that $u_t\in W_{\mathrm{\text{loc}}}$ for some large negative $t$.  
	Fix an  $\varepsilon>0$ so that $\psi\in C$, $|\psi|<\varepsilon$ imply $\psi\in N_P + N_Q$. By Step~1 and \eqref{1-beta}, choose $t<0$ such that 
	$$
	|u_s| < \varepsilon
	\quad\text{and}\quad
	\frac{|u_{s-1}|}{|u_s|} < \frac{1}{\beta}
	\quad
	\text{for all }s\le t.
	$$
	Set $\phi^k=u_{t+k}$ for $k\in -\mathbb{N}_0$.  Then 
	$\phi^k\in N_P + N_Q$ for all $k$, and an inductive argument shows that 
	$$
	\beta^{-k} \bigl|\phi^k\bigr|
	 < \varepsilon
	\quad
	\text{for all }k\in -\mathbb{N}_0.
	$$
	
By Theorem~I.4 in \cite{KWW}, $\lim_{k\to-\infty}\beta^{-k}\phi^k=0$, hence $\phi^0=u_t\in W^u_{\text{loc}}(0)$.
	
	The proof of (ii) is analogous. 
\end{proof}

The next result shows a strong stability property of the unstable set $W^u(0)$ under small perturbations of the right hand side of equation \eqref{eqn:Eh}.

\begin{proposition}\label{un-to-u+}
	Suppose $a,b,h$ satisfy \textup{(H)}.  Let $\left(a_n\right)_{n=1}^\infty$ and $\left(b_n\right)_{n=1}^\infty$ be sequences of positive real numbers such that $a_n\to a$ and $b_n\to b$ as $n\to\infty$.  Let $\left(h_n\right)_{n=1}^\infty$ be a sequence of continuously differentiable functions from  $\mathbb{R}$ to $\mathbb{R}$ such that
	\begin{equation}\label{h-hn}
		\|h_n-h\|_{[\kappa^-,\kappa^+]} + \|h_n'-h'\|_{[\kappa^-,\kappa^+]}
		\to 
		0
		\quad
		\text{as }n\to\infty.
	\end{equation}
	\begin{enumerate}
		\item[(i)]
		For each $n\in\mathbb{N}$, let $u^n:(-\infty,0]\to\mathbb{R}$ be a solution of 
		\begin{equation}\label{eqn:Eh_n}
			\tag{$E_{h,n}$}
			u'(t) = - a_n u(t) + b_n h_n\left(u(t-1)\right),
		\end{equation}
		in the interval $(-\infty,0]$ 	satisfying 
		$$
		\kappa^-  \le  u^n(t) \le u^n(0) = \kappa^+
		\quad\text{for all }t\le 0,
		\quad
		\text{and}
		\quad
		\inf_{t\in(-\infty,0]}u^n(t) \to 0
		\quad
		\text{as }n\to\infty.
		$$	
		Then 
		$u^n(t) \to u^+(t)$ as $n\to\infty$
		uniformly on compact subsets of $(-\infty,0]$.  
		\item[(ii)]	
		If the solution $u^n$ of \eqref{eqn:Eh_n} satisfies 
		$	\kappa^-  = u^n(0) \le  u^n(t) \le \kappa^+$ for all $t\le 0$ and $n\in\N$, and
		$\sup_{t\in(-\infty,0]}u^n(t) \to 0$ as $n\to\infty$, 
		then 
		$u^n(t) \to u^-(t)$ as $n\to\infty$
		uniformly on compact subsets of $(-\infty,0]$.  	
	\end{enumerate}	
\end{proposition}

\begin{proof} We show  statement (i), the proof of (ii) is analogous. 

	First, note that there exist positive constants $A,B,K,H>0$ such that for all $n\in\mathbb{N}$
	$$
	a_n < A,
	\quad
	b_n < B,
	\quad
	|u^n(t)| < K\text{ for all } t\le 0,
	\quad
	\sup_{\xi\in[\kappa^-, \kappa^+]}\left(  |h_n(\xi)| +|h_n'(\xi)|\right) < H.
	$$	
	
	From \eqref{eqn:Eh_n} we get

	Hence, by the Arzelà–Ascoli theorem and the diagonalization process, the sequence $(u^n)$ has a  subsequence $(u^{n_k})$ converging to a continuous function $u:(-\infty,0]\to \R$ so that the convergence is uniform 
	on every compact interval in $(-\infty,0]$.
The limit function $u$ satisfies  $u(t)\in[0,\kappa^+]$ for $t\le 0$, and $u(0)=\kappa^+$. 
 
For the functions $u^{n_k}$, the integral equations 
	\begin{equation*}
		u^{n_k}(t)=e^{-a_{n_k}(t-\tau)}u^{n_k}(\tau)+\int_{\tau}^t e^{-a_{n_k}(t-s)}
		b_{n_k}h_{n_k}\left(u^{n_k}(s-1)\right)  ds,
		\quad
		-\infty < \tau < t \le 0,
	\end{equation*}
hold. For fixed $\tau,t$ with $-\infty < \tau < t \le 0$, letting $k\to\infty$ in the integral equation, using the assumptions on 
$\left(a_n\right)$, $\left(b_n\right)$, $\left(h_n\right)$, and that 
 $(u^{n_k})$ converges to  $u$ uniformly on $[\tau,0]$, we obtain 
$$
u(t)=e^{-a(t-\tau)}u(\tau)+\int_{\tau}^t e^{-a(t-s)}
		bh\left(u(s-1)\right)  ds.
$$

As $\tau, t$ were arbitrary it follows that $u$ satisfies equation \eqref{eqn:Eh} for all $t\le 0$.
Then necessarily $u(t)=u^+(t)$, $t\le 0$, by Proposition \ref{prop:u=u+}. 

Considering an arbitrary subsequence of $(u^n)$, the above process gives 
a subsequnce converging to the same $u^+$. Consequently, 
	$u^n(t) \to u^+(t)$ as $n\to\infty$
	uniformly on compact subsets of $(-\infty,0]$.  
	This completes the proof. 
\end{proof}

		Assume that $p: \mathbb{R}\to\mathbb{R}$ is a non-trivial periodic solution of \eqref{eqn:Eh} with minimal period $\omega>0$ and $p(\mathbb{R})\subset (\xi_-,\xi_+)$.  
		The linearized period map $M  = 		D_2 \Phi(\omega, p_0)$ is called the monodromy operator. The operator $M^k$ is compact on $C$ 
		for a   $k\in\N$ with $k \omega > 1$. 
  Consequently, all nonzero points $\lambda$ in the spectrum $\sigma(M)$ 
  belong to the point spectrum, and they are called Floquet multipliers. 
 For $\lambda \in \sigma(M)\setminus\{0\}$, let 
	$$
	G_\lambda  =  \bigcup_{k=0}^\infty \ker \Bigl(M - \lambda I\Bigr)^k
	$$
	be the (complex) generalized eigenspace for $M$.  For any $\kappa>0$, define
	$$
	G_\kappa
	 = \mathrm{Re}
	\bigoplus_{\substack{|\lambda|=\kappa \\ \lambda\in\sigma(M)}} 
	G_\lambda,
	$$
	which is the real part of the spans of the generalized eigenspaces of eigenvalues with absolute value $\kappa$. 
	  As $p$ is non-trivial, $p_0'\ne 0$, we have $Mp_0'=p_0'$, that is 1 is a Floquet multiplier. 
	  
	 By \cite{MPS} or \cite{KWW}, and the  monotonicity hypothesis (e.g.\ $h'>0$ on $(\xi_-,\xi_+)$), there exist a Floquet multiplier $\lambda_u >1 $ and a $\psi^u\in C$ with $0\ll \psi^u$, $M\psi^u=\lambda_u\psi^u$. 
	  In addition, $\dim\left(G_{\lambda_u}\right) = 1$. i.e.\ the generalized eigenspace corresponding to $\lambda_u$ is one-dimensional, and	$\bigl|\lambda\bigr| < \lambda_u$ for all other Floquet multipliers. 
	Define
		$$
		C_{\lambda_u}
		 = 
		\{ c \psi^u : c\in\mathbb{R}\},
		$$
		and let $C_{<\lambda_u}$ denote the realified generalized eigenspace of $M$ associated with the nonempty compact spectral set $\{\lambda\in\sigma(M): |\lambda|<\lambda_u\}$.  Then $C = C_{\lambda_u} \oplus C_{<\lambda_u}$. 
 
		Choose $\beta\in(1,\lambda_u)$ such that 
		$$
		\beta
		 > 
		\max_{\lambda  \in \sigma(M) \setminus \{\lambda_u\}} 
		\bigl|\lambda\bigr|.
		$$
		By Theorem I.1 in \cite{KWW}, there exists a norm $\vert\cdot\vert$ equivalent to $\|\cdot\|$ on $C$ such that
		$$
		\bigl\vert  (M|_{C_{\lambda_u}})^{-1} \bigr\vert 
		 < 
		\frac{1}{\beta},
		\quad
		\text{and}
		\quad
		\bigl\vert  M|_{C_{<\lambda_u}} \bigr\vert 
		 < 
		\beta.
		$$

By Theorem~I.4 of \cite{KWW}, there exist open convex neighborhoods  
	$N_{\lambda_u}$ of $0$  in $C_{\lambda_u}$, $ N_{<\lambda_u} $ of $0$ in 
	$ C_{<\lambda_u}$, and 
	 a $C^1$-map $w^u: N_{\lambda_u} \to C_{<\lambda_u}$ with $w^u(N_{\lambda_u}) \subset N_{<\lambda_u}$, $w^u(0)=0$, and $Dw^u(0)=0$, and the 
	 shifted graph 
	$$
	W^u_{\text{loc}}(p_0,\Phi(\omega,\cdot))
	 = 
	\bigl\{
	p_0 + \chi + w^u(\chi)
	 : 
	\chi \in N_{\lambda_u}
	\bigr\}
	$$
of $w^u$ over $N_{\lambda_u}$	 is equal to the set  
\begin{align*}
W^u_\beta (p_0,\Phi(\omega,\cdot)) = &	\bigl\{
\phi \in p_0 + N_{\lambda_u} + N_{<\lambda_u}
: 
\exists \text{ a trajectory }(\phi_n)_{-\infty}^0 
\text{ of } \Phi(\omega,\cdot)\\
&\text{ with } \phi_0 = \phi,\ 
\beta^{-n} \left(\phi_n - p_0\right)\in N_{\lambda_u} + N_{<\lambda_u} \ 
 \forall n\in -\mathbb{N},\\
&\text{ and}\ 
\beta^{-n} \left(\phi_n - p_0\right)\to 0
\text{ as }n\to -\infty
\bigr\}.
\end{align*}
	
The unstable set of the fixed point $p_0$ of $\Phi(\omega,\cdot)$ is defined by
$$
W^u \left(p_0, \Phi(\omega,\cdot)\right)
 = 
\bigcup_{n\in\mathbb{N}} \Phi\left(n\omega, W^u_{\text{loc}} (p_0,\Phi(\omega,\cdot)) \right).
$$

The leading 2-dimensional unstable set of the periodic orbit
$
\mathcal{O}  =  \{ p_t : t \in \mathbb{R}\}
$
is given by
$$
W^u(\mathcal{O})
 = 
\Phi\left([0,\infty)\times W^u_{\text{loc}} (p_0,\Phi(\omega,\cdot)) \right).
$$

Choose $\delta_0>0$ so that $\{s\psi^u:|s|<\delta_0\}\subset N_{\lambda_u}$, and define 
$$
W^u_{\text{loc},\delta_0}(p_0,\Phi(\omega,\cdot))=\left\lbrace p_0 + s \psi^u + w^u\left(s \psi^u\right):
|s|<\delta_0 \right\rbrace.
$$ 

If $\phi\in W^u_\beta (p_0,\Phi(\omega,\cdot))$ and $(\phi_n)_{-\infty}^0$ 
is the trajectory  of  $\Phi(\omega,\cdot)$ in the definition of  
$W^u_\beta (p_0,\Phi(\omega,\cdot))$, then $\phi_n\in W^u_\beta (p_0,\Phi(\omega,\cdot))\supset W^u_{\text{loc}}(p_0,\Phi(\omega,\cdot))$, and $\phi_n\to p_0$ as $n\to -\infty$. 
Hence,  $\phi_n\in W^u_{\text{loc},\delta_0}$ follows for all large negative integers $n$. 
Then
$$
\bigcup_{n\in\mathbb{N}} \Phi\left(n\omega, W^u_{\text{loc},\delta_0}  \right)
=\bigcup_{n\in\mathbb{N}} \Phi\left(n\omega, W^u_{\text{loc}} (p_0,\Phi(\omega,\cdot)) \right),
$$
and 
$$
W^u(\mathcal{O})
 = 
 \Phi\left([0,\infty)\times W^u_{\text{loc},\delta_0} (p_0,\Phi(\omega,\cdot)) \right).
$$

By repeating the 
argument in the proof of Proposition \ref{prop1}, 
 one can choose   $\delta_0$ so small that 
\begin{equation}\label{order-wloc}
-\delta_0  \le s_1 < s_2 \le \delta_0 \text{ implies }
p_0 + s_1 \psi^u + w^u\left(s_1 \psi^u\right)
 \ll 
p_0 + s_2 \psi^u + w^u\left(s_2 \psi^u\right).
\end{equation}
Moreover, see also Proposition 12.1 in \cite{KWW},  there exist $\delta_1,\delta_2$ in $(0,\delta_0)$ and a strictly increasing function $j:[-\delta_1,\delta_2]\to [-\delta_0,\delta_0]$ such that 
$j(-\delta_1)=-\delta_0$, $j(\delta_2)=\delta_0$, $j(0)=0$, 
$|j(s)|>|s|$ for all $s\in [-\delta_1,\delta_2]\setminus\{0\}$, 
and 
\begin{equation}\label{order-j}
\Phi(\omega,p_0 + s \psi^u + w\left(s \psi^u\right))
=p_0 + j(s)\psi^u + w\left(j(s) \psi^u\right) 
\text{ for all }s\in[-\delta_1,\delta_2].
\end{equation}

Decompose $W^u_{\text{loc},\delta_0}(p_0,\Phi(\omega,\cdot))$ into three parts by
\begin{equation}\label{Wloc-decomp}
W^u_{\text{loc},\delta_0}(p_0,\Phi(\omega,\cdot))=W^{u,-}_{\text{loc},\delta_0}(p_0,\Phi(\omega,\cdot))\cup \mathcal{O} 
\cup W^{u,+}_{\text{loc},\delta_0}(p_0,\Phi(\omega,\cdot))
\end{equation}
where 
\begin{align*}
& W^{u,-}_{\text{loc},\delta_0}(p_0,\Phi(\omega,\cdot))=\left\lbrace p_0 + s \psi^u + w^u\left(s \psi^u\right):
-\delta_0<s<0 \right\rbrace,\\
& W^{u,+}_{\text{loc},\delta_0}(p_0,\Phi(\omega,\cdot))=\left\lbrace p_0 + s \psi^u + w^u\left(s \psi^u\right):
0<s<\delta_0 \right\rbrace.
\end{align*}

From \eqref{order-wloc} one finds $\psi\ll p_0 \ll \phi$ for 
each $\psi\in W^{u,-}_{\text{loc},\delta_0}(p_0,\Phi(\omega,\cdot))$ and $\phi\in W^{u,+}_{\text{loc},\delta_0}(p_0,\Phi(\omega,\cdot))$.
 Then, by using $p_0\in S$ and Proposition \ref{prop:conv},  
 it follows that the set $W^{u,-}_{\text{loc},\delta_0}(p_0,\Phi(\omega,\cdot))\cup W^{u,+}_{\text{loc},\delta_0}(p_0,\Phi(\omega,\cdot))$
is disjoint from $S$, and  for any 
$$\psi\in W^{u,-}_{\text{loc},\delta_0}(p_0,\Phi(\omega,\cdot)),\quad \phi\in W^{u,+}_{\text{loc},\delta_0}(p_0,\Phi(\omega,\cdot))
$$ 
there are unique solutions $u^\psi,u^\phi:\mathbb{R}\to \mathbb{R}$ 
with $u^\psi_0=\psi$, $u^\phi_0=\phi$, and 
$u^\psi_t\to \hat{\xi}_-$, $u^\phi_t\to \hat{\xi}_+$ as $t\to\infty$. 
Moreover, by \eqref{monotone1} and \eqref{order-j}, 
 $$
 u^\psi_t\ll p_t \text{ for all } t\in\mathbb{R} \text{ and }\psi\in W^{u,-}_{\text{loc},\delta_0}(p_0,\Phi(\omega,\cdot)),
 $$
  $$
 u^\phi_t\gg p_t \text{ for all } t\in\mathbb{R} \text{ and }\phi\in W^{u,+}_{\text{loc},\delta_0}(p_0,\Phi(\omega,\cdot)).
 $$

Recall the constants $\kappa^+, \kappa^-$ with $\xi_-<\kappa^-<0<\kappa^+<\xi_+$. 
We can summarize  the obtained results as follows.

\begin{proposition}\label{W(O)} The leading 2-dimensional unstable set $W^u(\mathcal{O})$
of the periodic orbit $\mathcal{O}$ has the decomposition 
$$
W^u(\mathcal{O})
 = 
W^{u,-}(\mathcal{O})\cup  \mathcal{O} \cup W^{u,+}(\mathcal{O})
$$
where 
$$
W^{u,-}(\mathcal{O})
 = 
 \Phi\left([0,\infty)\times W^{u,-}_{\text{loc},\delta_0} (p_0,\Phi(\omega,\cdot)) \right), \quad 
W^{u,+}(\mathcal{O})
 = 
 \Phi\left([0,\infty)\times W^{u,+}_{\text{loc},\delta_0} (p_0,\Phi(\omega,\cdot)) \right).
$$
\begin{itemize}
\item[(i)]
For each $\psi\in W^{u,-}_{\text{loc},\delta_0}(p_0,\Phi(\omega,\cdot)) $ there is a unique solution 
$u^{-,\psi}:\R\to\R$ of equation \eqref{eqn:Eh} so that 
\begin{equation*}
\begin{array}{cc}
u^{-,\psi}(0)=\kappa^-,\ u^{-,\psi}(t)>\kappa^- \text{ for all }t<0,\\
u^{-,\psi}_t\ll p_t \text{ for all }t\in\mathbb{R}, \text{ and }
u^{-,\psi}_t\to \mathcal{O} \text{ as }t\to -\infty, \quad 
u^{-,\psi}_t\to \hat{\xi}_- \text{ as }t\to \infty.
 \end{array}
\end{equation*}
\item[(ii)]
For each $\phi\in W^{u,+}_{\text{loc},\delta_0} (p_0,\Phi(\omega,\cdot))$ there is a unique solution 
$u^{+,\phi}:\R\to\R$ of equation \eqref{eqn:Eh} so that 
\begin{equation*}
\begin{array}{cc}
u^{+,\phi}(0)=\kappa^+,\ u^{+,\psi}(t)<\kappa^+ \text{ for all }t<0,\\
u^{+,\phi}_t\gg p_t \text{ for all }t\in\mathbb{R}, \text{ and }
u^{+,\phi}_t\to \mathcal{O} \text{ as }t\to -\infty, \quad 
u^{+,\phi}_t\to \hat{\xi}_+ \text{ as }t\to \infty.
 \end{array}
\end{equation*}
 \end{itemize}
\end{proposition}

\section{Leading unstable sets for \eqref{eqn:Eg}  and 
\eqref{eqn:Efn}}\label{sec4}

Throughout this section 
 assumptions $(C_g)(a)$ and $(C_{f_n})(a)(b)$ are assumed. 
The nonlinearities $g$ and $f_n$ of equations 
\eqref{eqn:Eg}  and \eqref{eqn:Efn} are monotone on some intervals 
containing $\xi_1$ and $\xi_{1,n}$, respectively. 
Then the results from Section \ref{sec3} can be used to construct 
some local invariant manifolds, and extensions of these local manifolds 
by the corresponding semiflows will give the unstable sets.  

Fix a $\kappa^+\in (0,1-\xi_1)$ and a $\kappa^-\in (-\xi_1,0)$.

\begin{proposition}\label{prop:W-Eg} 
There exist unique solutions $x^-,x^+:\mathbb{R}\to\mathbb{R}$ of equation \eqref{eqn:Eg} with the following properties. 
\begin{itemize}
\item[(i)]
There exists a unique $t_1>0$ such that $x^+(0)=\xi_1+\kappa^+< 1$, $x^+(t_1)=1$, 
$x^+$  strictly increases on  $(-\infty,t_1+1]$,   
$ (x^+)'(t)>0 $  on $(-\infty,t_1+1)$, 
$\lim_{t\to-\infty}x^+(t)=\xi_1$, $x^+(t_2)=1$, 
$x^+(t)>1$ for all $t\in(t_1,t_2)$, where 
$t_2=t_1+1+(1/c)\ln x^+(t_1+1)$, and 
$x^+(t)=x^+(t_1+1)e^{-c(t-t_1-1)}$ for all $t\in [t_1+1,t_2+1]$, 
$x^+(\mathbb{R})\subset(0,d/c)$.  
\item[(ii)]
Function $x^-$ strictly decreases, $x^-(\mathbb{R})=(0,\xi_1)$, $x^-(0)=\kappa^-+\xi_1$, and 
$x^-(t)\to \xi_1$ as $t\to-\infty$, 
$x^-(t)\to 0$  as $t\to\infty$.
\item[(iii)]
The 1-dimensional leading unstable set of $\Gamma$ at $\hat\xi_1$ is given by
$$
W^u(\hat\xi_1)=W^{u,-}(\hat\xi_1)\cup\{\hat\xi_1\}\cup W^{u,+}(\hat\xi_1)
=\{x^-_t:t\in\mathbb{R}\} \cup\{\hat\xi_1\}\cup \{x^+_t:t\in\mathbb{R}\}. 
$$
\end{itemize}
\end{proposition}

\begin{proof}
Define the function $g^e:\R\to\R$ as follows.  
If $ \xi \in [0, 1] $ then $g^e(\xi)=g(\xi)$. 
For $\xi>1$, let  $g^e(\xi) = 2 - \exp[g'(1)(1-\xi)]$, 
and, for $\xi<0$, set $g^e(\xi)=-g^e(-\xi)$.  It is clear, that $g^e$ is a bounded increasing function is $C^1(\mathbb{R})$ so that $g^e(1^+)=g'(1)$.
Let 
\begin{equation}\label{def-h}
h: \ \R\ni\xi \mapsto g^e(\xi + \xi_1)-g(\xi_1)\in\R,
\end{equation}
and $a=c$, $b=d$. It is easy to see that $a,b,h$ satisfy condition (H) with 
$\xi_-=-\xi_1$ and some $\xi_+>1-\xi_1$. 
Therefore, Propositions \ref{prop:conv}, \ref{prop:P}, \ref{prop1} and  
\ref{prop:u=u+} give a 
 unique  solution $u^+:\R\to\R$ of \eqref{eqn:Eh} such that $u^+$ strictly increases 
with $u'>0$, $u(t)\to 0$ as $t\to-\infty$, $u(t)\to\xi_+$ as  
 $t\to\infty$, and $u^+(0)=\kappa^+$. 
Then there is a unique $t_1>0$ with $u^+(t_1)=1-\xi_1$. 

Define $x^+:\R\to\R$ so that $x^+(t)=u^+(t)+\xi_1$ for all $t\le t_1+1$. 
Then $x^+$ is a solution of equation \eqref{eqn:Eg} on $(-\infty,t_1+1]$ 
since $x^+(t-1)\in (0,1)$ for $t<t_1+1$. 
Extend  $x^+$ from the interval $(-\infty,t_1+1]$ to $\R$ such that $x^+$ is a solution of \eqref{eqn:Eg} on $\R$. Clearly, $x^+(t_1+1)>1$. 
If $t_2>t_1+1$ is minimal with the property $x^+(t_2)=1$, then
$$
t_2=t_1+1+\dfrac{1}{c}\ln x^+(t_1+1).
$$
From \eqref{eqn:Eg}, it follows that
$(x^+)'(t)=-cx^+(t)$ for all $t\in (t_1+1,t_2+1)$, and hence 
$x^+(t)=x^+(t_1+1)e^{-c(t-t_1-1)}$  for all  $t\in [t_1+1,t_2+1)$. 
From $g([0,\infty))\subset [0,1]$, it is easy to see that $x^+(t)\in(0,d/c)$ for all $t\in\mathbb{R}$. 
Therefore, the solution $x^+:\R\to\R$  of equation \eqref{eqn:Eg}  satisfies 
the properties stated in (i). 

The proof of (ii) is immediate from Propositions \ref{prop:conv}, \ref{prop:P}, \ref{prop1} and  
\ref{prop:u=u+}. 

The local 1-dimensional leading unstable manifold of $\Gamma$ at 
$\hat\xi_1$ is clearly the local 1-dimensional leading unstable manifold of $\Phi$ at $0$ shifted by $\hat\xi_1$. Hence statement (iii) follows from 
the results of Section \ref{sec3}. 
\end{proof}

An analogous result is valid for $\Phi^n$ at $\hat\xi_{1,n}$. 
\begin{proposition}\label{prop:W-Efn} 
For all sufficiently large $n$, there exist unique solutions $y^{-,n},y^{+,n}:\mathbb{R}\to\mathbb{R}$ of equation \eqref{eqn:Efn} with the following properties. 
\begin{itemize}
\item[(i)]
Function  
$y^{+,n}$  strictly increases on  $(-\infty,0]$,   
$y^{+,n}(0)=\kappa^++\xi_{1,n}$, 
$\lim_{t\to-\infty}y^{+,n}(t)=\xi_{1,n}$, $y^{+,n}(\mathbb{R})\subset (0,2d/c)$.   
\item[(ii)]
Function $y^{-,n}$ strictly decreases on $\mathbb{R}$, $y^{-,n}(\mathbb{R})=(0,\xi_{1,n})$, $y^{-,n}(0)=\kappa^-+\xi_{1,n}$, and 
$y^{-,n}(t)\to \xi_{1,n}$ as $t\to-\infty$, 
$y^{-,n}(t)\to 0$  as $t\to\infty$.
\item[(iii)]
The 1-dimensional leading unstable set of $\Phi^n$ at $\hat\xi_{1,n}$ is given by
$$
W^u(\hat\xi_{1,n})=W^{u,-}(\hat\xi_{1,n})\cup\{\hat\xi_{1,n}\}\cup W^{u,+}(\hat\xi_{1,n})
=\{y^{-,n}_t:t\in\mathbb{R}\} \cup\{\hat\xi_{1,n}\}\cup \{y^{+_,n}_t:t\in\mathbb{R}\}. 
$$
\end{itemize}
\end{proposition}

\begin{proof}
Setting $\eta^+=\frac{1}{2}[x^+(0)+1]$, and applying Proposition \ref{prop:zeros-fn} 
with $\kappa=\eta^+$  it follows that, for all sufficiently large  $n$, 
 the only zeros of 
$\xi\mapsto -a_n\xi +b_n f_n(\xi)$ in the interval $[0,\eta^+]$ are 
$0$ and $\xi_{1,n}$, and $\xi_{1,n}\to \xi_1$, $n\to\infty$.  
In addition,  
\begin{equation}\label{xi-1n}
\frac{1}{2}\xi_1<\xi_{1,n}<\xi_1+\frac{1}{2}[1-x^+(0)], 
\ b_nf_n'(\xi_{1,n})>a_n,\ f_n'(\xi)>0
 \text{ for all }\xi\in(0,\eta^+].
\end{equation}

Let $f_n^e:\R\to\R$ be given by
	$$
	f_n^e(\xi) =
	\begin{cases}
		f_n(\xi), & \text{if } 0 \leq \xi \leq \eta^+, \\
		f_n(\eta^+) \left[ 2 - \exp\left(\dfrac{f_n'(\eta^+)}{f_n(\eta^+)} (\eta^+-\xi)\right) \right], & \text{for } \xi > \eta^+,
	\end{cases}
	$$
and $f_n^e(\xi) =-f_n^e(-\xi)$ for $\xi<0$. 
For large $n$, define 
\begin{equation}\label{def-hn}
h_n:\ \R\ni\xi \mapsto f_n^e(\xi_{1,n}+\xi)-f_n(\xi_{1,n})\in \R.
\end{equation}

Condition (H) of Section \ref{sec3} holds,  for
sufficiently large $n$, with $a=a_n$, $b=b_n$, $h=h_n$ is given in \eqref{def-hn}, $\xi_-=-\xi_{1,n}$, and some $\xi_+$ in $(\eta^+-\xi_{1,n},\infty)$. 
The proof goes analogously to that of Proposition \ref{prop:W-Eg}. 
\end{proof}

The the following statement is a straightforward consequence 
of Proposition \ref{prop:conv} and the monotone properties in 
\eqref{monotone1} if they applied in two situations: 
with $a=c$, $b=d$, $h$ is given in  \eqref{def-h}, $\xi_-=-\xi_1$, 
and some $\xi_+>1-\xi_1$;  moreover  with 
sufficiently large $n$, $a=a_n$, $b=b_n$, $h=h_n$ is given in \eqref{def-hn}, $\xi_-=-\xi_{1,n}$, and some $\xi_+$ in $(\eta^+-\xi_{1,n},\infty)$. 

\begin{proposition}\label{cor-conv}
\begin{itemize}
\item[(i)] If $\varphi\in C^+$ with $0< \varphi < \hat\xi_1$, then 
$0< \Gamma(t,\varphi)< \hat\xi_1$ for all $t\ge 0$, and 
$\Gamma(t,\varphi)\to 0$ as $t\to\infty$. 
\item[(ii)] If $\varphi\in C^+$ with $\|\varphi\|<1$ and $\varphi>\hat\xi_1$, then there is a $t_0>1$ such that $x^\varphi(t_0)=\Gamma(t_0,\varphi)(0)=1$.  
\item[(iii)] If $\varphi\in C^+$ with $0< \varphi < \hat\xi_{1,n}$, and $n$ is sufficiently large, then 
$0< \Phi^n(t,\varphi)< \hat\xi_{1,n}$ for all $t\ge 0$, and 
$ \Phi^n(t,\varphi)\to 0$ as $t\to\infty$. 
\end{itemize}
\end{proposition}

In the rest of this section we assume: 
\begin{enumerate}
\item[$(Y_n)$] 
A sequence of non-negative reals $(\varepsilon_n)_{n=1}^\infty$  
with $\lim_{n\to\infty}\varepsilon_n\to 0$, a sequence $(J_n)_{n=1}^\infty$ of index sets, and for each  $n\in\mathbb{N}$ and $j\in J_n$, 
 solutions $Y^{+,n,j}:\R\to \R$ of equation \eqref{eqn:Efn} 
are given such that 
 $$
 \xi_{1,n}-\varepsilon_n\le Y^{+,n,j}(t)\le Y^{+,n,j}(0)=\kappa^+ + \xi_{1,n} 
 \text{ for all } t<0
 $$ 
 holds. 
 \end{enumerate}
The above condition $(Y_n)$ will be applied in two situations: 
To study the set $W^{u,+}(\hat{\xi}_{1,n})=\{y^{+,n}_t:t\in\mathbb{R}\}$ given in Proposition \ref{prop:W-Efn}, we will choose $\varepsilon_n=0$, 
$J_n=\{0\}$, and 
$Y^{+,n,j}=y^{+,n}$, that is, $Y^{+,n,j}$ is independent of $j\in J_n$. 
In the proof of Theorem \ref{thm:Hopf} when the unstable set  $W^{u}(\mathcal{Q}^n)$ is considered, the index set 
will be a subset of  a local unstable set of a period map, and 
$\varepsilon_n$ will be a bound for $\sup_{t\in\mathbb{R}}|q^n(t)-\xi_{1,n}|$.  
  
In the proof of the next proposition we will apply Proposition \ref{un-to-u+} for solutions $u,u_n$ of equations \eqref{eqn:Eh}, \eqref{eqn:Eh_n} with 
nonlinearities $h,h_n$ defined above by \eqref{def-h}, \eqref{def-hn}.

\begin{proposition}\label{prop:yn-x-on-10}
If $(Y_n)$ is satisfied, then  
	$$
	\sup_{j\in J_n}\|Y_0^{+,n,j}-x^+_0\|\to 0 \text{ as }n\to\infty.
	  $$  
\end{proposition}

\begin{proof}
Suppose that the statement is not true. Then there exist a $\delta>0$, 
a sequence $(n_k)_{k=1}^\infty$ with $n_k\to\infty $ 
as $k\to\infty$, elements $j_k\in J_{n_k}$, and  solutions $Y^{+,n_k,j_k}:\R\to \R $ of \eqref{eqn:Efn} with $n=n_k$ 
satisfying 
$$
\xi_{1,n_k}-\varepsilon_{n_k}\le Y^{+,n_k,j_k}(t)\le Y^{+,n_k,j_k}(0)=\kappa^++\xi_{1,n_k} \quad \text{ for all }t\le 0,
$$ 
and 
$\|Y_0^{+,n_k,j_k}-x^+_0\|\ge\delta  $ for all $k\in\N$. 

Defining   $\kappa^-=-\xi_1/2$, by \eqref{xi-1n},  
for all $\xi\in [\kappa^-,\kappa^+]$ and for all large $n$, we have 
$
\xi+\xi_1\in  [0,\eta^+]$, 
and
$$
\xi+\xi_{1,n}\in \left[ -\frac{\xi_1}{2} + \xi_{1,n},
x^+(0)-\xi_1 + \xi_{1,n}\right]\subset \left[ 0,x^+(0)+\frac{1}{2}[1-x^+(0)]\right]  \subset [0,\eta^+].
$$

Recall $u^+(t)=x^+(t)-\xi_1$, $t\le 0$,  and that $u^+$ is a solution of \eqref{eqn:Eh} on $(-\infty,0]$. 
The functions 
$$
u^{n_k}(t)=Y^{+,n_k,j_k}(t)-\xi_{1,n_k} \quad (t\le 0)
$$
are solutions of equations \eqref{eqn:Eh_n} with $n=n_k$ on the interval 
$(-\infty,0]$, and 
$u^{n_k}(t)\in [-\varepsilon_{n_k},\kappa^+]\subset [\kappa^-,\kappa^+]$, 
$u^{n_k}(t)\le  u^{n_k}(0)=\kappa^+$ hold for all $t\le 0$ and $k\in\N$.  Moreover, by $\xi_{1,n}\to \xi_1$, 
$$
\|u^{n_k}_0-u^+_0\|=\|Y_0^{+,n_k,j_k}-\xi_{n_k}-x^+_0+\xi_1\|
\ge \|Y_0^{+,n_k,j_k}-x^+_0\|-|\xi_{1,n_k}-\xi_1|\ge \delta -\dfrac{1}{2}\delta=
\dfrac{1}{2}\delta
$$
for all sufficiently large $k\in\N$. 

It is sufficient to verify condition \eqref{h-hn} to arrive at a contradiction since  
 Proposition \ref{un-to-u+} shows the 
uniform convergence of $(u^{n_k})$ to $u^+$ on compact subintervals 
of  $(-\infty,0]$, contradicting 
$\|u^{n_k}_0-u^+_0\|\ge \delta/2$ for large $k\in\N$. 

If $\xi\in [\kappa^-,\kappa^+]$,  and $i\in\{0,1\}$, then
\begin{align*}
|h_n^{(i)}(\xi)-h^{(i)}(\xi)|
& = |f_n^{(i)}(\xi_{1,n}+\xi)-g^{(i)}(\xi_{1}+\xi)| \\
& \le |f_n^{(i)}(\xi_{1,n}+\xi)-g^{(i)}(\xi_{1,n}+\xi)| + |g^{(i)}(\xi_{1,n}+\xi)-g^{(i)}(\xi_{1}+\xi)| \to 0 
\end{align*}
as $n\to\infty$ uniformly in $\xi\in [\kappa^-,\kappa^+]$ by condition $(C_{f_n})(b)$ 
with $\kappa=1-\eta^+$, and by the uniform continuity of $g|_{[0,\eta^+]}$, 
 $g'|_{[0,\eta^+]}$ together with $\xi_{1,n}\to \xi_1$. Therefore condition \eqref{h-hn} of Proposition \ref{un-to-u+} holds, and the proof is complete. 
\end{proof}

The proof of the next result is a slight modification of Proposition 2.6 in \cite{BKSZ}. Recall the definition of $t_2$ from Proposition \ref{prop:W-Eg}. 

\begin{proposition}\label{prop:yn-x-on12}
Suppose that $(Y_n)$ is satisfied. 	Fix $m\in\N$ with $m>t_2$ then $x^+(t)<1$ for all $t\in(t_2,m]$
and
$$
\sup_{j\in J_n}\max_{t\in [0,m]}|Y^{+,n,j}(t)-x^+(t)|\to 0 \text{ as }n\to\infty. 
$$    
\end{proposition}

\begin{proof}
The set $\{t\in [-1,m]: x^+(t)=1\}$ consists of two points, $t_1,t_2$. 
Define the set 
$$
\Delta(\delta)=\{t\in[-1,m]: |x^+(t)-1|<\delta\}.
$$
There exist a sufficiently small $\delta_1>0$ and a $\mu>0$ such that, for all $\delta\in (0,\delta_1)$,  $\Delta(\delta)$
is the union of two open intervals around $t_1$ and  
$t_2$, moreover $(x^+)'(t)$ exists for all $t\in\Delta(\delta)$ with  
$$
|(x^+)'(t)|\ge \mu \text{ for all }t\in\Delta(\delta).
$$
Then it is easy to see that, for the sum of the lengths $|\Delta(\delta)|$ of the two open intervals, we have 
$$
|\Delta(\delta)|\le k_1\delta.
$$ 
with  $k_1=4/\mu$. 
Define the constants
$$
k_2=2+\dfrac{d}{c}+(d+\delta_1)(2+2k_1+\|g'\|_{[0,1]}), 
\  \delta_2=\dfrac{\delta_1}{2k_2^m}. 
$$
 
We show that, if 
 $\delta\in (0,\delta_2)$,  and 
$$
 \sup_{j\in J_n}\|x^{+}_0-Y^{+,n,j}_0\|<\delta,
 $$ 
 then 
	\begin{equation}\label{norm:x-y}
		\sup_{j\in J_n}\|x^{+}_k-Y^{+,n,j}_k\|\le \delta k_2^{j} 
	\end{equation}
	for all  $ k\in \{1,2,\dots,m\}$, provided $n$ is sufficiently large. 
This is sufficient  
 by Proposition \ref{prop:yn-x-on-10}. 
We prove by induction. 

Fix $\delta\in (0,\delta_2)$. Then, for all large enough $n$,  
\begin{equation}\label{n-large}
|c-a_n|<\delta,\ |d-b_n|<\delta,\ \|f_n-g\|_{[0,1-\delta]\cup[1+\delta,\infty)}<\delta,
\end{equation}
and, by  Proposition \ref{prop:yn-x-on-10},  
$$
 \sup_{j\in J_n}\|x^{+}_0-Y^{+,n,j}_0\|<\delta.
$$ 

Suppose $k\in\{0,\ldots,m-1\}$ and \eqref{norm:x-y} is satisfied. This is the case if $k=0$. 
It remains to show that  
$\sup_{j\in J_n}\|x^+_{k+1}-Y^{+,n,j}_{k+1}\|\le\delta k_2^{k+1}.$ 

For the sake of notational simplicity, set $x=x^+$ and for some $j\in J_n$, let $y=Y^{+,n,j}$. We will prove 
$\|x^+_{k+1}-y_{k+1}\|\le\delta k_2^{k+1}$ for any fixed $j\in J_n$. 
	
For $t\in [k,k+1]$ we have 
	$$x(t)=e^{-c(t-k)}x(k)+d\int_{k}^{t} e^{-c(t-s)}g(x(s-1))  ds$$ 
	and 
	$$y(t)=e^{-a_n(t-k)}y(k)+b_n\int_{k}^{t}e^{-a_n(t-s)}f_n(y(s-1)) ds.$$ 
Using these integral equations, $g([0,\infty))\subset[0,1]$, 
	the inequality $|e^{-cu}-e^{-a u}|\le |c-a|$ for $0\le u \le 1$, the upper bound $d/c$ for $x^+$,  and \eqref{n-large}, we obtain, for $k\leq t \leq k+1$, that
	\begin{align*}
		\left|x(t)-y(t)\right| 
		& \leq \left|e^{-c(t-k)}-e^{-a_n(t-k)}\right|x(k)+e^{-a_n(t-k)}\left|x(k)-y(k)\right| \\
		& + |d-b_n|\int_{k}^{t} e^{-c(t-s)}g(x(s-1))   ds + b_n\int_{k}^{t} \left|e^{-c(t-s)}-e^{-a_n(t-s)}\right|g(x(s-1))   ds\\
		& + b_n \int_{k}^{t} e^{-a_n(t-s)}\left|g(x(s-1))-f_n(y(s-1))\right|   ds\\
		& \leq |c-a_n|x(k)+ \|x_k-y_k\|+|d-b_n| + (d+\delta)|c-a_n|\\
		& +(d+\delta)\int_{k}^{k+1} \left|g(x(s-1))-f_n(y(s-1))\right|   ds\\
		& \leq \delta\left( \dfrac{d}{c}+1 +k_2^{k}+  (d+\delta)\right) +
		(d+\delta)\int_{k-1}^{k} \left|g(x(s))-f_n(y(s))\right|  ds.
	\end{align*} 
	Let 
	$$\Delta_k(\delta)=\{t\in[k-1,k]:\quad |x(t)-1|<2\delta k_2^{k}\}.$$  
	Then, as $2\delta k_2^{k}<2\delta_2 k_2^{m}\le \delta_1$, 
	 we have $|\Delta_k(\delta)|\leq 2 \delta k_2^{k}k_1$. 
	Hence 
	$$\int_{\Delta_k(\delta)} \left|g(x(s))-f_n(y(s))\right|   ds\leq  
	|\Delta_k(\delta)|\leq 2k_2^{k}k_1\delta$$ 
	since $|g(x(s))-f_n(y(s))|\leq 1$ for all $s\in [k-1,k]$. 
	Let us define 
	$$S_+ = \{t\in[k-1,k]:\quad  x(t)\geq 1 + 2\delta k_2^{k}\},\ 
	S_- = \{t\in[k-1,k]:\quad  x(t)\leq 1-2\delta k_2^{k}\}.$$
	From the inductive hypothesis $\|x_k-y_k\|<\delta k_2^{k}$, it follows that 
	$$
	1+\delta k_2^{k}<y(s)\quad \text{ for all } s \in S_+, \text{ and } 
	 y(s)<1-\delta k_2^{k-1}\quad \text{ for all } s \in S_-.
	$$ 
	For $s\in S_+$ we have
	\begin{align*}
		\left|g(x(s))-f_n(y(s))\right|& = 	f_n(y(s))\leq \|f_n\|_{[1+\delta k_2^{k},\infty)}\leq \|f_n\|_{[1+\delta,\infty)}\\ &=\|g-f_n\|_{[1+\delta,\infty)} < \delta.
	\end{align*}
	If $s\in S_-$ then 
	\begin{align*}
		\left|g(x(s))-f_n(y(s))\right| 
		& \leq \left|g(x(s))-g(y(s))\right|+\left|g(y(s))-f_n(y(s))\right|\\
		& \leq \|g'\|_{[0,1]}|x(s)-y(s)|+ \|g-f_n\|_{[0,1-\delta k_2^{k}]}\\
		& \leq \|g'\|_{[0,1)}\|x_k-y_k\| + \|g-f_n\|_{[0,1-\delta]}\\
		& \leq \left(\|g'\|_{[0,1)}k_2^{k}+1\right)\delta.
	\end{align*}
	Therefore 
	$$\int_{k}^{k+1} \left|g(x(s-1))-f_n(y(s-1))\right|  ds \leq \delta \left(2k_2^{k}k_1+\|g'\|_{[0,1)}k_2^{k}+1\right).$$ 
	It follows that 
	\begin{align*}
		\|x_{k+1}-y_{k+1}\|
		& = \delta\left(k_2^{k}\left[1+(d+\delta)(2k_1+\|g'\|_{[0,1]})\right]+\frac{d}{c}+2(d+\delta)+1\right).
	\end{align*}
	Clearly, $1\le k_2^{k}$, and thus
	$$\frac{d}{c}+2(d+\delta)+1\leq k_2^{k}\left(\frac{d}{c}+2(d+\delta)+1\right).$$
	Consequently, by $\delta<\delta_1$,
	\begin{align*}
		\|x_{k+1}-y_{k+1}\|
		& \leq \delta k_2^{k}\left(2+(d+\delta)\left(2+2k_1+\|g'\|_{[0,1]})+\frac{d}{c}\right)\right) < \delta k_2^{k+1}.
	\end{align*}
As $j\in J_n$ was arbitrary, the proof is complete. 
	\end{proof}

Recall that, by statement (i) of Proposition \ref{prop:W-Eg}, 
 $x^+(t)>1 $ for all $t\in(t_1,t_2)$, and $t_2-t_1=1+(1/c)\ln x^+(t_1+1)>1$. 
Then we can find a $t_3\in (t_1,t_2)$ and an $\varepsilon>0$ with 
$t_3+1<t_2$ and $x^+(t)\ge 1+2\varepsilon$ for all $t\in[t_3,t_3+1]$. 
By choosing  $m\in\mathbb{N}$ as the integer part of $t_2+1$,  Proposition \ref{prop:yn-x-on12} 
implies the following corollary.

\begin{corollary}\label{cor:yn>1+e}
Suppose that $(Y_n)$ is satisfied. Let $\varepsilon>0$ and $t_3\in (t_1,t_2-1)$ be given such that  $x^+(t)\ge 1+2\varepsilon$ for all $t\in[t_3,t_3+1]$.  Then
$$
Y^{+,n,j}(t)\ge 1+\varepsilon \quad \text{ for all }t\in [t_3,t_3+1], 
\text{ and for all }j\in J_n
$$
provided $n$ is sufficiently large. 
\end{corollary}	
		
Assume conditions $(C_g)(a)(b)$ and $(C_{f_n})(a)(b)$. 
Then, for each given $\varepsilon>0$, Proposition \ref{prop:main0-application} states that, for all sufficiently large $n$, 
equation \eqref{eqn:Efn} has a stable periodic orbit $\mathcal{O}^n$  such that the set $\{\psi\in C^+: \psi(s)\ge 1+\varepsilon,\ s\in[-1,0] \}$ belongs 
to the region of attraction of $\mathcal{O}^n$. 		
Corollary \ref{cor:yn>1+e} will guarantee that  
$Y^{+,n,j}_t\to \mathcal{O}^n$ as $t\to\infty$ for all $j\in J_n$, provided 
$n$ is sufficiently large.  

Another corollary  of Proposition \ref{prop:yn-x-on12} is as follows. 
If  $x^+_T\ll\hat{\xi}_1$ for some $T>t_2+1$, then by Proposition \ref{prop:yn-x-on12} with an integer $m\ge T$, it follows that 
$Y^{+,n,j}_T\ll \hat{\xi}_{1,n}$, for all $j\in J_n$, provided $n$ is large enough. Then Proposition \ref{cor-conv}(iii) applies to show $Y^{+,n,j}(t)\to 0$, for all $j\in J_n$, as 
$t\to\infty$.  

\begin{corollary}\label{cor:yntozero}
Suppose that $(Y_n)$ is satisfied and there exists a $T>t_2+1$ such that 
$x^+_T\ll \hat{\xi}_1$. 		Then, for all sufficiently large $n$,  
$$
\sup_{j\in J_n}Y^{+,n,j}(t)\to 0 \quad \text{ as }t\to\infty.
$$
\end{corollary}

\section{Proof of Theorems~\ref{thm:Eg} and \ref{thm:Efn}}\label{sec5}
Throughout this section, we suppose  condition $(C_g)(a)$, and fix 
 $c>0$. 
 For $d>c$, denote by $z^d:[0,\infty)\to \mathbb{R}$ 
the solution of equation \eqref{eqn:Eg} with initial condition 
 $z^{d}(t)=e^{-ct},\ 0\le t\le 1$.
 By Proposition \ref{prop:bounds}, 
$z^{d}(t)\in (0,d/c]$ for every $t\ge 0$.

After a sequence of propositions, we will be in a position to show 
Theorems~\ref{thm:Eg} and \ref{thm:Efn}. 

\begin{proposition}\label{prop:ordering}
	Let $d_{2}>d_{1}>c$ and suppose that, for some
	$\omega\in(1,\infty]$,  
	$z^{d_2}(t)<1$ for every $t\in(1,\omega)$.  
	Then $z^{d_1}(t) < z^{d_2}(t)$ for all $t\in(1,\omega)$.
\end{proposition}

\begin{proof}
	Set $z^{j}=z^{d_{j}}$, $j=1,2$.
	Let $t\in[1,\min\{2,\omega\}]$. The solutions $z^1,z^2$ of 
	\eqref{eqn:Eg} satisfy equation \eqref{inteqn:x} with $T_1=1$, $T_2=\infty$, and  
	$d=d_1$, $d=d_2$, respectively. 
Observe $z^{j}(1)=e^{-c}$ and, for $s\in[1,2]$,  $z^{j}(s-1)=e^{-c(s-1)}$.  Hence
	\[
	z^{j}(t)=e^{-ct}
	+d_{j}\int_{1}^{t}e^{-c(t-s)}
	g \left(e^{-c(s-1)}\right) ds ,
	\qquad
	t\in[1,\min\{2,\omega\}],
	\]
and 
	\[
	z^{2}(t)-z^{1}(t)
	=(d_{2}-d_{1})
	\int_{1}^{t}e^{-c(t-s)}
	g \left(e^{-c(s-1)}\right) ds>0. 
	\]
Therefore, 	 $z^{1}(t)<z^{2}(t)$ for $t\in(1,\min\{2,\omega\}]$, 
and the proof is complete if $\omega\le 2$. 

Assume $\omega>2$, and that there exists a $t_*\in (2,\omega)$
with  $z^{1}(t)<z^{2}(t)$ for $t\in(1,t_{*})$ and
	$z^{1}(t_{*})=z^{2}(t_{*})$. 
As 	$z^{1}(t_{*}-1)<z^{2}(t_{*}-1)<1$, both $z^1,z^2$ are differentiable 
at $t_*$, and 
	\[
	\left(z^{2}-z^{1}\right)'(t_{*})
	=d_{2} g \left(z^{2}(t_{*}-1)\right)
	-d_{1} g \left(z^{1}(t_{*}-1)\right)
	\ge(d_{2}-d_{1}) 
	g \left(z^{1}(t_{*}-1)\right)>0,
	\]
	since $g$ is increasing on $[0,1]$. 
	This contradicts the definition of $t_{*}$ as a first contact
	point, which would require
	$\left(z^{2}-z^{1}\right)'(t_{*})\le 0$.  Therefore such a $t_{*}$
	cannot exist, and $z^{1}(t) < z^{2}(t)$ for all $(1,\omega)$.
\end{proof}

\begin{proposition}\label{prop:zc-below-one}
	Let $z^{c}\colon[0,\infty)\to\mathbb{R}$ be a solution of 
\eqref{eqn:Eg}	with $d=c$, and let
	$z^{c}(t)=e^{-ct}$, for $0\le t\le 1$.
	Then $0<z^{c}(t)<1$ for every $t\ge 1$.
\end{proposition}

\begin{proof}
	Positivity for $t\ge 0$ is immediate from the equation and the
	positive history.  
	Suppose, aiming for a contradiction, that there exists a minimal 
	$t_{0}>1$ with $z^{c}(t_{0})=1$.
	
	As $z^{c}(1)=e^{-c}<1$, such a $t_{0}$ would indeed satisfy
	$t_{0}>1$.  
	At this first–hit time we have $(z^{c})'(t_{0})\ge 0$, while the
	differential equation gives
	\[
	(z^{c})'(t_{0})
	=-c z^{c}(t_{0})+c g \left(z^{c}(t_{0}-1)\right)
	=-c+c g \left(z^{c}(t_{0}-1)\right).
	\]
	As $t_{0}-1\in(0,t_{0})$ and $t_{0}$ was the first time
	$z^{c}$ reached~$1$, we have $z^{c}(t_{0}-1)<1$.  By
	$(C_{g})(a)$  this implies $g \left(z^{c}(t_{0}-1)\right)<1$, hence
	$(z^{c})'(t_{0})<0$, contradicting $(z^{c})'(t_{0})\ge 0$.
\end{proof}

Define
\[
\mathcal{D}=\Bigl\{ d>c: 
z^{d}(t)<1 \text{ for all }t>0,  \text{ and }  
z^{d}(t)\to 0 \text{ as } t\to\infty
\Bigr\}.
\]

\begin{proposition}\label{prop:D-interval}
	There exists a critical parameter value $d^{*}>c$ such that $
	\mathcal{D}=(c,d^{*})$.
\end{proposition}

\begin{proof}
	\textit{Step 1.}  $\mathcal{D}\neq\emptyset$. 
	 
	For every $k\in\mathbb{N}$ put
$	d_k  = c+(1/k)$,  $	z^{k} = z^{d_k}$. 
	Let $\zeta_{k}\in(0,1]$ be the unique solution of 
	$	-c\xi+d_k g(\xi)=0$, that is, $	\frac{g(\xi)}{\xi}=\frac{c}{d_k}$.
	Assumption $(C_g)(a)$ states that $\xi\mapsto g(\xi)/\xi$ is strictly
	increasing on $(0,1]$; hence $\zeta_{k}$ is well–defined and
	$\zeta_{k}\to 1$ as $k\to\infty$.

	Let $t\in[1,2]$. As  	$z^{k}(s-1)=z^c(s-1)=e^{-c(s-1)}$ for  $s\in[1,2]$, 
from \eqref{inteqn:x} 	
$$
		z^{k}(t)-z^c(t)=(d_k-c)  \int_{0}^{t-1} e^{-c(t-1-s)}g\left(e^{-cs}\right) ds.
		$$
	Thus 
	\[
	\bigl|z^{k}(t)-z^{c}(t)\bigr|
	\le \frac1k\int_{0}^{t-1}g(e^{-cs}) ds
	\le \frac1k, 
	\qquad t\in[1,2],
	\]
and $z^{k}\to z^{c}$ as $k\to\infty$, uniformly in $t\in [1,2]$.
	
	As $z^{c}(t)<1$ for $t\ge1$ and $\zeta_{k}\to1$,
	there is a $k_{0}\in\mathbb N$ such that, for each $k\ge k_0$, 
	\[
	z^{c}(t)+\frac1k<\zeta_{k}
	\quad\text{for all }t\in[1,2], 
	\]
and consequently $0\ll z^k_2\ll \hat\zeta_k$ if $k\ge k_0$. 	
Proposition \ref{cor-conv} can be applied to conclude 
$z^k(t)\in (0,1)$ for all $t>0$, and 
	$z^{k}(t)\to0$ as $t\to\infty$, provided $k\ge k_0$. 
	Thus $d_k\in\mathcal D$ whenever $k\ge k_{0}$, so $\mathcal D\neq\emptyset$.

	\textit{Step 2.  $\mathcal D$ is an interval.}
	
Suppose $d_2\in \mathcal{D}$. Then $z^{d_{2}}(t)<1$  for $t>0$, and $z^{d_2}(t)\to 0$ as $t\to\infty$.
  By Proposition~\ref{prop:ordering},   
  $c<d_{1}<d_{2}$ implies $z^{d_{1}}(t)<z^{d_{2}}(t)$ for all $t\in(1,\infty)$, and then   
	$z^{d_{1}}(t) \to 0$ as $t\to\infty$. 
	Hence $d_{1}\in\mathcal D$, and  $\mathcal D$ is an interval.
	
	\textit{Step 3.  $\mathcal D$ is bounded from above.}
	
	From \eqref{inteqn:x},  for every $d>c$, we have
	$$
		z^{d}(2)
	=e^{-c(2-1)}z^{d}(1)+d\int_{1}^{2}e^{-c(2-s)}g\left(z^{d}(s-1)\right)ds
=e^{-2c}+d e^{-c}\int_{0}^{1}e^{cs}g\left(e^{-cs}\right) ds .
	$$
Hence 	$z^{d}(2)\to\infty$  as $d\to\infty$.
So $d\notin\mathcal D$ for all large $d$.
 Therefore $\mathcal D$ is bounded from above.

	\textit{Step 4.  The value $d^{*}=\sup\mathcal D$ does not belong to $\mathcal D$.}
	
	Assume, towards a contradiction, that $d^{*}\in\mathcal D$.
	Then $z^{d^{*}}$ satisfies
	$0< z^{d^*}(t)<1$ for $t>0$, $z^{d^*}(t)\to 0$ as $t\to\infty$.

	For every $\varepsilon\ge0$ let
	\(
	\zeta_{\varepsilon}\in(0,1]
	\)
	be the unique solution of
	\(
	-c\xi+(d^{*}+\varepsilon)g(\xi)=0
	\). 
The map $\varepsilon\mapsto\zeta_{\varepsilon}$ is continuous and
	$\zeta_{0}\in (0,1)$. 
Then we can choose $\varepsilon_{0} > 0$ such that 
$\zeta_\varepsilon> (1/2)\zeta_0$ for all $\varepsilon\in(0,\varepsilon_{0}]$. 

By $0< z^{d^*}(t)<1$ for $t>0$ and  $\lim_{t\to\infty}z^{d^*}(t)=0$, 
there exists a $T>1$ such that 
$$
z^{d^*}_T\ll \frac{1}{4}\hat\zeta_0 .
$$

By the integral equation \eqref{inteqn:x} for $z^{d^*+\varepsilon}$, it is easy 
to show that
$$
		\sup_{t\in[0,T]}\bigl|z^{d^*+\varepsilon}(t)-z^{d^*}(t)\bigr|
		\to 0 \text{ as }\varepsilon\to 0.
$$
Hence, and by the choice of $\varepsilon_0$, there exists an $\varepsilon_1\in (0,\varepsilon_0)$ so that 
$$
z^{d^*+\varepsilon_1}(t)\in (0,1) \text{ for all }t\in (0,T], 
\text{ and } z^{d^*+\varepsilon_1}_T\ll \hat\zeta_{\varepsilon_1}. 
$$
An application of Proposition \ref{cor-conv} yields  
$z^{d^*+\varepsilon_1}(t)\in (0,1)$ for all $t>0$, and 
$z^{d^*+\varepsilon_1}(t)\to 0$ as $t\to\infty$. 
	That is,  
	$d^{*}+\varepsilon_{1}\in\mathcal{D}$, a 
	contradiction to the definition of $d^{*}=\sup\mathcal D$.
	
	Thus $d^{*}\notin\mathcal D$, and together with Steps 1--3 we have proved $d^{*}>c$ and 
$	\mathcal D=(c,d^{*})$,  completing the proof.
\end{proof}

\begin{proposition}\label{prop:first-hit-v0}
	For each $d_0>d^{*}$  there exists a time
	$\tau_{0}=\tau_{0}(d_0)\in(1,\infty)$
	such that
	$v^{d_0}(t)<1$ for $0< t<\tau_{0}$, and $v^{d_0}(\tau_{0})=1$.
\end{proposition}

\begin{proof}
	Suppose that 
	$v^{d_0}(t)<1$ for every $t>0$. As $d_0\notin\mathcal{D}$, 
 $v^{d_0}(t)\not\to 0$ as $t\to\infty$.  
Then, by Proposition \ref{cor-conv}, for any $t\ge 1$, we cannot have $v^{d_0}_t>\hat\xi_1$ or $v^{d_0}_t<\hat\xi_1$. Consequently, 
$v^{d_0}(t)-\xi_1$ oscillates around 0 on $[0,\infty)$. 
Recalling the definition of $S$ from Section \ref{sec3} for 
equation \eqref{eqn:Eh} with $a=c$, $b=d_0$, $\xi_-=-\xi_1(d_0)$, and $h$ given in 
\eqref{def-h}, one finds   
 $v^{d_0}_t-\hat\xi_1\in S$ for all $t\ge 1$. 

For a $d_1\in(d^*,d_0)$, by Proposition \ref{prop:ordering}, we have $v^{d_1}(t)<v^{d_0}(t)<1$ for all $t>1$.   
Hence $0\ll v^{d_1}_t\ll v^{d_0}_t$ for all $t>2$. 

Fix a $T>2$ and choose a $\psi\in C^+$ so that 
$v^{d_1}_T\ll \psi \ll v^{d_0}_T$. 

Let $z:[T-1,\infty)\to\mathbb{R}$ be the solution of 
\eqref{eqn:Eg} on $[T-1,\infty)$ with $d=d_0$ and with initial value 
$z_T=\psi$. 
From $z_T=\psi\ll v^{d_0}_T$, by the monotonicity of the 
semiflow $\Gamma$ restricted to $[-1,\infty)\times \{\varphi\in C: 
\varphi(s)\in (0,1),\ s\in[-1,0]\}$, it follows that 
$z(t)<v^{d_0}(t)$ for all $t\ge T-1$. Moreover, 
as $z_T-\hat\xi_1\ll v^{d_0}_T-\hat\xi_1\in S $, Proposition \ref{prop:conv} 
can be applied to conclude that $z(t)-\xi_1\to -\xi_1$, that is, 
$z(t)\to 0$ as $t\to\infty$. 

In order to compare $v^{d_1}$ and $z$ on $[T-1,\infty)$, observe that 
$v^{d_1}_T\ll \psi=z_T $, and, for all $t>T$, 
$$
\left(z-v^{d_1}\right)'(t)=-c(z(t)-v^{d_1}(t))+d_0g(z(t-1))-d_1g(v^{d_1}(t-1)).
$$
Hence, for a minimal $t_0>T$ with $z(t_0)=v^{d_1}(t_0)$, 
by using $d_0 g(z(t_0-1))> d_0 g(v^{d_1}(t_0-1))> d_1g(v^{d_1}(t_0-1))$, 
it follows that $(z-v^{d_1})'(t_0)>0$, a contradiction to the minimality of $t_0$. 
Then, $v^{d_1}(t)<z(t)$ for all $t>T-1$, and consequently $v^{d_1}(t)\to 0$ 
as $t\to\infty$. Hence, $d_1\in \mathcal{D}$ follows, a contradiction 
to $d_1>d^*$ and Proposition \ref{prop:D-interval}. 
\end{proof}

Now we define two special solutions  of 
equation \eqref{eqn:Eg}. 

For a $d_0>d^*$, let $w^0=v^{d_0}:[0,\infty)\to \mathbb{R}$ be the solution of \eqref{eqn:Eg} with $d=d_0$ and with initial function $w^0(t)=e^{-ct}$, 
$t\in[0,1]$. By Proposition \ref{prop:first-hit-v0}, there exists a 
$\tau_0=\tau_0(d_0)>1$ such that $w^0(t)<1$ for all $t\in(0,\tau_0)$, 
and   $w^{0}(\tau_{0})=1$. 

For a $d_1>d^*$, define 
\[
	\tau_{1}=1+\dfrac1c
	\ln \left(\frac{d_1}{c}(1-e^{-c})+e^{-c}\right)
	= 1+\dfrac1c
	\ln \left(\tfrac{d_1}{c}-e^{-c} \left(\tfrac{d_1}{c}-1\right)\right), 
	\]
and let $w^1:[0,\infty)\to \mathbb{R}$ be the solution of \eqref{eqn:Eg} with $d=d_1$ and with initial function 
$$
w^{1}(t)=
		\frac{d_1}{c}\left(1-e^{-ct}\right)+e^{-ct} \text{ for } 0\le t\le 1.		
$$
It is a straightforward calculation to show that $w^1(t)>1$ for $t\in(0,1]$, 
and hence
$$
	w^{1}(t)=
		e^{-c(t-1)} \left[\frac{d_1}{c}\left(1-e^{-c}\right)+e^{-c}\right] 
		\text{ for } 1<t\le\tau_{1}.
$$		
Then 
$$
w^0(0)=1, \ w^1(t)>1 \text{ for all } t\in (0,\tau_1), \ w^1(\tau_1)=1.
$$

Recall the definition of the solution $x^+:\mathbb{R}\to \mathbb{R}$
of equation \eqref{eqn:Eg} in Proposition \ref{prop:W-Eg}.  
By using the auxiliary functions $w^0,w^1$, the next proposition will give  
 information on $x^+$ on the interval $[t_2+1,\infty)$. 


\begin{proposition}\label{prop:envelope}
Let $d>d^{*}$ be fixed, and choose  $d_{0}\in(d^{*},d)$, 
$d_1=d$. 
	\begin{itemize}
	\item[(i)]
If $t_{*}\in\mathbb{R}$ and $\omega>0$ are given such that 
		$x^{+}(t_{*})=1$, 
		$x^{+}(t)<1$ for $t_{*}<t<t_{*}+\omega$, 
		then  $\omega\le\tau_{0}$ and
		$x^{+}(t_{*}+t) \ge w^{0}(t)$		for all $t\in[0,\omega]$.
		\item[(ii)]
If $t_{*}\in\mathbb{R}$ and $\omega>0$ are given such that 
		$x^{+}(t_{*})=1$, 
		$x^{+}(t)>1$ for $t_{*}<t<t_{*}+\omega$, 
		then  $\omega\le\tau_{1}$ and
		$x^{+}(t_{*}+t) \le w^{1}(t)$		for all $t\in[0,\omega]$.	
\end{itemize}
\end{proposition}

\begin{proof}
	Define 
	$x(t)=x^{+}(t_{*}+t)$, $t\in\mathbb{R}$.
	Then $x$ satisfies \eqref{eqn:Eg}, and 
	$x(0)=x^{+}(t_{*})=1$, $x(t)<1$ for $0<t<\omega$.
	
If $t\in [0,1]$, then $w^0(t)=e^{-ct}$. From \eqref{inteqn:x} 
for $x$, by $(C_g)(a)$ and $x(0)=1$, we find
$$
x(t)=e^{-ct}1 + d\int_{0}^{t} e^{-c(t-s)} 
g(x(s-1))  ds	\ge e^{-ct}=w^0(t).
$$
In particular, in case $\omega\le 1$, the statement is true. 

If $t\in (1,\min\{2,\omega,\tau_0\})$, then the integral equation  
\eqref{inteqn:x} 
with $T_0=1$,  $x(1)\ge e^{-c}$, $d>d_0$, $g(x(s-1))\ge g(w^0(s-1))$ for $s\in [1,2]$ yield
$x(t)>w^0(t).$ 

If $\omega>1$ and $x(t)>w^0(t)$ is not satisfied for all 
$t\in (1,\min\{\omega,\tau_{0}\})$, then  there exists a $T>1$ with 
$T<\min\{\omega,\tau_{0}\}$ such that 
$x(t)>w^0(t)$ for all $t\in (1,T)$ and $x(T)=w^0(T)$. 

By $x(1)\ge e^{-c}$, $d>d_0$, and $g(x(s-1))\ge g(w^0(s-1))$ for $s\in [1,T],$ and the integral equation \eqref{inteqn:x}  applied for $x$ and $w^0$, imply 
$x(T)>w^0(T),$ a contradiction. 

Consequently,  $x(t)\ge w^0(t)$ holds for all 
$t\in [0,\min\{\omega,\tau_{0}\}]$. 
Hence, in particular  $\omega\le\tau_0$ follows. 
This completes the proof of (i).  Part (ii) is shown analogously. 
\end{proof}

\begin{corollary}\label{cor:band}
	If $d>d_0>d^*$, and  $m_{0}=\min\{\xi_1, \min_{t\in[0,\tau_{0}]}w^{0}(t)\}$, 
	$m_{1}=d/c$, then $x^+(t)\in[m_0,m_1]$
	for all $t\in\mathbb{R}$.
\end{corollary}

	Suppose condition $(C_{f_n})(a)(b)$ in addition to $(C_g)(a)$. 
	Fix $d>d_{0}>d^{*}$. 
	Define $k_{1}=4+2d\|g'\|_{[0,1]}$, and choose $\Delta\in(0,1)$ so that  $c\Delta<\tfrac12 g(e^{-c\Delta})$, and set
	$k_{2}=2+\Delta k_{1}$.
There exists a $\delta\in(0,1)$ with $\delta<m_{0}/4$ such that
	\begin{enumerate}
		\item[(1)] $\displaystyle \delta<\frac{1}{k_{1}}(d-d_{0}) g \left(\frac{m_{0}}{2}\right)$,
		\item[(2)] $\displaystyle \delta<\frac{d-d_{0}}{2k_{1}} g(e^{-c\Delta})$,
		\item[(3)] $\displaystyle 
		\delta<\frac{d-d_{0}}{k_{2}} e^{-c(1+\Delta)}
		\int_{1}^{1+\Delta} e^{cs} g\left(e^{-c(s-1)}\right) ds,$
		\item[(4)] $1+\delta<\dfrac{d}{c}$ and $\delta<\dfrac{c}{2}$.
	\end{enumerate}	
	Moreover, by condition $(C_{f_n})(a)(b)$, there exists an $n_{0}\in\mathbb{N}$ such that for all $n\ge n_{0}$,
	\begin{enumerate}
		\item[(5)] $|a_{n}-c|<\delta$, $|b_{n}-d|<\delta$, and $d|a_{n}-c|<\delta$,
		\item[(6)] $b_{n} \|f_{n}\|_{[1+\delta,\infty)}<\delta$,
		\item[(7)] $d \|f_{n}-g\|_{[0,1-\delta]}<\delta$,
		\item[(8)] $1+\delta<\dfrac{b_{n}}{a_{n}}<2 \dfrac{d}{c}$,
		\item[(9)] $a_{n}>\dfrac{c}{2}$.
	\end{enumerate}
Notice that $\delta$ can be chosen arbitrarily small.


Recall the solutions $y^{+,n}$, $y^{-,n}$ of \eqref{eqn:Efn} 
given in Proposition \ref{prop:W-Efn}.

\begin{proposition}\label{prop:interval-exit}
Suppose conditions $(C_g)(a)$ and $(C_{f_n})(a)(b)$, and let 
$d>d_{0}>d^{*}$. Choose
 $\delta>0$ and $n_0\in\mathbb{N}$ so that inequalities 
 (1)--(9) above are satisfied. Let $n\ge n_0$.
	\begin{itemize}
\item[(i)] If $I\subset\mathbb{R}$ is a maximal open interval such that 
$t_{*}=\inf I\ge t_{2}$,  $y^{+,n}(t_{*})=1-\delta$, and 
$y^{+,n}(t)<1-\delta$ for all $t\in I$, then $t_{**}=\sup I<t_{*}+\tau_{0}$, $y^{+,n}(t_{**})=1-\delta$, and 
$y^{+,n}(t)>\dfrac{m_{0}}{2}$ for every $t\in I$.
\item[(ii)] There exists a constant $\nu_{1}>1$  such that
if $I\subset\mathbb{R}$ is a maximal open interval with 
$t_{*}=\inf I\ge t_{2}$,
$y^{+,n}(t_{*})=1+\delta$,
$y^{+,n}(t)>1+\delta$ for all $t\in I$,
then $t_{**}=\sup I<t_{*}+\nu_{1}$, $y^{+,n}(t_{**})=1+\delta$, and 
$y^{+,n}(t)<2\frac{d}{c}$ for every $t\in I$.
\end{itemize}
 
\end{proposition}

\begin{proof}{\it Proof of (i).}
	Put $w=w^0, y(t)=y^{+,n}(t_*+t)$ for $t\ge0$.
	Then $y(0)=1-\delta$, $y(t)<1-\delta$ for $t\in(0,t_{**}-t_*)$, and $w(0)=1$,
	while $w(t)<1$ for $t\in(0,\tau_0)$ and $w(\tau_0)=1$.  
	
	From the integral 
	equation \eqref{inteqn:y},  
	$$y(t)=e^{-a_nt}(1-\delta)+b_n\int_0^{t}e^{-a_n(t-s)}f_n\left(y(s-1)\right) ds. 
	$$
Hence, for  $t\in[0,1]$,	by $ w(t)=e^{-ct}$,
	\[
	y(t)-w(t)\ge e^{-a_nt}(1-\delta)-e^{-ct}
	\ge e^{-a_nt}-e^{-ct}-\delta e^{-a_nt}.
	\]	
We have 
	$|e^{-a_n t}-e^{-ct}|\le |a_n-c| t\le|a_n-c|<\delta$ Then 
	\begin{equation}\label{eq:step1}
		y(t)-w(t)\ge -|a_n-c|-\delta >-2\delta,\qquad 0\le t\le1.
	\end{equation}	
For $t\in[1,2]$, by \eqref{inteqn:y} and \eqref{inteqn:x}, 
	\begin{align*}
		y(t)&=e^{-a_n(t-1)}y(1)+b_n\!\int_{1}^{t}\! e^{-a_n(t-s)}f_n\left(y(s-1)\right) ds,\\
		w(t)&=e^{-c(t-1)}w(1)+d_0\int_{1}^{t}\! e^{-c(t-s)}g\left(w(s-1)\right) ds.
	\end{align*}
From $|e^{-a_nu}-e^{-cu}|\le |a_n-c|u$, $u\ge 0$, $w(1)=e^{-c}$, and 
by (5) we obtain
$$
e^{-a_n(t-1)}(y(1)-w(1))+\left(e^{-a_n(t-1)}-e^{-c(t-1)}\right)w(1)> -2\delta - \delta(t-1). 
$$ 
By the triangle inequality and the choices of $\delta, n$,
\begin{align*}
& b_n\int_{1}^{t}e^{-a_n(t-s)}f_n(y(s-1))ds-d_0\int_{1}^{t}e^{-c(t-s)}g(w(s-1))ds\\
=& (b_n-d)\int_{1}^{t}(e^{-a_n(t-s)})f_n(y(s-1))ds +d\int_{1}^{t} (e^{-a_n(t-s)}-e^{-c(t-s)}) f_n(y(s-1))ds\\
& + d\int_{1}^{t} e^{-c(t-s)} (f_n(y(s-1))-g(y(s-1))) ds\\
&+\  d\int_{1}^{t} e^{-c(t-s)}(g(y(s-1))-g(w(s-1))) ds + (d-d_0)\int_1^t e^{-c(t-s)}g(w(s-1)) ds\\
\geq & -\left[|b_n-d|+d|a_n-c|+d\|f_n-g\|_{[0,1-\delta]}+d\|g'\|_{[0,1]}2\delta\right](t-1)\\
&+(d-d_0)\int_1^t e^{-c(t-s)}g(w(s-1)) ds\\
> &-\left[3+2d\|g'\|_{[0,1]}\right]\delta (t-1) + (d-d_0)\int_1^t e^{-c(t-s)}g(w(s-1)) ds.
\end{align*} 
Consequently, by $w(s-1)=e^{-c(s-1)}$, $s\in[1,2]$,
$$
	y(t)-w(t)\ge A(t)\  \text{ for all } t\in[1,2],
$$
where
$$
A(t)= -\delta\left[ 2+k_1(t-1)\right] 
	 +(d-d_0)\int_{1}^{t}e^{-c(t-s)}g\left(e^{-c(s-1)}\right) ds, \qquad t\in[1,2].
	 $$
The derivative of $A$ is 
	\[
	A'(t)=-\delta k_1
	+(d-d_0)\Bigl[-ce^{-ct}\int_{1}^{t}e^{cs}g(e^{-c(s-1)}) ds+g\left(e^{-c(t-1)}\right)\Bigr].
	\]
If $t\in [1,1+\Delta]$ then by the  choice of $\Delta$, we obtain
	\[
	A'(t)\ge -\delta k_1 -c(d-d_0)\Delta +(d-d_0)g(e^{-c\Delta})
	\ge -\delta k_1 +\tfrac12(d-d_0)g(e^{-c\Delta})>0
	\]
	by (2). Thus, $A$ is increasing on $[1,1+\Delta]$, and
	\[
	y(1+\Delta)-w(1+\Delta)\ge A(1+\Delta)
	>-\delta k_2 
	+(d-d_0)e^{-c(1+\Delta)}\!\int_{1}^{1+\Delta}\!e^{cs}g(e^{-c(s-1)}) ds>0
	\]
	by (3). Combining with \eqref{eq:step1} we conclude
	\begin{equation}\label{eq:pre-touch}
		y(t)> w(t)-2\delta\quad\text{for }t\in[0,1+\Delta],\qquad
		y(1+\Delta)>w(1+\Delta).
	\end{equation}
From $\delta<m_0/4$ and $w(t)\ge m_0$ on $[0,\tau_0]$, we also get
	\begin{equation}\label{eq:lower-m0}
		y(t)> w(t)-2\delta\ge m_0-2\delta \ge  \frac{m_0}{2},
		\qquad 0\le t\le 1+\Delta.
	\end{equation}
	
We claim that 
$$
y(t)>w(t) \text{ for all } t\in [1+\Delta, t_{**}-t_*).
$$
Assume that there exists a  $t_0\in (1+\Delta,\min\{\tau_0, t_{**}-t_*\})$ 
such that  $y(t)>w(t)$ on $(1+\Delta,t_0)$, and $y(t_0)=w(t_0)$. 
	Then $y'(t_0)\le v'(t_0)$.
On the other hand, 
	\begin{align*}
		y'(t_0)-w'(t_0)
		=&-(a_n-c)w(t_0)+(b_n-d)f_n\left(y(t_0-1)\right)
		+d\bigl[f_n\left(y(t_0-1)\right)-g\left(y(t_0-1)\right)\bigr]\\
		& +d\bigl[g\left(y(t_0-1)\right)-g\left(w(t_0-1)\right)\bigr]
		+(d-d_0)g\left(w(t_0-1)\right).
	\end{align*}
By \eqref{eq:lower-m0} we have
	$y(t_0-1)\ge w(t_0-1)-2\delta> m_0/2$.  
	Using the choices of $\delta $ and $n$, we obtain
	\[
	y'(t_0)-w'(t_0)>
	-3\delta-2d\|g'\|_{[0,1]}\delta
	+(d-d_0)g\!\left(\frac{m_0}{2}\right)
	> -\delta k_1+(d-d_0)g\!\left(\frac{m_0}{2}\right)>0,
	\]
	again by (1). This is a contradiction. Therefore,
	\begin{equation}\label{eq:strict-ineq}
		y(t)>w(t)\qquad\text{for all }t\in[1+\Delta,\min\{\tau_0, t_{**}-t_*\}).
	\end{equation}
Hence it is clear that $t_{**}-t_*<\tau_0$, and the claim holds. 
	Finally, combining \eqref{eq:lower-m0},  \eqref{eq:strict-ineq} and 
$w(t)\ge m_0$ for $t\in[0,\tau_0]$,  	it follows that 
	$y^{+,n}(t)\ge m_0/2$ for every $t\in I$. This proves (i).

{\it Proof of (ii).}
	Let $I$ be as in the statement and write $y(t)=y^{+,n}(t)$ for simplicity.  
	For $t\in[t_{*},t_{*}+1]\subset I$ equation \eqref{inteqn:y} gives
	\begin{align*}
	y(t) &= e^{-a_{n}(t-t_{*})}y(t_{*})
+b_{n}\int_{t_{*}}^{t}e^{-a_{n}(t-s)}f_{n}(y(s-1)) ds\\
 &\le e^{-a_{n}(t-t_{*})}(1+\delta)
+\frac{b_{n}}{a_{n}}\left(1-e^{-a_{n}(t-t_{*})}\right).
	\end{align*} 
	Rearranging,
	\[
	y(t)\le \Bigl(1+\delta-\frac{b_{n}}{a_{n}}\Bigr)e^{-a_{n}(t-t_{*})}+\frac{b_{n}}{a_{n}}
	<\frac{b_{n}}{a_{n}}<2\frac{d}{c},
	\]
	by property (8).  Thus $y(t)$ is bounded by $2d/c$ on $[t_{*},t_{*}+1]$.
	
	If $t_{*}+1\in I$ and $t\in(t_{*}+1,t_{**})\subset I$, then $y(t-1)\ge 1+\delta$, and hence	
\begin{align*}
y'(t)&=-a_{n}y(t)+b_{n}f_{n}(y(t-1))
\le -a_{n}y(t)+b_{n}\|f_{n}\|_{[1+\delta,\infty)}\\
&\le -\frac{c}{2}y(t)+\delta
\le -\frac{c}{2}(1+\delta)+\delta<-\frac{c}{2} \delta<0,
\end{align*}	
	where we used properties (4), (6), (9).
	Thus $y$ is strictly decreasing on $(t_{*}+1,t_{**})$, and it must reach the level $1+\delta$ in finite time. 
Indeed,	 from $y'(t)\le-\frac{c}{2}\delta$ on $(t_{*}+1,t_{**})$ and $y(t_{*}+1)< 2d/c$, we get
	\[
	t_{**}-t_{*} \le 1+\frac{2}{c\delta}\Bigl(y(t_{*}+1)-(1+\delta)\Bigr)
	 \le 1+\frac{2}{c\delta}\Bigl(\tfrac{2d}{c}-1\Bigr)
	=\nu_{1}>1.
	\]
	Hence $t_{**}<t_{*}+\nu_{1}$ and $y(t_{**})=1+\delta$ by the maximality of $I$.
\end{proof}

  Now we are in a position to prove  Theorems~\ref{thm:Eg} and 
  \ref{thm:Efn}. 

\begin{proof}[Proof of Theorem~\ref{thm:Eg}]
Statement (A) follows from Proposition \ref{prop:W-Eg}(ii).

By Proposition \ref{prop:W-Eg}(i), $x^+_{t_2+1}=z_1^d$, where $z^d$ was 
introduced at the beginning of this Section. Propositions  \ref{prop:W-Eg}, \ref{prop:D-interval} and 
the definition 
of $\mathcal{D}$ combined yield the proof of Statement (B)(i). 
 
By Corollary \ref{cor:band}, $x^+(\mathbb{R})\subset [m_0,m_1]$. 
Defining $\sigma=\max\{\tau_0,\tau_1\}$, recalling $x^+(t_2)=1$, 
Proposition \ref{prop:envelope}  implies that, if $T\ge t_2$, then each interval $[T,T+\sigma]$ 
contains at least one $t$ with $x^+(t)=1$. 

From  $x^+(\mathbb{R})\subset [m_0,m_1]$ and Proposition 
\ref{prop:bounds} it follows that 
$$
\{x^+_{t}:t\in\mathbb{R}\}\subset C_{m_0,m_1}^{2d}\subset 
C_{0,d/c}^{2d}. 
$$
The statement for 
$\mathcal{A}=\omega_\Gamma(x^+_{t_2+1})$ follows from the remarks after 
Proposition \ref{prop:bounds}. 
Clearly, $\hat{0}\notin \mathcal{A}$. 
The semiflow $\Gamma$ restricted to $[0,\infty)\times C_{0, (1+\xi_1)/2}$
is continuous. So, in case $\hat{\xi}_1\in \mathcal{A}$, 
$x^+(t)<1$ would follow on some interval $[T,T+2\sigma]$ with 
$T>t_2$, a contradiction to the fact that there is a $t $ with 
$x^+(t)=1$ in each interval of length $\sigma$. 

Statement B(iii) is clear from $x^+_{t_2+1}=p_0$ and uniqueness of solutions. 

The proof is complete.  
\end{proof}

\begin{proof}[Proof of Theorem~\ref{thm:Efn}]
Statement (A) follows from Proposition \ref{prop:W-Efn}(ii).

By Theorem \ref{thm:Eg}(B)(i), there is a $T>t_2+1$ such that $x^+_{T}\ll \hat{\xi}_1$. Then Corollary \ref{cor:yntozero} can be applied with
$Y^{+,n,j}(t)=y^{+,n}(t)$ and $\varepsilon_n=0$, $J_n=\{0\}$, to obtain $y^{+,n}(t)\to 0$ as $t\to\infty$. Thus,  Statement (B)(i) holds.

In order to prove Statement (B)(ii), choose $d_0,\delta,n_0$ 
as in Proposition \ref{prop:interval-exit}.
By Proposition \ref{prop:yn-x-on12}, with $Y^{+,n}(t)=y^{+,n}(t)$,  
$\varepsilon_n=0$, $J_n=\{0\}$, and by 
$x^+(t_2)=1$, for all large $n$, we have 
$|y^{+,n}(t_2)-1|<\delta$. 
Therefore, Proposition \ref{prop:W-Efn}, Proposition \ref{prop:interval-exit} with $\tilde{\sigma}=
\max\{\tau_0,\nu_1\}$, and   $\tilde{m}_0=m_0/2$, $\tilde{m}_1=2d/c$ 
imply 
$y^{+,n}(\mathbb{R})\subset [\tilde{m}_0,\tilde{m}_1]$, and the 
intervals  $[T,T+\tilde{\sigma}]$ contain a $t$ with 
$y^{+,n}(t)\in [1-\delta,1+\delta]$ provided $T\ge t_2$. 
 
By Proposition \ref{prop:bounds} 
the  set 
$\{y^{+,n}_{t}:t\in \mathbb{R}\}$ is in the compact subset
$C_{\tilde{m}_0,\tilde{m}_1}^{8d}$ of $C^+$. 
Then 
$\mathcal{A}_n=\omega_{\Phi^n}(y^{+,n}_{t_2+1})$ is 
a nonempty, compact, invariant subset of $C_{\tilde{m}_0,\tilde{m}_1}^{8d}$ of $C^+$ with $y_t^{+,n}\to\mathcal{A}_n$ as $t\to\infty.$

It remains to show that $\hat{0}$ and $\hat{\xi}_{1,n}$ are not in 
$\mathcal{A}_n$.   
Obviously, $\hat{0}\notin\mathcal{A}_n$.   
If $\hat{\xi}_{1,n}\in \mathcal{A}_n$, then, for large $n$, 
$\hat{\xi}_{1,n}<1-\delta $, and 
by continuity 
$y^{+,n}(t)<1-\delta $ on some interval $I\subset[t_2,\infty)$ with length 
greater than $\tilde{\sigma}$, a contradiction.  

For the proof of B(iii), suppose $d>d^*$ and $(C_g)(a)(b)$, $(C_{f_n})(a)(b)$. 

A consequence of Proposition \ref{prop:W-Eg} is the existence of 
an $\varepsilon>0$ and a $t_3\in (t_1,t_2)$ with 
$[t_3,t_3+1]\subset (t_1,t_2) $ such that  
$$
x^+(t)\ge 1+2\varepsilon \text{ for all }t\in [t_3,t_3+1]. 
$$
Proposition \ref{prop:main0-application} guarantees that, for all 
sufficiently large $n$, equation \eqref{eqn:Efn} admits a stable 
periodic orbit $\mathcal{O}^n$ such that 
its region of attraction contains the set 
$\{\psi\in C^+: \psi(s)\ge 1+\varepsilon,\ s\in[-1,0]\}$.

Proposition \ref{prop:W-Efn} gives the functions $y^{+,n}$ 
so that condition $(Y_n)$ is satisfied with $Y^{+,n,j}=y^{+,n}$ and 
$\varepsilon_n=0$, $J_n=\{0\}$. 
By Corollary \ref{cor:yn>1+e} we conclude $y^{+,n}(t)\ge 1+\varepsilon$ 
for all $t\in [t_3,t_3+1]$, that is $y^{+,n}_{t_3+1}$ is in the 
region of attraction of $\mathcal{O}^n$. 
Consequently, $\mathcal{A}_n=\mathcal{O}^n$, for all large $n$.
This completes the proof. 
\end{proof}

\section{Proof of Theorem \ref{thm:Hopf}}\label{sec6}

{\it Step 1: A local Hopf bifurcation.}	
Consider the parametrized  equation  
	\begin{equation}\label{eqn:Ehalpha}
		\tag{$E_{h,\alpha}$}
		u'(t) 
		 = 
		(1+\alpha)\Bigl[- a u(t) + b h\left(u(t-1)\right)\Bigr]
	\end{equation}
with $\alpha\in(-1,1)$, $a>0$, $b>0$  and a $C^2$-smooth real function $h$ 
such that, for some $j\in\N$,  
\begin{equation}\label{hopf-cond}
\begin{aligned}
&a=bh'(0) \cos \Theta_j \text{ where } \Theta_j \text{ is the unique solution }\\
& \text{of } \Theta=-a\tan \Theta \text{ in }(2j\pi -\frac{\pi}{2},2j\pi). 
\end{aligned}
\end{equation}
The right hand side of \eqref{eqn:Ehalpha} can be written as
$$
(1+\alpha)[- a u(t) + b h'(0)u(t-1)]+(1+\alpha)b[h(u(t-1)) - h'(0)u(t-1)].
$$
Then the characteristic function $H(\alpha,\lambda)$ of the linearization  of 
\eqref{eqn:Ehalpha} is 
$$
H(\alpha,\cdot): \ \C\ni\lambda\mapsto \lambda + (1+\alpha)[a-bh'(0)e^{-\lambda}]\in\C.
$$
By \eqref{hopf-cond}, $\lambda_j=i\Theta_j$ and $\overline{\lambda_j}$ 
are simple characteristic roots, and no other roots with zero real part. 
Let  
$	H_1,H_2: (-1,1)\times \mathbb{R}\times \mathbb{R}
	 \to 	\mathbb{R}$ 
	 be given by
	$
	H(\alpha, \mu+i \nu)=H_1 (\alpha, \mu, \nu) +iH_2 (\alpha, \mu, \nu)$. 
Then $H_1(0,0,\Theta_j)=0$ and $H_2(0,0,\Theta_j)=\Theta_j$. 
	$$
\textrm{det}
	\begin{bmatrix}
		\dfrac{\partial H_1}{\partial \mu} & \dfrac{\partial H_1}{\partial \nu}
		\\
		\dfrac{\partial H_2}{\partial \mu} & \dfrac{\partial H_2}{\partial \nu}
	\end{bmatrix}_{(0,0,\Theta_j)} = (1+a)^2 + \Theta_j^2\ne 0.
	$$
By the Implicit Function Theorem there is a unique $C^1$ function 
	$\alpha\mapsto (\mu(\alpha), \nu(\alpha))$  
	such that $H(\alpha,\mu(\alpha)+i\nu(\alpha))=0$ for small $|\alpha|$ 
	and $\mu(0)=0$, $\nu(0)=\Theta_j$. That is, the characteristic root 
	$\lambda (0)=\lambda_j=i\Theta_j$ has a unique local $C^1$ continuation 
	$\lambda (\alpha)=\mu(\alpha)+i\nu(\alpha) $. 
	Differentiating  the equations $H_k(\alpha,\mu(\alpha),\nu(\alpha))=0$, $k\in\{1,2\}$, with respect to $\alpha$ at $\alpha=0$, it is easy to obtain 
	\begin{equation}\label{transveral}
	\mu'(0)=\dfrac{\Theta_j^2}{(1+a)^2+\Theta_j^2}>0. 
	\end{equation}	
Therefore, 	
the Local Hopf Bifurcation Theorem  (see e.g., \cite{Hale} or \cite{LW})
can be applied to find a sequence $(\alpha_k)_{k=1}^\infty $, 
 with $\alpha_k \to 0$ as $k\to\infty$, such that for each $k\in\mathbb{N}$, 
		\eqref{eqn:Ehalpha} with $\alpha=\alpha_k$ admits a nontrivial periodic solution $		r^k: \mathbb{R} \to \mathbb{R}$ with minimal period 
		$\omega^k>0$, and 
		$
		\max_{t\in\mathbb{R}}
		\bigl| r^k(t)\bigr|
		 \to  0$ 
as  $k\to\infty$.

{\it Step 2:	Construction of the sequences $(a_n)$, $(b_n)$ and the periodic solutions $(q^n)$. }

Let $d>c>0$ and $j\in\mathbb{N}$ be given so that $(C_g)(c)$ holds. 
We look for $a_n$ and $b_n$ in equation \eqref{eqn:Efn} in the form 
$a_n=\gamma_n c$, $b_n=\gamma_n d$ for some $\gamma_n>0$ with $\gamma_n\to 1$
as $n\to\infty$. 

Consider the equations  
	\begin{equation}\label{eqn:Efnbeta}
		\tag{$E_{f_n,\beta}$}
		y'(t)
		 = 
		\beta \Bigl[- c y(t)
		 + 
		d f_n\left(y(t-1)\right)\Bigr]
	\end{equation}
	with a parameter $\beta\in(0,2)$. 
The zeros of $\xi\mapsto \beta[- c \xi + d f_n(\xi)]$ are independent of 
\(\beta\in(0,2)\). 
Proposition \ref{prop:zeros-fn} with $a_n=c$, $b_n=d$ applies to 
get that, for a given $\kappa\in (\xi_1,1)$,  for all sufficiently large $n$,  the only zeros 
in $[0,\kappa] $ are $0$ and $\xi_{1,n}$, and  \(\xi_{1,n}\to \xi_1\) as \(n\to\infty\). 

By $(C_g)(a)$ we have $g'(\xi_1)>\tfrac{c}{d}$, and by  $(C_{f_n})(b)$, 
	$$
	0 < \frac{c}{d f'_n(\xi_{1,n})}
	 < 
	1
	$$
	for all sufficiently large \(n\). 
Hence, 
there is a  unique \(\Theta_{j,n}\in\left(2j\pi-\tfrac{\pi}{2}, 2j\pi\right)\) with
\begin{equation}\label{thete-jn}
\cos \Theta_{j,n}
	 = 
	\frac{c}{ d f'_n(\xi_{1,n}) }, 
\end{equation}
where $j\in\N $ is given in $(C_g)(c)$.  
Define
	\begin{equation}\label{beta-jn}
	\beta_{j,n}
	 = 
	- \frac{\Theta_{j,n}}{ c \tan \Theta_{j,n} }.
	\end{equation}
Notice that, from $\xi_{1,n}\to\xi_1$ and  	$(C_{f_n})(b)$, 
it follows that \(\Theta_{j,n}\to\Theta_j\), and
	$
	\beta_{j,n}
	 \to 
	- \frac{\Theta_j}{ c \tan(\Theta_j) }
	 = 1$ 
as $n\to\infty$.
	
For a fixed $n\in\mathbb{N}$, consider equation \eqref{eqn:Ehalpha} with 
$h=h_n$ defined by \eqref{def-hn} in Section \ref{sec4}, and $a=\beta_{j,n} c$, $b=\beta_{j,n}d$.   
Equations \eqref{thete-jn} and \eqref{beta-jn} show that condition \eqref{hopf-cond} is satisfied. 
One easily checks that the transversality condition \eqref{transveral} 
holds as well in this case. 

Consequently, Step 1 implies that, for each  $n\in\mathbb{N}$, 
there exists a sequence $(\alpha_{n,k})_{k=1}^\infty$ 
with $\alpha_{n,k}\to 0$ as $k\to\infty$ such that, for each 
$k\in\mathbb{N}$, the equation 
$$
u'(t)=(1+\alpha_{n,k})\beta_{j,n}[-cu(t)+dh_n(u(t-1))] 
$$
has a nontrivial periodic solution $r^{n,k}:\R\to\R$ with minimal period  $\omega^{n,k}>0$ and with 
$ \max_{t\in\R} |r^{n,k}(t)|\to 0$ as $k\to\infty$.  

For each $n\in\mathbb{N}$, choose $k(n)$ so large that 
$$
|\alpha_{n,k(n)}| <\dfrac{1}{n} \text{ and }
\max_{t\in\R} |r^{n,k(n)}(t)|<\dfrac{1}{n}.
$$
Setting 
\begin{equation}\label{def-an-bn}
	a_n	 = (1+\alpha_n)\beta_{j,n}  c,
	\quad
	b_n  = (1+\alpha_n)\beta_{j,n} d,
	\quad
	r^n=r^{n,k(n)},\quad
	\omega^n=\omega^{n,k(n)},
\end{equation}
it is clear that, for all sufficiently large $n$, \(q^n:\mathbb{R}\to\mathbb{R}\), given by 
$q^n(t)=r^{n}(t)+\xi_{1,n}$, is a nonconstant periodic 
solution of equation \eqref{eqn:Efn} with minimal period $\omega^n>0$
satisfying 
\begin{equation}\label{qn-xin}
	\max_{t\in\mathbb{R}}
	\bigl| q^n(t) - \xi_{1,n}\bigr|
	 < \frac{1}{ n }.
\end{equation}	
In addition, $a_n\to c$, $b_n\to d$ as $n\to\infty$. 
Define 
$\mathcal{Q}^n=\{q_t^n: t\in [0,\omega^n]\}$.

{\it Step 3: The unstable set $W^u(\mathcal{Q}^n)$}. 

We apply the results of Section \ref{sec3} in the case when 
$n$ is sufficiently large,  $h=h_n$ defined by \eqref{def-hn},  
$a=a_n$, $b=b_n$,  $p=r^n$, 
$\omega=\omega^n$ are given in \eqref{def-an-bn}. 
Let $\Phi_n$ denote the solution semiflow for this case. 


As in Section \ref{sec3}, for each large $n$, define 1-dimensional leading unstable manifolds 
$$
W^u_{\text{loc},\delta_n}(r^n_0,\Phi_n(\omega^n,\cdot))
$$ 
of the fixed point $r^n_0$ of the map $\Phi_n(\omega^n,\cdot)$ with 
some $\delta_n>0$. 
By \eqref{Wloc-decomp}, $W^u_{\text{loc},\delta_n}(r^n_0,\Phi_n(\omega^n,\cdot))$ is decomposed into three subsets: 
\begin{equation}\label{Wlocn-decomp}
W^u_{\text{loc},\delta_n}(r^n_0,\Phi_n(\omega^n,\cdot))=
W^{u,-}_{\text{loc},\delta_n} (r^n_0,\Phi_n(\omega^n,\cdot))
\cup \mathcal{Q}^n 
\cup W^{u,+}_{\text{loc},\delta_n}(r^n_0,\Phi_n(\omega^n,\cdot)),
\end{equation}
where for $\psi\in 
W^{u,-}_{\text{loc},\delta_n} (r^n_0,\Phi_n(\omega^n,\cdot))$ 
and 
$\phi\in W^{u,+}_{\text{loc},\delta_n}(r^n_0,\Phi_n(\omega^n,\cdot)) $, 
the relation $\psi\ll r^n_0 \ll \phi $ is satisfied. 

By the definition of $h_n$, see \eqref{def-hn}, the 1-dimensional leading local unstable manifold of  the map $\Phi^n(\omega^n,\cdot)$ at its fixed point $q^n_0$ can be defined by 
$$
W^u_{\text{loc},\delta_n}(q^n_0,\Phi_n(\omega^n,\cdot))
=
W^u_{\text{loc},\delta_n}(r^n_0,\Phi_n(\omega^n,\cdot))+\hat{\xi}_{1,n}, 
$$ 
and analogously for the sets $W^{u,-}_{\text{loc},\delta_n} (q^n_0,\Phi^n(\omega^n,\cdot))$ 
and 
$W^{u,+}_{\text{loc},\delta_n}(q^n_0,\Phi^n(\omega^n,\cdot)) $. 

The 2-dimensional leading unstable set  of the periodic orbit
$\mathcal{Q}^n$  
is given by
$$
W^u(\mathcal{Q}^n)
 = 
\Phi^n\left([0,\infty)\times W^u_{\text{loc},\delta_n} (q_0^n,\Phi^n(\omega^n,\cdot)) \right).
$$
Moreover, 
$$
W^u(\mathcal{Q}^n)
 = 
W^{u,-}(\mathcal{Q}^n)\cup  \mathcal{Q}^n \cup W^{u,+}(\mathcal{Q}^n)
$$
where 
$$
W^{u,-}(\mathcal{Q}^n)
 = 
 \Phi^n\left([0,\infty)\times W^{u,-}_{\text{loc},\delta_n} (q^n_0,\Phi^n(\omega^n,\cdot)) \right)
 $$
 and
 $$
W^{u,+}(\mathcal{Q}^n)
 = 
 \Phi^n\left([0,\infty)\times W^{u,+}_{\text{loc},\delta_n} (q^n_0,\Phi^n(\omega^n,\cdot)) \right).
$$

Fix $\kappa^+\in (0,1-\xi_1)$ and $\kappa^-\in (-\xi_1,0)$ 
as in Section \ref{sec4}. 
For each 
 $\psi\in W^{u,-}_{\text{loc},\delta_n}(r^n_0,\Phi_n(\omega^n,\cdot)) $ 
 and 
 $\phi\in W^{u,+}_{\text{loc},\delta_n} (r^n_0,\Phi_n(\omega^n,\cdot))$, 
 Proposition \ref{W(O)} gives the solutions 
 $u^{-,n,\psi}:\R\to\R$, $u^{+,n,\phi}:\R\to\R$ of equation \eqref{eqn:Eh},  
 with $h=h_n$, $a=a_n$, $b=b_n$ given as above, 
 such that statements (i) and (ii) of Proposition \ref{W(O)} hold. 
 
 For large $n$, $\psi\in W^{u,-}_{\text{loc},\delta_n}(q^n_0,\Phi^n(\omega^n,\cdot)) $ 
 and 
 $\phi\in W^{u,+}_{\text{loc},\delta_n} (q^n_0,\Phi^n(\omega^n,\cdot))$,
 define 
the solutions $y^{-,n,\psi}:\mathbb{R}\to \mathbb{R}$, 
 $y^{+,n,\phi}:\mathbb{R}\to \mathbb{R}$ of equation \eqref{eqn:Efn}
so that 
\begin{equation}\label{def-yn}
y^{-,n,\psi}(t) = 
u^{-,n,\psi-\hat{\xi}_{1,n}}(t)+\xi_{1,n} \quad (t\in\mathbb{R})
\end{equation}
and
\begin{equation}\label{def+yn}
y^{+,n,\phi}(t) = 
\begin{cases}
u^{+,n,\phi-\hat{\xi}_{1,n}}(t)+\xi_{1,n} \quad (t\le 0)\\
\Phi^n\left(t, u^{+,n,\phi-\hat{\xi}_{1,n}}_0+\hat{\xi}_{1,n}\right)(0)
\quad (t>0).
\end{cases}
\end{equation}
Then, by the definition of the unstable set $W^u(\mathcal{Q}^n)$,
\begin{equation}\label{Wu-}
W^{u,-}(\mathcal{Q}^n)=
\left\lbrace y_t^{-,n,\psi}: t\in\mathbb{R},\ \psi\in W^{u,-}_{\text{loc},\delta_n} (q^n_0,\Phi^n(\omega^n,\cdot))
\right\rbrace 
\end{equation}
and
 \begin{equation}\label{Wu+}
W^{u,+}(\mathcal{Q}^n)=
\left\lbrace y_t^{+,n,\phi}: t\in\mathbb{R},\ \phi\in W^{u,+}_{\text{loc},\delta_n} (q^n_0,\Phi^n(\omega^n,\cdot))
\right\rbrace .
 \end{equation} 
From the construction above it is clear that the decomposition
$$
W^u(\mathcal{Q}^n)
 = 
W^{u,-}(\mathcal{Q}^n)\cup  \mathcal{Q}^n \cup W^{u,+}(\mathcal{Q}^n)
$$
holds.  
 
{\it Step 4: The proof of Statement (i).} 
Statement (i) of 
 Proposition \ref{W(O)} combined with 
 \eqref{def-yn} and \eqref{Wu-} yield the proof.

{\it Step 5: The proof of Statement (ii).}
From Proposition \ref{W(O)}(ii), 
 \eqref{def+yn} and \eqref{Wu+} it follows that $y_t^{+,n,\phi}\to 
\mathcal{Q}^n$ as $t\to-\infty$ for all  
$\phi\in W^{u,+}_{\text{loc},\delta_n} (q^n_0,\Phi^n(\omega^n,\cdot))$. 
It remains to show that 
$y_t^{+,n,\phi}\to \hat{0}$
 as $t\to+\infty$ for all  
$\phi\in W^{u,+}_{\text{loc},\delta_n} (q^n_0,\Phi^n(\omega^n,\cdot))$ 
which can be shown by applying Corollary \ref{cor:yntozero}. 
Indeed, 
Theorem \ref{thm:Eg} implies the existence of a $T>t_2+1$ such that 
$x^+_T\ll \hat{\xi}_1$. 
Condition $(Y_n)$ is satisfied with $\varepsilon_n=1/n$, 
$J_n=W^{u,+}_{\text{loc},\delta_n} (q^n_0,\Phi^n(\omega^n,\cdot)),
$, $Y^{+,n,\phi}=y^{+,n,\phi}$. 
The inequality $\xi_{1,n}-1/n\le y^{+,n,\phi}(t)\le y^{+,n,\phi}(0)=\kappa^++
\xi_{1,n}$, $t<0$, is guaranteed by \eqref{qn-xin}, \eqref{def+yn} and 
Theorem \ref{W(O)}(ii). 

{\it Step 6: The proof of Statement (iii).} 
The proof is similar to that of Theorem \ref{thm:Efn}(B)(ii). 

By Proposition \ref{W(O)}(ii) and \eqref{def+yn}, \eqref{Wu+}, 
the relation $y^{+,n,\phi}_t\gg q_t^n$ holds for all $t\le 0$, and for all 
$\phi\in W^{u,+}_{\text{loc},\delta_n} (q^n_0,\Phi^n(\omega^n,\cdot))$.

Choose $d_0,\delta,n_0$ 
as in Proposition \ref{prop:interval-exit}. 
Condition $(Y_n)$ is used as in Step 5, that is, $\varepsilon_n=1/n$, 
$J_n=W^{u,+}_{\text{loc},\delta_n} (q^n_0,\Phi^n(\omega^n,\cdot)),
$, $Y^{+,n,\phi}=y^{+,n,\phi}$. 
By Proposition \ref{prop:yn-x-on12},  and by 
$x^+(t_2)=1$, for all large $n$, we have 
$|y^{+,n,\phi}(t)-x^+(t)|<\delta$ for all $t\in[0,t_2]$ uniformly in $\phi\in W^{u,+}_{\text{loc},\delta_n} (q^n_0,\Phi^n(\omega^n,\cdot))$. 
In particular, $|y^{+,n,\phi}(t_2)-1|<\delta$. 

These facts and Proposition \ref{prop:interval-exit} with $\tilde{\sigma}=
\max\{\tau_0,\nu_1\}$, and   $\tilde{m}_0=m_0/2$, $\tilde{m}_1=2d/c$ 
imply that, for all $\phi\in W^{u,+}_{\text{loc},\delta_n} (q^n_0,\Phi^n(\omega^n,\cdot))$, the solutions $y^{+,n,\phi}$ satisfy
$y^{+,n,\phi}(\mathbb{R})\subset [\tilde{m}_0,\tilde{m}_1]$, and the 
intervals  $[T,T+\tilde{\sigma}]$ contain a $t$ with 
$y^{+,n,\phi}(t)\in [1-\delta,1+\delta]$ provided $T\ge t_2$. 
 
By using  Proposition  \ref{prop:bounds} as well, 
the  set 
$$
W^{u,+}(\mathcal{Q}^n)= \left\lbrace  y^{+,n,\phi}_{t}:t\in \mathbb{R}, \ 
\phi\in W^{u,+}_{\text{loc},\delta_n} (q^n_0,\Phi^n(\omega^n,\cdot))
\right\rbrace 
$$
is in the compact subset
$C_{\tilde{m}_0,\tilde{m}_1}^{8d}$ of $C^+$. 
Then 
$\omega_{\Phi^n}(y^{+,n,\phi}_{t_2+1})$ is 
a nonempty, compact, invariant subset of $C_{\tilde{m}_0,\tilde{m}_1}^{8d}$ of $C^+$ with $y_t^{+,n}\to \omega_{\Phi^n}(y^{+,n,\phi}_{t_2+1})$ as $t\to\infty.$

Clearly, the point $\hat{0}$ is not in 
$\omega_{\Phi^n}(y^{+,n,\phi}_{t_2+1})$.   
If $\hat{\xi}_{1,n}\in \omega_{\Phi^n}(y^{+,n,\phi}_{t_2+1})$
 or $\mathcal{Q}^n \cap \omega_{\Phi^n}(y^{+,n,\phi}_{t_2+1})\ne \emptyset$
 then  
by continuity 
$y^{+,n,\phi}(t)<1-\delta $ on some interval $I\subset[t_2,\infty)$ with length  
greater than $\tilde{\sigma}$, a contradiction. 

For any $\psi\in W^{u,+}(\mathcal{Q}^n)$ there are 
$\phi\in W^{u,+}_{\text{loc},\delta_n} (q^n_0,\Phi^n(\omega^n,\cdot))$ 
and $t\in\mathbb{R}$ so that $\psi=y^{+,n,\phi}_{t}$. 
Hence all stated properties of $y^\psi$ follow from those of 
$y^{+,n,\phi}$, and from 
 $\omega_{\Phi^n}(\psi)= \omega_{\Phi^n}(y^{+,n,\phi}_{t_2+1})$.

{\it Step 7: The proof of Statement (iv).} 
By Proposition \ref{prop:W-Eg} there are  
an $\varepsilon>0$ and a $t_3\in (t_1,t_2)$ with 
$[t_3,t_3+1]\subset (t_1,t_2) $ such that  
$
x^+(t)\ge 1+2\varepsilon \text{ for all }t\in [t_3,t_3+1]. $ 
Proposition \ref{prop:main0-application} guarantees that, for all 
sufficiently large $n$, equation \eqref{eqn:Efn} admits a stable 
periodic orbit $\mathcal{O}^n$ such that 
its region of attraction contains the set 
$\{\psi\in C^+: \psi(s)\ge 1+\varepsilon,\ s\in[-1,0]\}$.

Condition $(Y_n)$ is satisfied in the same way as in Step 5. 
Then Corollary \ref{cor:yn>1+e} can be applied to conclude 
that, for any fixed sufficiently large $n$, 
$$
y^{+,n,\phi}_{t}(t)\ge 1+\varepsilon \quad (t\in[t_3,t_3+1])
$$
for all 
$\phi\in W^{u,+}_{\text{loc},\delta_n} (q^n_0,\Phi^n(\omega^n,\cdot))$. 
Therefore, for any $\psi\in W^{u,+}(\mathcal{Q}^n)$, there is a $t>0$ 
such that $\Phi^n(t,\psi)$ is in the region of attraction of 
$\mathcal{O}^n$, and $\omega_{\Phi^n}(\psi)=\mathcal{O}^n$. 
This completes the proof. 


\section{Examples}\label{sec7}

Here we consider examples for our prototype equation \eqref{eqn:proto}. 
More precisely, for four  
pairs of parameters $a,b$ and large $n$, solutions $y^{\pm,n}$ are portrayed in the $(t,y)$ plane  complementing the visualization  on Figure \ref{fig:1}. The behaviors of the solutions, on the following figures, are rigorously proved, just like the other results in this article. 
The paper \cite{BKSZ} 
  guarantees the existence of the stable periodic orbits $\mathcal{O}^n$,  see also Propositions \ref{thm:main0} and  \ref{prop:main0-application}. 
The novelty of this paper is the existence of the the connecting orbits and the periodic orbits $\mathcal{Q}^n$.  

On Figures \ref{fig:x1} and \ref{fig:x2},  the blue curves are the solutions  $y^{+,n}$ connecting the stationary points $\xi_{1,n}$ to  periodic 
solutions. The green curves are the solutions $y^{-,n}$  connecting $\xi_{1,n}$ to $0$. The black one is the stationary solution $\xi_{1,n}$.  
 The projections  $\mathbb{R}\ni t\mapsto (y^{+,n}(t),y^{+,n}(t-1))\in \mathbb{R}^2$ are depicted on the right of the figures. 

\begin{figure}[h]
	\centering
	\includegraphics[width=\linewidth]{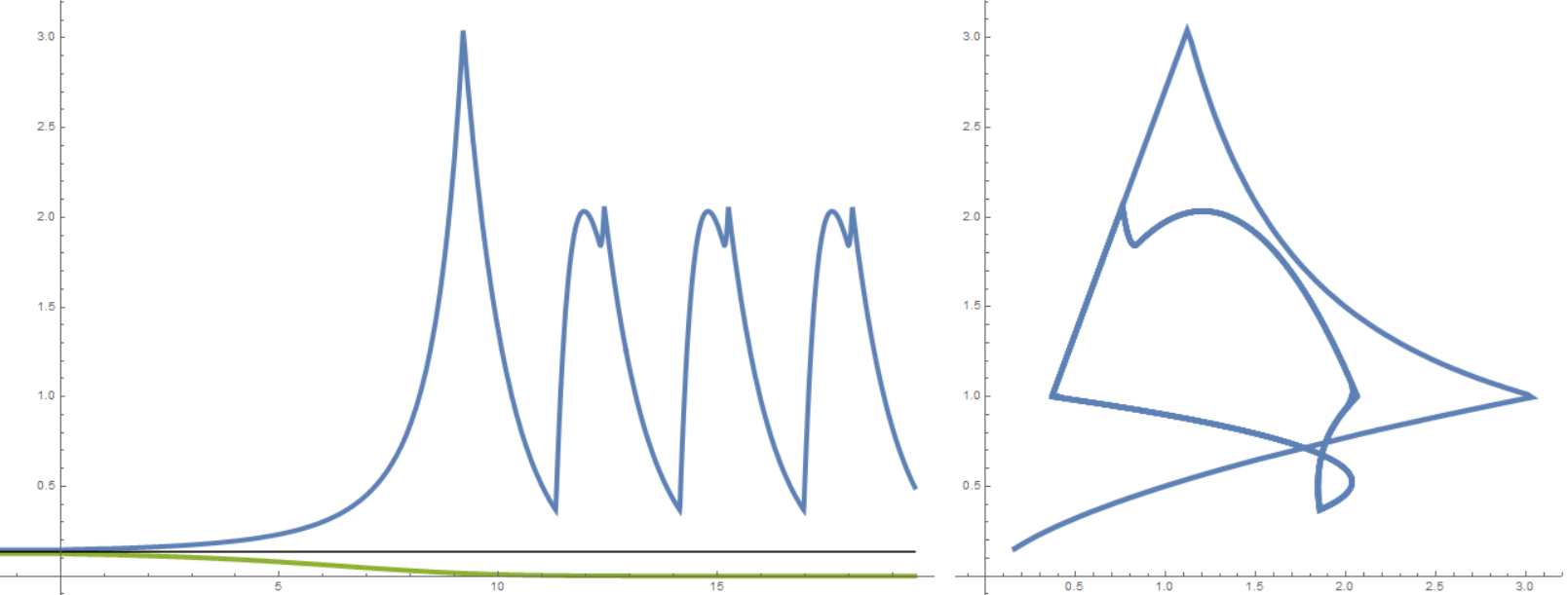}
	\caption{The case $a=1$, $b=7.38$, $n$ large. 	
	}
	\label{fig:x1}
\end{figure}

\begin{figure}[h]
	\centering
	\includegraphics[width=\linewidth]{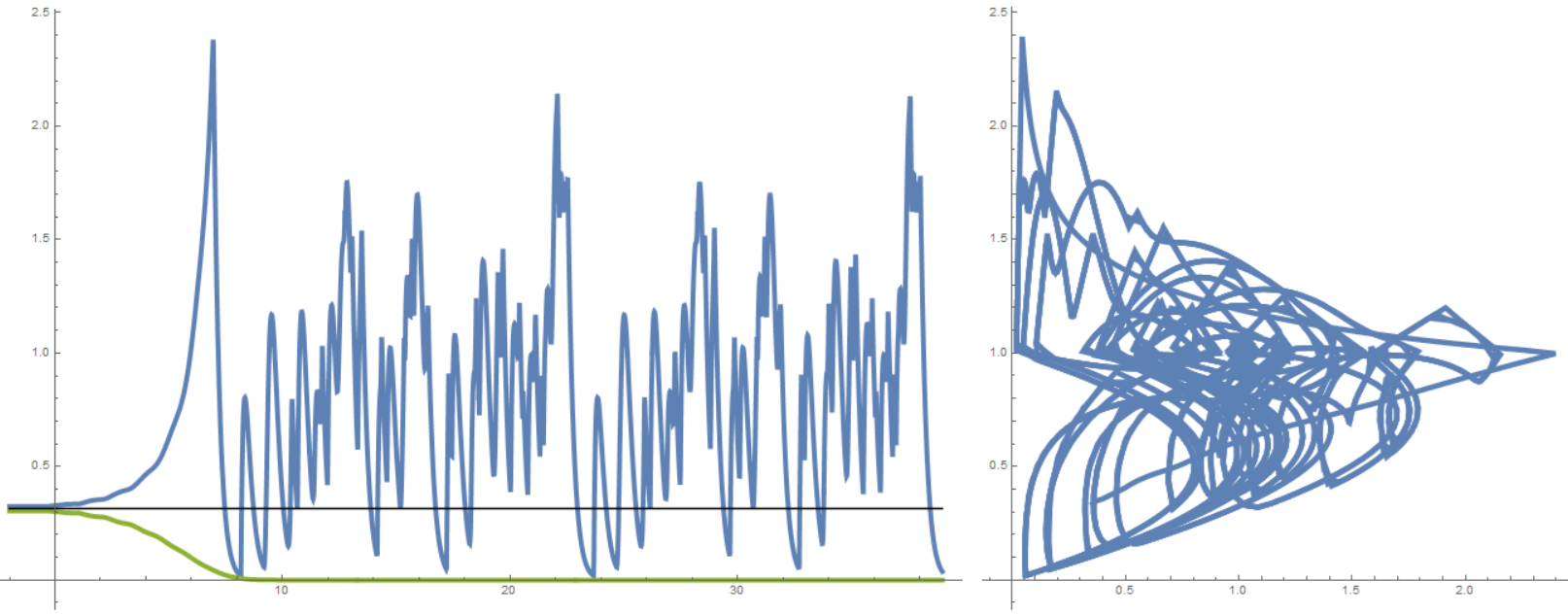}
	\caption{
The case $a=4$, $b=12.71$, $n$ large. 	The structure of the periodic solution looks complicated, although the corresponding periodic orbit $\mathcal{O}^n$ is stable. 	
	}
	\label{fig:x2}
\end{figure}

On Figures \ref{fig:x4} and \ref{fig:x3}, $a$ is close to $c=\frac{5\pi}{3\sqrt{3}}$, for which condition $(C_g)(c)$ holds independently of $d$.  
Theorem \ref{thm:Hopf} guarantees the 
existence of a periodic orbit $\mathcal{Q}^n$. 
The blue curves are solutions $y^{+,n}$ on the 2-dimensional unstable set 
		of the periodic orbit $\mathcal{Q}^n$, connecting $\mathcal{Q}^n$ 
		to $\mathcal{O}^n$. 
		The green curves are  solutions $y^{-,n}$ on the 2-dimensional unstable set 
		of the periodic orbit $\mathcal{Q}^n$, connecting $\mathcal{Q}^n$ 
		to $\hat{0}$. 
		The orange curves are the periodic solutions $q^n$.  

\begin{figure}[h]
	\centering
	\includegraphics[width=\linewidth]{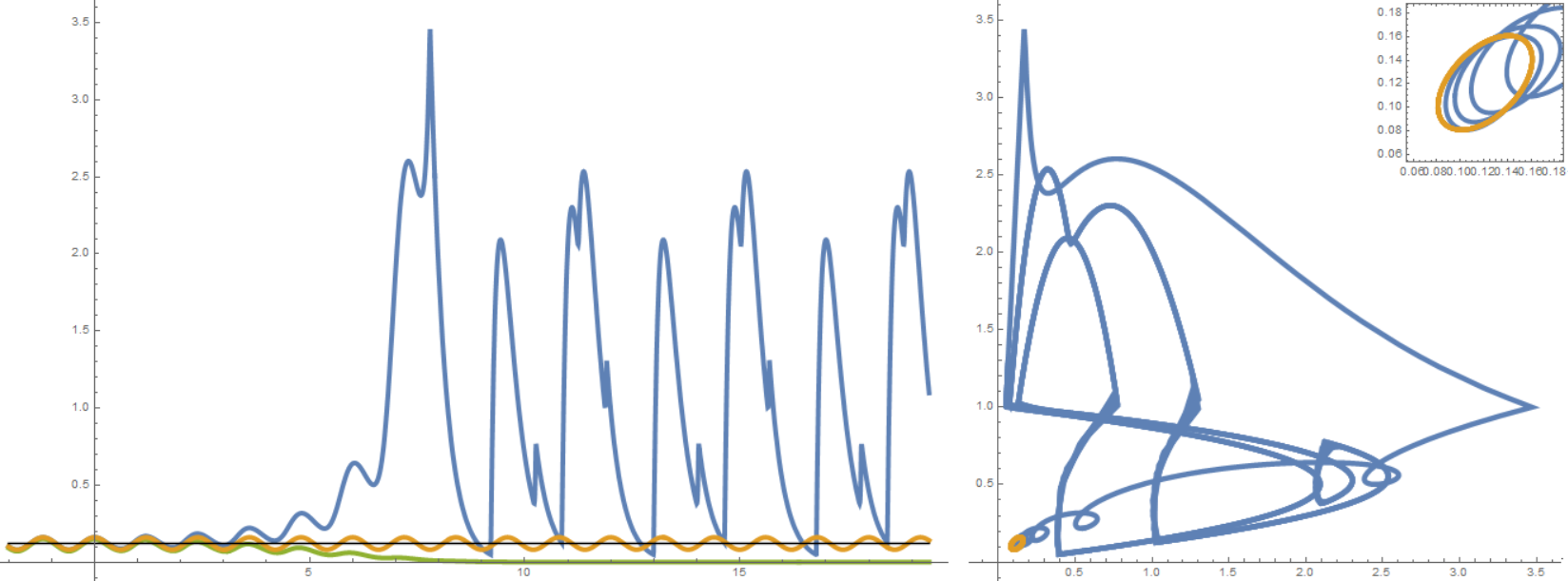}
	\caption{
		The case 
		$a\approx \frac{5\pi}{3\sqrt{3}}$, $b\approx 25$, $n$ large. 	
	}
	\label{fig:x4}
\end{figure}

\begin{figure}[h]
	\centering
	\includegraphics[width=\linewidth]{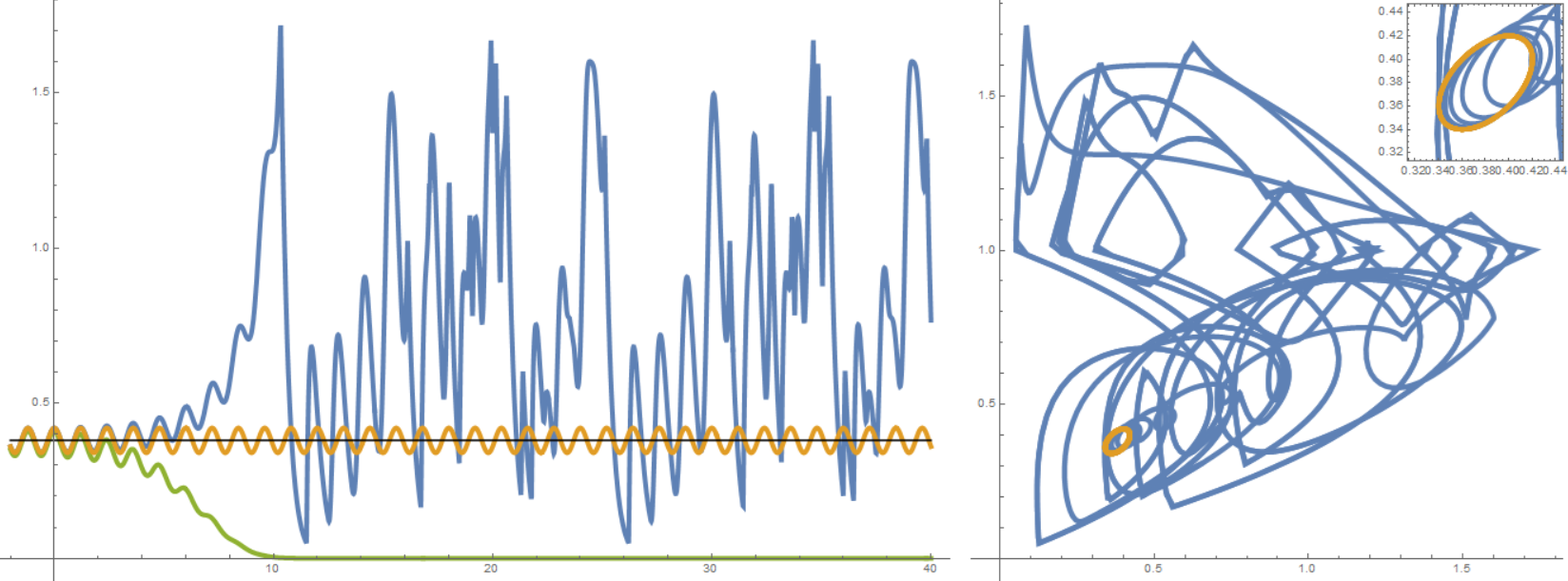}
	\caption{
The case 
$a\approx \frac{5\pi}{3\sqrt{3}}$, $b\approx 7.95$, $n$ large. 	
	}
	\label{fig:x3}
\end{figure}

\section*{Acknowledgement}
G. Benedek has been supported by the M\'oricz Doctoral Scholarship, 2022-23.
The reserach of T. Krisztin
 has been supported by the National Research, Development and Innovation Fund, Hungary [project no.  TKP2021-NVA-09, and grants K-129322, KKP-129877] and by the National Laboratory for Health Security [RRF-2.3.1-21-2022-00006].


\end{document}